\documentclass[10pt,a4paper,reqno]{amsart}
\title[Bar category of modules and homotopy adjunction for tensor functors]{Bar category of modules and \\ homotopy adjunction for tensor functors}
\author{Rina Anno}
\email{ranno@ksu.edu}
\address{Department of Mathematics \\
Kansas State University \\
138 Cardwell Hall \\
Manhattan, KS 66506\\
USA}
\author{Timothy Logvinenko} 
\email{LogvinenkoT@cardiff.ac.uk} 
\address{School of Mathematics\\ 
Cardiff University\\
Senghennydd Road\\
Cardiff, CF24 4AG\\
UK}
\usepackage{amsmath,amsfonts,amssymb,amsthm,epsfig,amscd,latexsym,comment}
\usepackage{tikz-cd}
\usepackage{caption}
\usepackage{graphicx}
\usepackage{array}
\usepackage{subfigure}
\usepackage{leftidx}
\usepackage{xparse}
\usepackage[all]{xy}

\usepackage[colorlinks=true, pdfmenubar=false, pdfstartview=FitH, linkcolor=blue, citecolor=blue, urlcolor=blue]{hyperref}

\let\amsamp=&

\begingroup
\catcode`\&=13
\gdef\smallampmatrix{%
  \begingroup
  \let&=\amsamp
  \begin{smallmatrix}%
}
\gdef\endsmallampmatrix{\end{smallmatrix}\endgroup}
\endgroup

\DeclareMathOperator{\img}{Im}

\DeclareMathOperator{\homm}{Hom}

\DeclareMathOperator{\eend}{End}

\DeclareMathOperator{\picr}{Pic}
\DeclareMathOperator{\tot}{Tot}

\DeclareMathOperator{\cl}{Cl}

\DeclareMathOperator{\ev}{ev}
\DeclareMathOperator{\trace}{tr}
\DeclareMathOperator{\composition}{cmps}
\DeclareMathOperator{\action}{act}
\DeclareMathOperator{\mlt}{mlt}

\DeclareMathOperator{\modd}{\bf Mod}
\DeclareMathOperator{\lder}{\bf L}
\DeclareMathOperator{\rder}{\bf R}
\DeclareMathOperator{\ldertimes}{\overset{\lder}{\otimes}}

\DeclareMathOperator{\id}{Id}

\DeclareMathOperator{\cone}{Cone}
\DeclareMathOperator{\conv}{Conv}
\DeclareMathOperator{\opp}{{opp}}
\DeclareMathOperator{\fg}{{\it fg}}
\DeclareMathOperator{\qrep}{\it \mathcal{Q}r}
\DeclareMathOperator{\hproj}{\mathcal{P}}

\DeclareMathOperator{\acyc}{\it \mathcal{A}c}

\DeclareMathOperator{\semi-free}{\mathcal{S}\mathcal{F}}
\DeclareMathOperator{\sffg}{\mathcal{S}\mathcal{F}_{\fg}}
\DeclareMathOperator{\perf}{{\it \mathcal{P}erf}}
\DeclareMathOperator{\hmtpy}{{Ho}}

\DeclareMathOperator{\tria}{{Tria}}
\DeclareMathOperator{\twcx}{{Tw}}
\DeclareMathOperator{\pretriag}{{Pre\text{-}Tr}}
\DeclareMathOperator{\DGFun}{{DGFun}}

\DeclareMathOperator{\TPair}{{\bf TPair}}
\DeclareMathOperator{\alg}{{\bf Alg}}

\DeclareMathOperator{\strict}{{strict}}

\DeclareMathOperator{\DGMod}{\it{\mathcal{D}\mathcal{G}\mathcal{M}od}}
\begin{document}

\def\bv{\mathbf{v}}
\def\kgc_{K^*_G(\mathbb{C}^n)}
\def\kgchi_{K^*_\chi(\mathbb{C}^n)}
\def\kgcf_{K_G(\mathbb{C}^n)}
\def\kgchif_{K_\chi(\mathbb{C}^n)}
\def\gpic_{G\text{-}\picr}
\def\gcl_{G\text{-}\cl}
\def\trch_{{\chi_{0}}}
\def\regring{{R}}
\def\regrep{{V_{\text{reg}}}}
\def\givrep{{V_{\text{giv}}}}
\def\lbar{{(\mathbb{Z}^n)^\vee}}
\def\genpx_{{p_X}}
\def\genpy_{{p_Y}}
\def\genpcn_{p_{\mathbb{C}^n}}
\def\gnat{gnat}
\def\twalg{{\regring \rtimes G}}
\def\L{{\mathcal{L}}}
\def\O{{\mathcal{O}}}
\def\gcd{\mbox{gcd}}
\def\lcm{\mbox{lcm}}
\def\tf{{\tilde{f}}}
\def\tD{{\tilde{D}}}
\def\A{{\mathcal{A}}}
\def\B{{\mathcal{B}}}
\def\C{{\mathcal{C}}}
\def\D{{\mathcal{D}}}
\def\F{{\mathcal{F}}}
\def\H{{\mathcal{H}}}
\def\L{{\mathcal{L}}}
\def\R{{\mathcal{R}}}
\def\T{{\mathcal{T}}}
\def\U{{\mathcal{U}}}
\def\barA{{\bar{\mathcal{A}}}}
\def\tildeA{{\tilde{\mathcal{A}}}}
\def\barAi{{\bar{\mathcal{A}}_1}}
\def\barAj{{\bar{\mathcal{A}}_2}}
\def\barB{{\bar{\mathcal{B}}}}
\def\barC{{\bar{\mathcal{C}}}}
\def\barD{{\bar{\mathcal{D}}}}
\def\barT{{\bar{T}}}
\def\M{{\mathcal{M}}}
\def\Aopp{{\A^{\opp}}}
\def\Bopp{{\B^{\opp}}}
\def\Copp{{\C^{\opp}}}
\def\aA{\leftidx{_{a}}{\A}}
\def\bA{\leftidx{_{b}}{\A}}
\def\Aa{{\A_a}}
\def\Ea{E_a}
\def\aE{\leftidx{_{a}}{E}{}}
\def\Eb{E_b}
\def\bE{\leftidx{_{b}}{E}{}}
\def\Fa{F_a}
\def\aF{\leftidx{_{a}}{F}{}}
\def\Fb{F_b}
\def\bF{\leftidx{_{b}}{F}{}}
\def\aM{\leftidx{_{a}}{M}{}}
\def\aMb{\leftidx{_{a}}{M}{_{b}}}
\def\Ma{{M_a}}
\def\modk{{\modd\text{-}k}}
\def\kmodk{{k\text{-}\modd\text{-}k}}
\def\modA{{\modd\text{-}\A}}
\def\modbar{{\overline{\modd}}}
\def\modbarA{{\overline{\modd}\text{-}\A}}
\def\modbarAopp{{\overline{\modd}\text{-}\Aopp}}
\def\modB{{\modd\text{-}\B}}
\def\modbarB{{\overline{\modd}\text{-}\B}}
\def\modbarBopp{{\overline{\modd}\text{-}\Bopp}}
\def\modAB{{\modd\text{-}\A}\text{-}\B}
\def\modbarAB{{\overline{\modd}\text{-}\A\text{-}\B}}
\def\modBA{{\modd\text{-}\B}\text{-}\A}
\def\modbarBA{{\overline{\modd}\text{-}\B\text{-}\A}}
\def\modAA{{\modd\text{-}\A}\text{-}\A}
\def\modbarAA{{\overline{\modd}\text{-}\A\text{-}\A}}
\def\modBB{{\modd\text{-}\B}\text{-}\B}
\def\modbarBB{{\overline{\modd}\text{-}\B\text{-}\B}}
\def\modbarAC{{\overline{\modd}\text{-}\A\text{-}\C}}
\def\modCA{{\modd\text{-}\C}\text{-}\A}
\def\modbarCA{{\overline{\modd}\text{-}\C\text{-}\A}}
\def\modbarBC{{\overline{\modd}\text{-}\B\text{-}\C}}
\def\modCB{{\modd\text{-}\C}\text{-}\B}
\def\modbarCB{{\overline{\modd}\text{-}\C\text{-}\B}}
\def\modCD{{\modd\text{-}\C}\text{-}\D}
\def\modbarCD{{\overline{\modd}\text{-}\C\text{-}\D}}
\def\modDB{{\modd\text{-}\D}\text{-}\B}
\def\modbarDB{{\overline{\modd}\text{-}\D\text{-}\B}}
\def\modDC{{\modd\text{-}\D}\text{-}\C}
\def\modbarDC{{\overline{\modd}\text{-}\D\text{-}\C}}
\def\AmodbarA{\A\text{-}{\overline{\modd}\text{-}\A}}
\def\AmodbarB{\A\text{-}{\overline{\modd}\text{-}\B}}
\def\AmodbarC{\A\text{-}{\overline{\modd}\text{-}\C}}
\def\AmodbarD{\A\text{-}{\overline{\modd}\text{-}\D}}
\def\BmodbarA{\B\text{-}{\overline{\modd}\text{-}\A}}
\def\BmodbarB{\B\text{-}{\overline{\modd}\text{-}\B}}
\def\BmodbarC{\B\text{-}{\overline{\modd}\text{-}\C}}
\def\BmodbarD{\B\text{-}{\overline{\modd}\text{-}\D}}
\def\CmodbarA{\C\text{-}{\overline{\modd}\text{-}\A}}
\def\CmodbarB{\C\text{-}{\overline{\modd}\text{-}\B}}
\def\CmodbarC{\C\text{-}{\overline{\modd}\text{-}\C}}
\def\CmodbarD{\C\text{-}{\overline{\modd}\text{-}\D}}
\def\DmodbarA{\D\text{-}{\overline{\modd}\text{-}\A}}
\def\DmodbarB{\D\text{-}{\overline{\modd}\text{-}\B}}
\def\DmodbarC{\D\text{-}{\overline{\modd}\text{-}\C}}
\def\DmodbarD{\D\text{-}{\overline{\modd}\text{-}\D}}
\def\sfA{{\semi-free(\A)}}
\def\sfB{{\semi-free(\B)}}
\def\sffgA{{\sffg(\A)}}
\def\sffgB{{\sffg(\B)}}
\def\hprojA{{\hproj(\A)}}
\def\hprojB{{\hproj(\B)}}
\def\qrepA{{\qrep(\A)}}
\def\qrepB{{\qrep(\B)}}
\def\opp{{\text{opp}}}
\def\perfsf{{\semi-free^{\perf}}}
\def\prfhpr{{\hproj^{\scriptscriptstyle\perf}}}
\def\prfhprA{{\prfhpr(\A)}}
\def\prfhprB{{\prfhpr(\B)}}
\def\prfhprAopp{{\prfhpr(\Aopp)}}
\def\prfhprBopp{{\prfhpr(\Bopp)}}
\def\perfsfA{{\perfsf(\A)}}
\def\perfsfB{{\perfsf(\B)}}
\def\qrhpr{{\hproj^{qr}}}
\def\qrhprA{{\qrhpr(\A)}}
\def\qrhprB{{\qrhpr(\B)}}
\def\qrsf{{\semi-free^{qr}}}
\def\qrsf{{\semi-free^{qr}}}
\def\qrsfA{{\qrsf(\A)}}
\def\qrsfB{{\qrsf(\B)}}
\def\Aperfsf{{\semi-free^{\A\text{-}\perf}(\AbimB)}}
\def\Bperfsf{{\semi-free^{\B\text{-}\perf}(\AbimB)}}
\def\Aprfhpr{{\hproj^{\A\text{-}\perf}(\AbimB)}}
\def\Bprfhpr{{\hproj^{\B\text{-}\perf}(\AbimB)}}
\def\Aqrhpr{{\hproj^{\A\text{-}qr}(\AbimB)}}
\def\Bqrhpr{{\hproj^{\B\text{-}qr}(\AbimB)}}
\def\Aqrsf{{\semi-free^{\A\text{-}qr}(\AbimB)}}
\def\Bqrsf{{\semi-free^{\B\text{-}qr}(\AbimB)}}
\def\modAopp{{\modd\text{-}\Aopp}}
\def\modBopp{{\modd\text{-}\Bopp}}
\def\AmodA{{\A\text{-}\modd\text{-}\A}}
\def\AmodB{{\A\text{-}\modd\text{-}\B}}
\def\AmodC{{\A\text{-}\modd\text{-}\C}}
\def\AmodD{{\A\text{-}\modd\text{-}\D}}
\def\BmodA{{\B\text{-}\modd\text{-}\A}}
\def\BmodB{{\B\text{-}\modd\text{-}\B}}
\def\BmodC{{\B\text{-}\modd\text{-}\C}}
\def\BmodD{{\B\text{-}\modd\text{-}\D}}
\def\CmodA{{\C\text{-}\modd\text{-}\A}}
\def\CmodB{{\C\text{-}\modd\text{-}\B}}
\def\CmodC{{\C\text{-}\modd\text{-}\C}}
\def\CmodD{{\C\text{-}\modd\text{-}\D}}
\def\DmodA{{\D\text{-}\modd\text{-}\A}}
\def\DmodB{{\D\text{-}\modd\text{-}\B}}
\def\DmodC{{\D\text{-}\modd\text{-}\C}}
\def\DmodD{{\D\text{-}\modd\text{-}\D}}
\def\AbimA{{\A\text{-}\A}}
\def\AbimC{{\A\text{-}\C}}
\def\BbimA{{\B\text{-}\A}}
\def\BbimB{{\B\text{-}\B}}
\def\BbimC{{\B\text{-}\C}}
\def\CbimA{{\C\text{-}\A}}
\def\CbimB{{\C\text{-}\B}}
\def\CbimC{{\C\text{-}\C}}
\def\AbimD{{\A\text{-}\D}}
\def\BbimD{{\B\text{-}\D}}
\def\CbimD{{\C\text{-}\D}}
\def\DbimD{{\D\text{-}\D}}
\def\DbimA{{\D\text{-}\A}}
\def\DbimB{{\D\text{-}\B}}
\def\DbimC{{\D\text{-}\C}}
\def\AhprA{{\hproj\left(\AbimA\right)}}
\def\BhprB{{\hproj\left(\BbimB\right)}}
\def\AhprB{{\hproj\left(\AbimB\right)}}
\def\BhprA{{\hproj\left(\BbimA\right)}}
\def\AbarA{{\overline{\A\text{-}\A}}}
\def\AbarB{{\overline{\A\text{-}\B}}}
\def\BbarA{{\overline{\B\text{-}\A}}}
\def\BbarB{{\overline{\B\text{-}\B}}}
\def\QAbimB{{Q\A\text{-}\B}}
\def\AbimB{{\A\text{-}\B}}
\def\AonebimB{{\A_1\text{-}\B}}
\def\AtwobimB{{\A_2\text{-}\B}}
\def\BbimA{{\B\text{-}\A}}
\def\Aperf{{\A\text{-}\perf}}
\def\Bperf{{\B\text{-}\perf}}
\def\MddA{{M^{\tilde{\A}}}}
\def\MddB{{M^{\tilde{\B}}}}
\def\MhdA{{M^{h\A}}}
\def\MhdB{{M^{h\B}}}
\def\NhdB{{N^{h\B}}}
\def\Cat{{Cat}}
\def\DGCat{{DG\text{-}Cat}}
\def\HoDGCat{{\hmtpy(\DGCat)}}
\def\HoDGCatV{{\hmtpy(\DGCat_\mathbb{V})}}
\def\tr{{tr}}
\def\pretr{{pretr}}
\def\kctr{{kctr}}
\def\PreTrCat{{\DGCat^\pretr}}
\def\KcTrCat{{\DGCat^\kctr}}
\def\HoPretrCat{{\hmtpy(\PreTrCat)}}
\def\HoKcTrCat{{\hmtpy(\KcTrCat)}}
\def\Aquasirep{{\A\text{-}qr}}
\def\QAquasirep{{Q\A\text{-}qr}}
\def\Bquasirep{{\B\text{-}qr}} 
\def\lderA{{\tilde{\A}}} 
\def\lderB{{\tilde{\B}}} 
\def\adjunit{{\mathrm{adj.unit}}}
\def\adjcounit{{\mathrm{adj.counit}}}
\def\degzero{{\mathrm{deg.0}}}
\def\degone{{\mathrm{deg.1}}}
\def\degminusone{{\mathrm{deg.-1}}}
\def\bareta{{\overline{\eta}}}
\def\barzeta{{\overline{\zeta}}}
\def\Ract{{R {\action}}}
\def\barRact{{\overline{\Ract}}}
\def\actL{{{\action} L}}
\def\baractL{{\overline{\actL}}}
\def\Ainfty{{A_{\infty}}}
\def\noddinf{{{\bf Nod}_{\infty}}}
\def\noddinfstr{{{\bf Nod}^{\text{strict}}_{\infty}}}
\def\noddinfA{{\noddinf\A}}
\def\noddinfB{{\noddinf\B}}
\def\noddinfAB{{\noddinf\AbimB}}
\def\noddinfBA{{\noddinf\BbimA}}
\def\noddinfCD{{\noddinf\CbimD}}
\def\noddinfu{{({\bf Nod}_{\infty})_u}}
\def\noddinfuA{{(\noddinfA)_u}}
\def\noddinfhu{{({\bf Nod}_{\infty})_{hu}}}
\def\noddinfhuA{{(\noddinfA)_{hu}}}
\def\noddinfdg{{({\bf Nod}_{\infty})_{dg}}}
\def\noddinfdgA{{(\noddinfA)_{dg}}}
\def\noddinfdgAA{{(\noddinf\AbimA)_{dg}}}
\def\noddinfdgAB{{(\A\text{\rm-}\noddinf\text{\rm-}\B)_{dg}}}
\def\noddinfdgB{{(\noddinfB)_{dg}}}
\def\moddinf{{\modd_{\infty}}}
\def\moddinfA{{\modd_{\infty}\A}}
\def\infbar{{B_\infty}}
\def\reduced{{\text{red}}}
\def\naug{{\text{na}}}
\def\infbarnaug{{B^{\naug}_\infty}}
\def\infbarbim{{B^{\text{bim}}_\infty}}
\def\infbarA{{B^\A_\infty}}
\def\infbarB{{B^\B_\infty}}
\def\infbarC{{B^\C_\infty}}
\def\inftimes{{\overset{\infty}{\otimes}}}
\def\infhom{{\overset{\infty}{\homm}}}
\def\barhom{{\mathrm H\overline{\mathrm{om}}}}
\def\barend{{\overline{\eend}}}
\def\bartimes{{\;\overline{\otimes}}}
\def\triaA{{\tria \A}}
\def\TPairdg{{\TPair^{dg}}}
\def\algA{{\alg(\A)}}
\def\Ainfty{{A_{\infty}}}
\def\alphahat{{\hat{\alpha}}}
\def\betahat{{\hat{\beta}}}
\def\gammahat{{\hat{\gamma}}}

\theoremstyle{definition}
\newtheorem{defn}{Definition}[section]
\newtheorem*{defn*}{Definition}
\newtheorem{exmpl}[defn]{Example}
\newtheorem*{exmpl*}{Example}
\newtheorem{exrc}[defn]{Exercise}
\newtheorem*{exrc*}{Exercise}
\newtheorem*{chk*}{Check}
\newtheorem*{remarks*}{Remarks}
\theoremstyle{plain}
\newtheorem{theorem}{Theorem}[section]
\newtheorem*{theorem*}{Theorem}
\newtheorem{conj}[defn]{Conjecture}
\newtheorem*{conj*}{Conjecture}
\newtheorem{prps}[defn]{Proposition}
\newtheorem*{prps*}{Proposition}
\newtheorem{cor}[defn]{Corollary}
\newtheorem*{cor*}{Corollary}
\newtheorem{lemma}[defn]{Lemma}
\newtheorem*{claim*}{Claim}
\newtheorem{Specialthm}{Theorem}
\renewcommand\theSpecialthm{\Alph{Specialthm}}
\numberwithin{equation}{section}
\renewcommand{\textfraction}{0.001}
\renewcommand{\topfraction}{0.999}
\renewcommand{\bottomfraction}{0.999}
\renewcommand{\floatpagefraction}{0.9}
\setlength{\textfloatsep}{5pt}
\setlength{\floatsep}{0pt}
\setlength{\abovecaptionskip}{2pt}
\setlength{\belowcaptionskip}{2pt}
\begin{abstract}
Given a DG-category $\A$ we introduce the 
\em bar category of modules \rm $\modbarA$. It is a DG-enhancement
of the derived category $D(\A)$ of $\A$ which is isomorphic to 
the category of DG $\A$-modules with $\Ainfty$-morphisms between 
them. However, it is defined intrinsically in the language of 
DG-categories and requires no complex machinery or sign conventions
of $\Ainfty$-categories. We define for these bar categories 
Tensor and Hom bifunctors, dualisation functors, and a convolution
of twisted complexes. The intended application is to working with 
DG-bimodules as enhancements of exact functors between triangulated 
categories. As a demonstration we develop a homotopy 
adjunction theory for tensor functors between derived categories 
of DG-categories. It allows us to show in an enhanced setting that   
given a functor $F$ with left and right adjoints $L$ and $R$ 
the functorial complex 
$FR \xrightarrow{F\action{R}} FRFR \xrightarrow{FR\trace - \trace{FR}} 
FR \xrightarrow{\trace} \id$ lifts to a canonical twisted 
complex whose convolution is the square of the spherical twist of $F$. 
We then write down four induced functorial Postnikov systems 
computing this convolution. 
\end{abstract}

\maketitle

\section{Introduction}
\label{section-introduction}

DG-enhancements of triangulated categories were introduced by
Bondal and Kapranov in
\cite{BondalKapranov-EnhancedTriangulatedCategories} to overcome
the axiomatic imperfections of the latter 
\cite{Verdier-DesCategoriesDeriveesDesCategoriesAbeliennes}. 
A DG-enhancement of a triangulated category $\mathcal{T}$ is a
pretriangulated differentially graded (DG) category $\A$ with 
$H^0(\A) \simeq \T$. Working formally in $\A$ and truncating 
down to $\T$ fixes a number of issues. Many constructions 
of this kind are independent of the choices of lifts to $\A$ and 
even $\A$ itself. See \cite{Keller-DerivingDGCategories},
\cite{Toen-LecturesOnDGCategories}, 
\cite[\S2]{AnnoLogvinenko-SphericalDGFunctors} for an introduction to 
DG-categories, 
\cite{BondalKapranov-EnhancedTriangulatedCategories},
\cite{LuntsOrlov-UniquenessOfEnhancementForTriangulatedCategories}, 
\cite[\S4]{AnnoLogvinenko-SphericalDGFunctors} for an introduction 
to DG-enhancements, and  
\cite{Toen-TheHomotopyTheoryOfDGCategoriesAndDerivedMoritaTheory}
for some key technical results.  

Most triangulated categories which arise in algebra and geometry 
are derived categories. These are $H^0(-)$ truncations of 
DG-categories of complexes of objects in an abelian category, 
and thus possess a natural DG-enhancement. 
Examples include the derived categories of sheaves or constructible 
sheaves on a topological space or of quasi-coherent sheaves, coherent 
sheaves, or $D$-modules on an algebraic variety. Moreover, 
in a number of these examples (e.g. the derived
category of any Grothendieck category) this natural enhancement 
is unique up to quasi-equivalence
\cite{LuntsOrlov-UniquenessOfEnhancementForTriangulatedCategories}, 
\cite{CanonacoStellari-UniquenessOfDGEnhancementsForTheDerivedCategoryOfAGrothendieckCategory}.
Thus functorial constructions carried out in 
a DG-enhancement will produce the same results in the triangulated 
category regardless of the choice of the enhancement. 
For technical reasons, it is best to work in a Morita framework:
instead of enhancing $\T$ with some DG-category $\A$ we enhance
$\T$ with the DG-category $\modA$ of modules over $\A$. 

Let $\T$ and $\U$ be triangulated categories. The exact functors
$\T \rightarrow \U$ do not form a triangulated category. 
However, if we Morita enhance $\T$ and $\U$ by DG-categories 
$\A$ and $\B$, by a fundamental result of To{\" e}n 
\cite{Toen-TheHomotopyTheoryOfDGCategoriesAndDerivedMoritaTheory}
the enhanceable functors form a triangulated category
equivalent to the derived category of $
\B$-perfect $\A$-$\B$-bimodules. In algebraic 
geometry, let $\T = D(X)$ and $\U = D(Y)$
be the derived categories of coherent sheaves 
of separated schemes $X$ and $Y$ of finite type over a field. 
The DG-category of perfect $\AbimB$-bimodules DG-enhances 
$D(X \times Y)$ and the enhanceable functors are the Fourier-Mukai transforms 
\cite{Toen-TheHomotopyTheoryOfDGCategoriesAndDerivedMoritaTheory}, 
\cite{LuntsSchnurer-NewEnhancementsOfDerivedCategoriesOfCoherentSheavesAndApplications}. Thus studying 
$\AbimB$-bimodules as DG-enhancements of exact functors 
$D(\A) \rightarrow D(\B)$ can be viewed as the universal 
Fourier-Mukai theory for enhanced triangulated categories. 

In their work on spherical functors \cite{AnnoLogvinenko-SphericalDGFunctors} 
and $\mathbb{P}^n$-functors \cite{AnnoLogvinenko-PFunctors}
the authors used the above approach for various 
technical constructions such as taking cones of natural 
transformations and, more generally, convolutions of complexes of functors. 
Though working in quasi-equivalent enhancements produces 
the same results, we discovered that for carrying out 
explicit computations choosing a suitable enhancement makes
a world of difference. In this paper we define and study 
the DG-enhancement framework for the derived categories 
of DG-modules and bimodules which we found most suitable. 
As a demonstration, we construct an explicit homotopy 
adjunction theory and thus explicit $2$-categorical adjunctions
for DG-bimodules. 

Let $\A$ be a DG-category and let $\modA$ be the DG-category 
of right $\A$-modules. Two enhancements commonly used in 
the literature for $D(\A)$ are the subcategory $\hprojA$ of
the $h$-projective modules in $\modA$, and the Drinfield quotient 
$\modA / \acyc(\A)$ by the subcategory $\acyc(\A)$ of 
acyclic modules \cite{Drinfeld-DGQuotientsOfDGCategories}. 
Neither turned out to be suitable for our purposes. 
The problem with the Drinfeld quotient is that its morphisms 
are inconvenient to work with explicitly.  The problem with 
$\hprojA$ is that when working with bimodules
the diagonal bimodule $\A$, which corresponds to the
identity functor $D(\A) \rightarrow D(\A)$, is not $h$-projective.
Hence every construction involving the identity functor has to be
$h$-projectively resolved leading to many formulas becoming 
vastly more complicated than they should be, 
as seen in \cite{AnnoLogvinenko-SphericalDGFunctors}.

This can be fixed by working with DG $\A$-modules
and $\Ainfty$-morphisms between them. In other words, 
the full subcategory $\noddinfdgA$ of the DG-category 
$\noddinfA$ of $\Ainfty$ $\A$-modules which consists 
of DG-modules. In $\noddinfdgA$ all quasi-isomorphisms are 
already homotopy equivalences, there is no need to 
take resolutions and $D(\A) \simeq H^0(\noddinfdgA)$. 
However, the machinery of $\Ainfty$-categories lends itself 
poorly to explicit computations due to e.g. 
the complicated sign conventions involved. It seems
wasteful to use the full generality of the $\Ainfty$-language
to only consider $\Ainfty$-morphisms between DG-modules
over a DG-category. To this end we introduce the 
\em bar category of modules \rm over $\A$, 
a technical tool which simplifies the $\Ainfty$-machinery involved 
to the extent actually necessary. It is a category 
\em isomorphic \rm to $\noddinfdgA$, yet it has an 
intrinsic definition entirely in terms of DG-modules 
and the bar complex $\barA$. This builds on the ideas  
of \cite[\S6.6]{Keller-DerivingDGCategories}. 
Keller works with a set of compact generators to obtain 
a Morita enhancement of $D(\A)$. We work with all the $\A$-modules
to obtain a usual enhancement of $D(\A)$ and we establish the 
isomorphism to $\noddinfdgA$:
\begin{defn*}[Definition \ref{defn-bar-category-of-modules}]
Let $\A$ be a DG-category. Define the \em bar category of modules \rm
$\modbarA$ as follows:
\begin{itemize}
\item 
The object set of $\modbarA$ is the same as that of $\modA$:
DG-modules over $\A$.  
\item For any $E,F \in \modA$ set
\begin{align*}
\homm_{\modbarA}(E,F) = \homm_\A(E \otimes_\A \barA, F)
\end{align*}
and write $\barhom_\A(E,F)$ to denote this $\homm$-complex. 
\item For any $E \in \modA$ set $\id_E \in \barhom_\A(E,E)$
to be the element given by 
\begin{align*}
E \otimes_\A \barA \xrightarrow{\id \otimes \tau} E \otimes_\A \A
\xrightarrow{\sim} E
\end{align*}
where $\tau \colon \barA \rightarrow \A$ is the canonical projection. 
\item For any $E,F,G \in \modA$ define the composition map 
\begin{align*}
\barhom_\A(F,G) \otimes_k \barhom_\A(E,F)
\longrightarrow 
\barhom_\A(E,G)
\end{align*}
by setting for any $\alpha\colon E \otimes_\A \barA \rightarrow F$
and $\beta\colon F \otimes_\A \barA \rightarrow G$ 
their composition to be 
\begin{align*}
E \otimes_\A \barA \xrightarrow{\id \otimes \Delta}
E \otimes_\A \barA \otimes_\A \barA
\xrightarrow{\alpha \otimes \id}
F \otimes_\A \barA 
\xrightarrow{\beta}
G
\end{align*}
where $\Delta \colon \barA \rightarrow \barA \otimes_\A \barA$ is the
canonical comultiplication. 
\end{itemize}
\end{defn*}
In Prop.~\ref{prps-modbarA-to-noddinfdgA-isomorphism} we show
that $\modbarA$ is isomorphic to $\noddinfdgA$, thus it is also 
a DG-enhancement of $D(\A)$.
Let $\B$ be another DG category. We similarly 
define $\AmodbarB$, the bar category of DG-bimodules. 
We write down bifunctors $\bartimes$ and $\barhom$
for DG-bimodules which correspond to their $\Ainfty$-counterparts. 
In Prop.~\ref{prps-bartimes-and-barhom-are-dg-adjoint}
we prove the Tensor-Hom adjunction for 
$\bartimes$ and $\barhom$ and give formulas for 
its adjunction\hspace{0.10cm}units and counits.
Next is the dualisation theory: we define the 
functors $(-)^\barA$ and $(-)^\barB$ which are equivalences on 
the subcategories of $\A$- and $\B$-perfect bimodules, 
respectively. We show in 
Lemma~\ref{lemma-bringing-a-factor-into-the-barhom-is-a-homotopy-equivalence}
that for any $M \in \AmodbarB$ the bimodules $M^\barA$ and $M^\barB$ 
enhance the derived functors $\rder\homm_\A(M,-)$ and
$\rder\homm_\B(M,-)$ whenever $M$ is $\A$- and $\B$-perfect, respectively. 
This is crucial for the homotopy adjunction theory we develop later. 

The constructions such as cones of natural transformations
or convolutions of complexes of functors are usually done 
via twisted complexes \cite{BondalKapranov-EnhancedTriangulatedCategories}, 
\cite[\S3]{AnnoLogvinenko-SphericalDGFunctors}. Ideally, 
this needs the DG-enhancement to be strongly pretriangulated,
which the bar category $\AmodbarB$ is not. 
However, in Defn.~\ref{defn-convolution-in-modbarA} we define 
the convolution functor 
which is a quasi-equivalence and a homotopy inverse to 
the inclusion $\AmodbarB \hookrightarrow \pretriag(\AmodbarB)$. 
Together with the formulas for Tensor, Hom and dualisation of 
twisted complexes of bimodules given in 
Lemmas~\ref{lemma-tensoring-and-homming-for-twisted-complexes}
and \ref{lemma-duals-of-twisted-complexes}
it enables the constructions we need. 

Next we examine the biggest drawback of bar categories:
the natural map $\A \bartimes_\A M \xrightarrow{\alpha} M$, 
analogous to the $\A$-action isomorphism $\A \otimes_\A M \simeq M$, 
is not an isomorphism, but only a homotopy equivalence. 
We study the higher homotopies involved. It is well known 
that any homotopy equivalence in a DG-category can be completed
to a certain universal system
\eqref{eqn-quiver-presentation-of-generating-cofibration-category}
of morphisms and relations
\cite[\S3.7]{Drinfeld-DGQuotientsOfDGCategories}\cite{Tabuada-UneStructureDeCategorieDeModelesDeQuillenSurLaCategorieDesDG-Categories}
\cite[App]{AnnoLogvinenko-SphericalDGFunctors}. We do better and write down 
a homotopy inverse $\beta_0\colon M \rightarrow \A \bartimes_\A M$
and a degree $-1$ map $\theta$ for which: 
\begin{prps*}[Prop.~\ref{prps-alpha-beta-theta}]
Let $\A$ and $\B$ be DG-categories and let $M \in \AmodbarB$. 
The sub-DG-category of $\AmodbarB$ generated by $\alpha, \beta_0$ and 
$\theta$ is the free DG-category generated by these modulo
the following relations:
\begin{equation}
\label{eqn-quiver-presentation-of-alpha-beta-theta-intro}
\begin{minipage}[c][1in][c]{1.5in}
$d\alpha = d\beta_0 = 0$, \\
$d\theta = \id - \beta_0 \circ \alpha$, \\
$0 =  \alpha \circ \beta_0 - \id$, \\
$\alpha \circ \theta = 0$. 
\end{minipage}
\quad\quad
\begin{tikzcd}[column sep={2cm},row sep={1.5cm}] 
M
\ar[bend left=20]{rr}{\beta_0}
& &
\A \bartimes_\A M
\ar[bend left=20]{ll}{\alpha}
\ar[out=30, in=-30,loop,distance=6em, dotted]{}{\theta}
\end{tikzcd}
\end{equation}
\end{prps*}
The relations in 
\eqref{eqn-quiver-presentation-of-alpha-beta-theta-intro} 
can be obtained from those in 
the universal system 
\eqref{eqn-quiver-presentation-of-generating-cofibration-category}
by setting $\theta_x = 0$, $\alpha \circ \theta_y = 0$, 
and $\phi = - \theta_y^2 \circ \beta$ in the notation thereof. 
In Prop.~\ref{prps-gamma-delta-kappa}
we prove analogous results 
for the adjoint homotopy equivalence $\gamma\colon M \rightarrow \barhom_\A(\A,M)$.  

In the latter half of the paper we use the bar categories to give 
a homotopy adjunction theory for DG-bimodules. 
Our first main result is the following straightforward, but very useful fact:
\begin{theorem}[cf. Theorem \ref{theorem-TFAE-M-is-B-and-A-perfect}]
Let $\A$ and $\B$ be DG-categories and let $f\colon D(\A) \rightarrow
D(\B)$ be a tensor functor. Let $M \in \AmodbarB$ be any enhancement of $f$. 
\begin{enumerate}
\item 
\label{item-intro-TFAE-M-is-B-perfect}
The following are equivalent:
\begin{enumerate}
\item 
\label{item-intro-TFAE-M-is-B-perfect-item-right-adjoint-is-cts}
The right adjoint $r$ of $f$ is continuous.
\item 
\label{item-intro-TFAE-M-is-B-perfect-item-f-preserves-compact-objects}
$f$ restricts to $D_{c}(\A) \rightarrow D_{c}(\B)$. 
\item 
\label{item-intro-TFAE-M-is-B-perfect-item-M-is-B-perfect}
$M$ is $\B$-perfect.
\item 
\label{item-intro-TFAE-M-is-B-perfect-item-MbarB-is-the-right-adjoint}
$M^{\barB}$ enhances the right adjoint $r$ of $f$. 
\end{enumerate}
\item 
\label{item-intro-TFAE-M-is-A-perfect}
The following are equivalent:
\begin{enumerate}
\item 
\label{item-intro-TFAE-M-is-A-perfect-item-l-exists}
The left adjoint $l$ of $f$ exists.
\item 
\label{item-intro-TFAE-M-is-A-perfect-item-l-preserves-compact-objects}
The left adjoint $l$ of $f$ exists and restricts to $D_{c}(\B) \rightarrow D_{c}(\A)$.
\item 
\label{item-intro-TFAE-M-is-A-perfect-item-l-is-A-perfect}
$M$ is $\A$-perfect. 
\item 
\label{item-intro-TFAE-M-is-A-perfect-item-MbarA-is-the-left-adjoint}
$M^{\barA}$ enhances the left adjoint $l$ of $f$.  
\end{enumerate}
\end{enumerate}
\end{theorem}

In the $2$-categorical language
of \cite{Benabou-IntroductionToBicategories}, let 
$\DGMod$ be the bicategory whose objects are DG-categories, 
whose 1-morphism categories are the derived categories of DG-bimodules,
whose composition is given by the derived tensor product, 
and whose identity object is the diagonal bimodule. 
Let $M \in \AmodbarB$ be $\A$- and $\B$-perfect. 
Write $F$, $L$, and $R$ for $M$, $M^\barA$, and $M^\barB$
considered as $1$-morphisms in $\DGMod$. 
In Defns.~\ref{defn-homotopy-trace-maps}-\ref{defn-homotopy-action-maps}
we write down the homotopy adjunction units and counits for $(L,F,R)$
$$ \id_\A \xrightarrow{\action} RF, \quad 
\id_\B \xrightarrow{\action} FL, \quad 
FR \xrightarrow{\trace} \id_\B, \quad
LF \xrightarrow{\trace} \id_\A, $$
and in Proposition \ref{prps-homotopy-adjunction} we show that the
following identities hold up to homotopy: 
\begin{align}
\label{eqn-intro-F-FRF-F-and-R-RFR-R-maps}
F \xrightarrow{F\action} FRF \xrightarrow{{\trace}F} F = \id_F
\quad \text{ and } \quad R \xrightarrow{\action R} RFR
\xrightarrow{R\trace} R = \id_R, 
\\
\label{eqn-intro-F-FLF-F-and-L-LFL-L-maps}
F \xrightarrow{{\action}F} FLF \xrightarrow{F{\trace}} F = \id_F
\quad \text{ and } \quad L \xrightarrow{L{\action}} LFL 
\xrightarrow{{\trace}L} L = \id_L. 
\end{align}
This implies $2$-categorical adjunctions between $L$, $F$,
and $R$ in $\DGMod$. Such $2$-categorical adjunctions, 
when they exist, are strongly unique: $L$ and $R$ are
determined by $F$ up to the unique isomorphism. 

An exact functor $f\colon D(\A) \rightarrow D(\B)$ is a \em
tensor \rm functor if it is 
isomorphic to tensor multiplication by some $M \in \AmodB$. 
In Theorem \ref{theorem-TFAE-M-is-B-and-A-perfect} we show $f$ 
has left and right adjoints $l$ and $r$ which are also tensor functors 
if and only if $M$ is $\B$- and $\A$-perfect. Our homotopy adjunction
theory then implies that if we fix an enhancement of such $f$
the enhancements of $l$ and $r$ which  
satisfy \eqref{eqn-intro-F-FRF-F-and-R-RFR-R-maps} and 
\eqref{eqn-intro-F-FLF-F-and-L-LFL-L-maps} exist and are unique. 

Next we study the homotopies up to which  
\eqref{eqn-intro-F-FRF-F-and-R-RFR-R-maps} 
and \eqref{eqn-intro-F-FLF-F-and-L-LFL-L-maps}
hold, and the relations between these homotopies. 
These can be packaged up as several canonical twisted 
complexes associated to a homotopy adjunction. We
write down explicit degree $-1$ maps $\xi_\B$, $\xi'_\B$,
$\upsilon_\B$,
and a degree $-2$ map $\nu_\B$ such that:

\begin{theorem}[cf. Theorems 
\ref{theorem-canonical-twisted-complex-associated-to-homotopy-adjunction}
and 
\ref{theorem-the-convolution-of-FR-FRFR-FR-Id-is-T^2-etc}]
The following are twisted complexes over $\BmodbarB$ and $\AmodbarA$:
\begin{small}
\begin{equation}
\label{eqn-intro-FR-FRFR-FR-Id-twisted-complex}
\begin{tikzcd}[column sep={3cm},row sep={1.5cm}] 
FR
\ar{r}{F{\action}R}
\ar[bend left=15,dashed]{rr}[']{\xi'_\B}
&
FRFR
\ar{r}{FR\trace - \trace FR}
&
FR
\ar{r}{\trace}
&
\id_\B
\end{tikzcd}
\end{equation}
\begin{equation}
\label{eqn-intro-Id-RF-RFRF-RF-twisted-complex}
\begin{tikzcd}[column sep={3cm},row sep={1.5cm}] 
\id_\A
\ar{r}{{\action}}
\ar[bend left=16,dashed]{rrr}[']{\nu_\B}
\ar[bend left=15,dashed]{rr}[']{-\upsilon_\B}
&
RF
\ar{r}{{\action}RF - RF{\action}}
\ar[bend left=15,dashed]{rr}[']{\xi_\B}
&
RFRF
\ar{r}{R{\trace}F}
&
RF
\end{tikzcd}
\end{equation}
\end{small}
Their convolutions are homotopic 
to $T^2$ and $C^2$, the squares of the spherical twist and co-twist of $F$. 
\end{theorem}
We also give analogous twisted complexes for $FL$ and $LF$. 
The spherical twist and cotwist of a functor
are well-studied notions with numerous applications. 
If $F$ is fully faithful, then $C = 0$ and $T$ is the
mutation functor which kills $\A$ and maps $\leftidx{^\perp}\A$ equivalently 
to $\A^\perp$ 
\cite{Bondal-RepresentationOfAsssociativeAlgebrasAndCoherentSheaves}. 
If $F$ is spherical \cite{AnnoLogvinenko-SphericalDGFunctors}, 
then $C$ and $T$ are autoequivalences. Moreover, as spherical
functors are $\mathbb{P}^1$-functors
\cite[Prop.~7.1]{AnnoLogvinenko-PFunctors}, the theorem above shows that
the $\mathbb{P}$-twist of $F$ is isomorphic to $T^2$, generalising
similar statements made for split $\mathbb{P}^1$-functors in
\cite[Ex.~4.2(3)]{Addington-NewDerivedSymmetriesOfSomeHyperkaehlerVarieties} and
$\mathbb{P}^1$-objects in 
\cite[Prop. 2.9]{HuybrechtsThomas-PnObjectsAndAutoequivalencesOfDerivedCategories}. 

Finally, in \S\ref{section-derived-category-perspective} we
intepret the data of 
\eqref{eqn-intro-FR-FRFR-FR-Id-twisted-complex} in terms of
the derived category $D(\BbimB)$ of $\BbimB$-bimodules.
Any twisted complex in the DG-enhancement defines 
several \em Postnikov systems \rm
in the triangulated category which compute its convolution and the
convolutions of its subcomplexes, cf.
\S\ref{section-zolotom-po-mramoru}. 
The data of \eqref{eqn-intro-FR-FRFR-FR-Id-twisted-complex}
defines four Postnikov systems in $D(\BbimB)$. These turn out
to have an intrinsic description we give in Theorem 
\ref{theorem-canonical-postnikov-systems-for-FR-FRFR-FR-Id}
which implies a number of explicit relations 
between various natural morphisms involving $\id$, $FR$, $FRFR$, $T$, 
$FRT$, $T^2$ and extensions thereof which are highly useful in computations 
of spherical and $\mathbb{P}$-twists. 

\subsection{The layout of the paper} 
In \S\ref{section-preliminaries} we give
prerequisites on DG-categories and $A_\infty$-categories, 
and on their modules and bimodules. 
In \S\ref{section-bar-category-of-modules} we 
introduce the bar categories of modules and bimodules.
Then in \S\ref{section-homotopy-adjunction-for-tensor-functors} 
we construct homotopy adjunctions for DG-bimodules 
and their associated twisted complexes. 

\subsection{Acknowledgements} The authors would like to thank
Alexander Efimov and Sergey Arkhipov for helpful discussions. They would also like to thank Olaf Schn{\"u}rer for many useful comments which led to significant improvements in the manuscript.  The first author would like 
to thank Kansas State University for a stimulating research 
environment, as well as for inviting the second author to visit.  
The second author would like to offer similar thanks to Cardiff University. 
Both authors would like to thank Banff International Research Station for Mathematical Innovation 
and Discovery (BIRS) for extending its hospitality to them during 
the ``Homological Mirror Geometry'' workshop in March 2016. 

\section{Preliminaries}
\label{section-preliminaries}

Let $k$ be any field. Throughout this paper,  we work over $k$ as the
base field and all the categories we consider are $k$-linear.  Some of
our results, e.g. the  definition of the bar-category of modules and
its isomorphism to the category of DG-modules with $\Ainfty$-morphisms
between them, work just as well when $k$ is only a commutative ring.
However, crucially, Lemma
\ref{lemma-over-a-field-N-otimes_k-A-is-semi-free} stops working when
$k$ is not a field: DG-tensoring over $k$ with the diagonal bimodule 
no longer necessarily produces a semi-free module.
Consequently, $\Ainfty$-tensoring with the diagonal bimodule is no
longer a functorial semi-free resolution as described in
\S\ref{section-functorial-semi-free-resolution-for-noddinfhuA} and
$\Ainfty$-quasi-isomorphisms are not necessarily all homotopy
equivalences.  Thus when $k$ is not a field the bar-category of
modules $\modbarA$ of a DG-category $\A$, while still being
a well-defined pretriangulated DG category, isn't necessarily a
DG-enhancement of the derived category of $\A$. 

We use the following notation for the derived categories we work with. 
For a DG-category $\A$ we denote by $D(\A)$ the derived category of right DG $\A$-modules, cf. \S\ref{section-preliminaries-on-DG-categories}. For an $\Ainfty$-category $\A$ we denote by $D(\A)$ the derived category of right $\Ainfty$ $\A$-modules, cf. \S\ref{section-the-derived-category-of-an-Ainfty-category}. In case of a DG-category $\A$ we further denote by $D_\infty(\A)$ the derived category of right $\Ainfty$ $\A$-modules as per \S\ref{section-the-derived-category-of-an-Ainfty-category}. Similarly, given two DG or $\Ainfty$-categories $\A$ and $\B$ we denote by $D(\AbimB)$ the derived category of the corresponding $\AbimB$-bimodules, etc.  For all these triangulated categories, we denote by $D_c(\bullet)$ their full subcategories consisting of compact objects. For a scheme $X$ over $k$ we denote by $D_{qc}(X)$ the derived category of complexes of $\mathcal{O}_X$-modules with quasi-coherent cohomologies and by $D(X)$ the derived category of complexes with coherent and bounded cohomologies. 

We also need to introduce a notation for maps between direct sums 
of modules. For any two direct sums
$E = \oplus_{i=1}^n E_i$ and $F = \oplus_{j=1}^m F_j$ of objects in an 
additive category we denote any map $$\alpha\colon E \rightarrow F$$
between them by the $m \times n$-matrix $(\alpha_{ij})$ where 
each $\alpha_{ij}$ is the restriction of $\alpha$ to a map $E_j
\rightarrow F_i$. 

\subsection{DG-categories, modules, and bimodules}
\label{section-preliminaries-on-DG-categories}

For a brief introduction to DG-categories, DG-modules, and 
the technical notions involved we direct the reader to 
\cite{AnnoLogvinenko-SphericalDGFunctors}, 
\S2-4. The present paper was written with that survey in mind. 
We employ freely any notion or piece of notation 
introduced in \cite{AnnoLogvinenko-SphericalDGFunctors},
\S2-4. However, for the convenience of the reader, below we briefly 
summarise some of the most relevant facts and notation. 

Let $\A$ be a DG-category. A (right) $\A$-module is a DG-functor 
from $\Aopp$ to $\modk$, the DG-category of DG
$k$-modules. Its underlying graded module is its 
composition with the forgetful functor to the graded
category of graded $k$-modules. 
Given an $\A$-module $E$ for each $a \in \A$ we write
$E_a$ for the complex $E(a) \in \modk$. An $\A$-module $E$ is 
\em acyclic \rm if it is acyclic levelwise in $\modk$, i.e. for 
any $a \in \A$ the complex $E_a$ is acyclic. 

Let $E$ and $F$ be $\A$-modules. Define $\homm_{\A}(E,F)$
to be the DG complex of $k$-modules whose $i$-th graded part 
consists of all degree $i$ natural transformations of the graded 
modules underlying $E$ and $F$ and whose differential is
defined levelwise in $\modk$. The DG category $\modA$ has 
$\A$-modules as objects, morphism complexes $\homm_{\A}(E,F)$, and 
the composition defined levelwise in $\modk$. 
An $\A$-module $E$ is \em $h$-projective \rm  
if the complex $\homm_\A(E,A)$ is acyclic for any acyclic $A$.  
Write $\acyc(\A)$ and $\hprojA$ for the full subcategories of $\modA$
consisting of acyclic and $h$-projective modules, respectively.  

A morphism in $\modA$ is a \em quasi-isomorphism \rm
if it is a quasi-isomorphism levelwise in $\modk$. 
The \em derived category $D(\A)$ \rm of $\A$ is the
localisation of the homotopy category $H^0(\modA)$ 
by quasi-isomorphisms. It is constructed as the
Verdier quotient of $H^0(\modA)$ by $H^0(\acyc (\A))$. 
It comes, in particular, with the
canonical projection $H^0(\modA) \rightarrow D(\A)$
whose composition with the inclusion 
$H^0(\hprojA) \hookrightarrow H^0(\modA)$ 
is an equivalence of categories. An $\A$-module $E$ is \em perfect \rm
if its image in $D(\A)$ is a compact object, i.e. the functor
$\homm_{D(\A)}(E,-)$ commutes with infinite direct sums.  

Let $\B$ be another DG category. An \em $\A$-$\B$-bimodule \em is 
an $\Aopp \otimes_k \B$-module. We write $\AmodB$ for the DG category
of $\A$-$\B$-bimodules. For any $\A$-$\B$-bimodule $M$ and any $a \in \A$
and $b \in \B$ write $\aM_b$, $\aM$, and $M_b$ 
for the complex $M(a,b) \in \modk$, the $\B$-module $M(a,-)$, and 
the $\Aopp$-module $M(-,b)$, respectively. 
The bimodule $M$ is $\A$-perfect if it is perfect 
levelwise in $\modA$, i.e. for any $b \in \B$ the $\A$-module $M_b$
is perfect. We similarly define $\B$-perfection and 
$\A$- and $\B$-versions of acyclicity, $h$-projectivity, etc. 

The \em diagonal bimodule \rm $\A \in \AmodA$ is defined by 
$$ \aA_b = \homm_\A(b,a) $$ 
with the left and right $\A$-action given by the post- and 
pre-composition in $\A$, respectively. 
The composition in $\A$ defines a canonical map 
\begin{align}
\label{eqn-canonical-map-Aotimes_kA-to-A}
\A \otimes_k \A \mapsto \A
\end{align}
where $\otimes_k$ on the LHS denotes the tensor product 
of $\A$ as an $\A$-$k$-bimodule with $\A$ as a $k$-$\A$-bimodule. 

For any $\A$-$\B$-bimodule $M$ we have canonical maps
\begin{align}
\A \xrightarrow{\action} \homm_\B(M,M) \\
\B \xrightarrow{\action} \homm_\A(M,M) 
\end{align}
in $\AmodA$ and $\BmodB$, respectively. These are called 
\em $\A$- and $\B$-action maps\rm, respectively, because
they represent the action of $\A$ (resp. $\B$) on $M$ 
by $\B$-module (resp. $\A$-module) morphisms. Note that 
e.g. the bimodule $\homm_\B(M,M)$ has an $\A$-algebra structure
defined by the composition. It therefore defines a DG-category
with the same set of objects as $\A$. This DG-category is 
precisely the image of the functor $\A \rightarrow \modB$ 
which corresponds to $M$, cf.  
\cite[\S2.1.5]{AnnoLogvinenko-SphericalDGFunctors}. 

For any $\A$-$\B$-bimodule $M$ 
the \em shift of $M$ by $n \in \mathbb{Z}$ to the left \rm is the
$\A$-$\B$-bimodule $M[n]$ defined by 
$$ \left(_a\left(M[n]\right))_b\right)_i = (\aMb)_{i+n} $$
where $\B$ acts naturally, $\A$ acts with a sign twist 
$a._{M[n]}m = (-1)^{n \deg(a)} a._{M} m$, and the differential is $(-1)^n d_M$. 

Let $\C$ and $\D$ be DG-categories. For any 
$M \in \AmodB$, $L \in \CmodB$, and $N \in \DmodB$
we have the \em composition \rm  map in $\DmodC$
$$ \homm_\B(M,N) \otimes_\A \homm_\B(L,M)
\xrightarrow{\composition} \homm_\B(L,N) $$
which is defined levelwise by composition in $\modB$. 

Let $M \in \AmodB$. We have the usual Tensor-Hom adjunction: 
for any DG-category $\C$ 
$$ (-) \otimes_\A M\colon \CmodA \rightarrow \CmodB $$
is left adjoint to 
$$ \homm_\B(M,-) \colon \CmodB \rightarrow \CmodA. $$
The adjunction counit 
\begin{align}
\label{eqn-adjunction-counit-for-(-)-otimes-M} 
\homm_\B\bigl(M,-\bigr) \otimes_\A M \xrightarrow{\ev} \id 
\end{align}
is called the \em evaluation map\rm, as it is defined by
$$ \alpha \otimes m \mapsto \alpha(m). $$
Similarly, we call the adjunction unit
\begin{align}
\label{eqn-adjunction-unit-for-(-)-otimes-M} 
\id \xrightarrow{\mlt} \homm_\B\bigl(M,(-) \otimes_\A M\bigr) 
\end{align}
the \em tensor multiplication map\rm, as it is defined by
$$ s \mapsto s \otimes (-). $$

Analogously,
$$ M \otimes_{\B} (-) \colon \BmodC \rightarrow \AmodC $$
is left adjoint to 
$$ \homm_\Aopp(M,-) \colon \AmodC \rightarrow \BmodC $$
with the adjunction counit
\begin{align}
\label{eqn-adjunction-counit-for-M-otimes-(-)} 
M \otimes_\B \homm_\Aopp\bigl(M,-\bigr) & \xrightarrow{\ev} \id 
\\
\nonumber
m \otimes \alpha & \mapsto (-1)^{\deg(m)\deg(\alpha)} \alpha(m)
\end{align}
and the adjunction unit
\begin{align}
\label{eqn-adjunction-unit-for-M-otimes-(-)} 
\id &\xrightarrow{\mlt}
\homm_\Aopp\bigl(M,M\otimes_\B(-)\bigr) 
\\
\nonumber
s &\mapsto (-1)^{\deg(-)\deg(s)} (-) \otimes s.
\end{align}

Let $M \in \AmodB$. The action of $\A$ on $M$ defines 
the canonical isomorphism
\begin{align}
\label{eqn-DG-A-otimes-M-to-M-isomorphism}
\A \otimes_\A M &\xrightarrow{\sim} M \\
a \otimes m &\mapsto a.m. 
\end{align}
The right adjoint of \eqref{eqn-DG-A-otimes-M-to-M-isomorphism}  
with respect to $(-) \otimes_\A M$ is 
the $\A$-action map $\A \xrightarrow{\action} \homm_\B(M,M)$. 
The right adjoint of \eqref{eqn-DG-A-otimes-M-to-M-isomorphism}
with respect to $\A \otimes_\A (-)$ is the canonical isomorphism 
\begin{align}
\label{eqn-DG-M-to-hom-A-M-isomorphism}
M &\xrightarrow{\sim} \homm_\A\left(\A,M\right) \\
m &\mapsto (-1)^{\deg(-)\deg(m)} (-).m. 
\end{align}
Similarly, we have canonical isomorphisms 
\begin{align}
\label{eqn-DG-M-otimes-B-to-M-isomorphism}
M \otimes_\B \B \xrightarrow{\sim} M, \\
\label{eqn-DG-M-to-hom-B-M-isomorphism}
M \xrightarrow{\sim} \homm_\B(\B,M). 
\end{align}

The canonical isomorphisms \eqref{eqn-DG-M-to-hom-A-M-isomorphism}
and \eqref{eqn-DG-M-to-hom-B-M-isomorphism} identify evaluation maps 
with composition maps. E.g. 
$$
\homm_\B(M,E) \otimes_\A M 
\xrightarrow{\id \otimes \eqref{eqn-DG-M-to-hom-B-M-isomorphism} } 
\homm_\B(M,E) \otimes_\A \homm_\B(\B,M) 
\xrightarrow{\composition}
\homm_\B(\B, E) 
\xrightarrow{\eqref{eqn-DG-M-to-hom-B-M-isomorphism}^{-1}}
E
$$ 
is the evalution map \eqref{eqn-adjunction-counit-for-(-)-otimes-M}. 

Finally, for all $M \in \AmodB$, $N \in \DmodB$ and $L \in \CmodA$
we have a canonical map 
\begin{align}
\label{eqn-bringing-a-factor-into-the-hom} 
L \otimes_\A \homm_\B(N,M) &\longrightarrow \homm_\B(N,L \otimes_\A M) \\
\nonumber
l \otimes \alpha &\mapsto l \otimes \alpha(-)
\end{align}
which is a quasi-isomorphism when $N$ is $\B$-perfect or $L$ is 
$\A$-perfect, cf.~\cite[\S2.2]{AnnoLogvinenko-SphericalDGFunctors}.
We can also write $\eqref{eqn-bringing-a-factor-into-the-hom}$ as
the composition
\begin{align*}
L \otimes_\A \homm_\B(N,M) 
\xrightarrow{\mlt \otimes \id}
\homm_\B(M, L \otimes_\A M)
\otimes_\A
\homm_\B(N,M)
\xrightarrow{\composition}
\homm_\B(N, L \otimes_\A M).
\end{align*}

\subsection{Twisted complexes and their convolutions}

Twisted complexes were originally defined in
\cite{BondalKapranov-EnhancedTriangulatedCategories}, 
and then re-defined in 
\cite{BondalLarsenLunts-GrothendieckRingofPretriangulatedCategories}. 
For the convenience of the reader, we give below a brief summary 
of the definitions we use in this paper. See 
\cite[\S3]{AnnoLogvinenko-SphericalDGFunctors} for the detailed 
version of the same exposition.  

\begin{defn}
\label{defn-twisted-complex} 
Let $\A$ be a DG-category. A \em twisted complex $(E_i, \alpha_{ij})$ \rm 
over $\A$ is a collection of
\begin{itemize}
\item An object $E_i \in \A$ for each $i \in \mathbb{Z}$. 
\item An element $\alpha_{ij} \in \homm^{i - j +1}_\A(E_i, E_j)$ for
each $i,j \in \mathbb{Z}$. 
\end{itemize}
satisfying:
\begin{itemize}
\item $E_i = 0$ for all but a finite number of $i$, 
\item $(-1)^j d\alpha_{ij} + \sum_k \alpha_{kj} \circ \alpha_{ik} =
0$. 
\end{itemize}
\end{defn} 

A twisted complex is called \em one-sided \rm if $\alpha_{ij} = 0$ for all $i
\geq j$. 

\begin{defn}
The \em DG-category of twisted complexes \rm $\twcx(\A)$ 
has as objects all twisted complexes over $\A$. The complex
of morphisms 
$$ \homm_{\twcx(\A)}\bigl( (E_i, \alpha_{ij}), (F_i, \beta_{ij}) \bigr), $$
has as its degree $p$ part 
$$ \bigoplus_{k,l \in \mathbb{Z}} \homm^{p - l + k}_{\A}(E_k, F_l) $$
and for each $\gamma \in \homm^{p-l+k}_{\A}(E_k, F_l)$ we have
$$ d\gamma = (-1)^l d_{\A} \gamma + \sum_{m \in
\mathbb{Z}}\left( \beta_{lm} \circ \gamma - (-1)^{p} \gamma
\circ \alpha_{mk} \right), $$ 
where $d_\A$ is the differential on morphisms in $\A$. 
\end{defn}

\begin{defn}
\label{defn-convolution-functor}
Let $\A$ be a DG-category. We define the \em convolution functor \rm
\begin{equation*}
\conv\colon \twcx(\A) \rightarrow \modA 
\end{equation*}
as follows. On objects, for any $(E_i, \alpha_{ij}) \in \twcx(\A)$ 
its \em convolution \rm $\conv(E_i, \alpha_{ij})$
is obtained by taking the $\A$-module $\bigoplus_i E_i[-i]$, 
where we use Yoneda embedding to embed $E_i$ into $\modA$, and 
equipping it with the new differential 
$d_{old} + \sum_{i,j} \alpha_{ij}$, where $d_{old}$ denotes
the natural differential on $\bigoplus_i E_i[-i]$. 
For brevity, we use curly brackets to denote the convolution,
e.g. $\bigl\{ E_i, \alpha_{ij} \bigr\}$. 

On morphisms, for any
$\gamma \colon  (E_i, \alpha_{ij}) \rightarrow  (F_i, \beta_{ij}) $
its \em convolution \rm 
$$ \conv(\gamma) \colon 
\bigl\{ E_i, \alpha_{ij} \bigr\} \rightarrow 
\bigl\{ F_i, \alpha_{ij} \bigr\} $$
is the $\A$-module morphism whose morphism of 
the underlying graded modules 
$$ \bigoplus_{k \in \mathbb{Z}} E_k[-k] \rightarrow \bigoplus_{l \in
\mathbb{Z}} F_l[-l] $$
has as the components $E_k[-k] \rightarrow F_l[-l]$ 
the corresponding components of $\gamma \in
\bigoplus_{k,l \in \mathbb{Z}} \homm^{\bullet}_{\A}(E_k, F_l)$. 
\end{defn}

\subsection{Pretriangulated categories}

Let $\A, \B$ be DG-categories.  A functor $F\colon \A \rightarrow \B$ 
is a \em quasi-equivalence \rm if 
it acts by quasi-isomorphisms on morphism complexes and
if $H^0(F)$ is an equivalence of categories. 

The DG-category $\A$ is \em pretriangulated \rm if 
$H^0(\A) \subset H^0(\modA)$ is a triangulated subcategory. 
The \em pretriangulated hull $\pretriag(\A)$ \rm of $\A$ is the
the full subcategory of $\twcx(\A)$ supported at one-sided 
twisted complexes. The convolution functor 
$\conv\colon \pretriag(\A) \rightarrow \modA$ 
of Defn.~\ref{defn-convolution-functor} is fully faithful, 
and its image descends to the triangulated hull of 
$H^0(\A)$ in $H^0(\modA)$. 

Thus $\A$ is pretriangulated if and only if 
$\A \hookrightarrow \pretriag(\A)$ is a quasi-equivalence. 
It is \em strongly pretriangulated \rm if 
$\A \hookrightarrow \pretriag(\A)$ is an equivalence. 
In other words, if it has a quasi-inverse 
$T\colon \pretriag(\A) \rightarrow \A$. 
Then 
$$\pretriag(\A) \xrightarrow{T} \A \hookrightarrow \modA$$
is isomorphic to the convolution functor. By abuse of notation, 
we also refer to $T$ as the \em convolution functor.\rm

It is one of the main results of 
\cite{BondalKapranov-EnhancedTriangulatedCategories}
that the category $\pretriag(\A)$ is strongly pretriangulated. 
For any strongly pretriangulated $\C$, 
the category $\DGFun(\A,\C)$ is strongly pretriangulated since 
convolutions can be defined levelwise in $\C$. 
In particular, $\modA$ is strongly pretriangulated since $\modk$ is. 
Finally, a full subcategory of a strongly pretriangulated DG-category
is strongly pretriangulated if it is pretriangulated and closed under
homotopy equivalences. Thus $\hprojA$ and $\prfhprA$ are 
strongly pretriangulated subcategories of $\modA$, and 
$\Aprfhpr$ and $\Bprfhpr$ are strongly preriangulated 
subcategories of $\AmodB$.  

\subsection{Postnikov systems}
\label{section-zolotom-po-mramoru}

A limited version of this notion appears in e.g.
\cite[\S1.3]{Orlov-EquivalencesOfDerivedCategoriesAndK3Surfaces}
and \cite[\S{IV.2}]{GelfandManin-MethodsOfHomologicalAlgebra}. 
Roughly, it is the data of computing a convolution of a complex of 
objects in a triangulated category. We give a brief summary below. 

Let $\mathcal{T}$ be a triangulated category. Let $(E_i, f_i)$ be 
a complex of objects in $T$ of length $n$:
$$ E_1 \xrightarrow{f_1} E_2 \xrightarrow{f_2} E_3 \rightarrow \dots
\rightarrow E_{n-1} \xrightarrow{f_{n-1}} E_n $$
with all $f_i \circ f_{i-1} = 0$. 

\begin{defn}
\label{defn-postnikov-system}
A \em Postnikov system \rm on the complex $(E_i,f_i)$ is the set of data 
defined as follows.  

For $n = 1$, i.e. the complex consists of a single 
object $E_1$, a Postnikov system on it is an empty set of data.  

For any $n > 1$, the data of a Postnikov system on $(E_i,f_i)$ consists of:
\begin{itemize}
\item A choice of $ i \in \left\{1, \dots, n-1\right\}$. 
\item For the morphism $f_i$, its completion to an exact triangle in
$\mathcal{T}$:
\begin{equation}
\begin{tikzcd}
E_i 
\ar{rr}{f_i}
& 
& 
E_{i+1}. 
\ar{dl}{g_i}
\\
&
E_{i,i+1}
\ar[dashed]{ul}{h_i}
\end{tikzcd}
\end{equation}
Here, dashed and dotted arrows denote morphisms of degree $1$
and $-1$ respectively. 
\item For those of the morphisms $f_{i-1}$ and $f_{i+1}$ which exist,
their lifts to morphisms $f_{i-1,i+1}$ and $f_{i,i+2}$ to 
$E_{i,i+1}[-1]$ and from $E_{i,i+1}$, respectively, with 
$f_{i,i+2} \circ f_{i-1,i+1} = 0$:
\begin{equation}
\begin{tikzcd}
E_{i-1}
\ar{rr}{f_{i-1}}
\ar[dotted]{drrr}[description]{f_{i-1,i+1}}
&
&
E_i 
\ar{rr}{f_i}
& 
& 
E_{i+1}
\ar{rr}{f_{i+1}}
\ar{dl}[description]{g_i}
&
&
E_{i+2}. 
\\
&
&
&
E_{i,i+1}
\ar[dashed]{ul}[description]{h_i}
\ar{urrr}[description]{f_{i,i+2}}
&
&
&
\end{tikzcd}
\end{equation}
\item A Postnikov system for the resulting length $n-1$ complex:
\begin{equation}
\label{eqn-new-length-n-1-complex-in-postnikov-system}
E_1[1] \xrightarrow{f_1} E_2[1] \xrightarrow{f_2} E_3[1] \rightarrow \dots
\rightarrow E_{i-1}[1] \xrightarrow{f_{i-1,i+1}} E_{i,i+1} 
\xrightarrow{f_{i,i+2}} E_{i+2} \rightarrow \dots
\rightarrow E_{n-1} \xrightarrow{f_{n-1}} E_n. 
\end{equation}
\end{itemize}
\end{defn}

In other words, the data of a Postnikov system may be viewed
as the data of an iterative process. First we choose any two 
consecutive objects of a complex and replace them by the cone of the
differential between them. Then, we lift the neighbouring
differentials so that they compose to zero and we thus obtain a complex again. 
Then we repeat, with the length of the complex decreasing by one 
at each iteration. 

Postnikov systems are not unique and, as it may not always be possible 
to lift the neighbouring differentials so that they compose to zero, 
Postnikov systems do not necessarily exist. When they do, they compute 
convolutions of $(E_i, f_i)$:

\begin{defn}
The \em convolution \rm of a Postnikov system on the complex $(E_i,
f_i)$ as defined in Defn.~\ref{defn-postnikov-system} is the
$E_1$ if $n = 1$ and the convolution of the Postnikov system 
\eqref{eqn-new-length-n-1-complex-in-postnikov-system} if $n > 1$. 
\end{defn}

That is, it is the single object left
in the end of the iterative process described by the Postnikov system. 

\begin{exmpl}
Two of the $6$ possible types of Postnikov systems on 
a length $4$ complex $E_1 \rightarrow E_2 \rightarrow E_3 \rightarrow E_4$:
\begin{equation}
\label{eqn-example-length-4-postnikov-system}
\begin{tikzcd}[column sep={1.0825cm,between origins}, row sep={1.25cm,between origins}] 
E_1
\ar{rr}{f_1}
\ar[dotted]{ddrrrr}[description]{f_{14}}
&&
E_2
\ar{rr}{f_2}
\ar[dotted]{drrr}[description]{f_{24}}
\ar[phantom]{dddr}[description, pos=0.6]{\star}
&&
E_3
\ar{rr}{f_3}
\ar[phantom]{dd}[description, pos=0.6]{\star}
&&
E_4
\ar{ld}{g_3}
\arrow[to=N24,phantom]{}[description, pos=0.275]{\star}
\\
&&
& |[alias=N24]|
&
& 
E_{3,4}
\ar[dashed]{lu}[description]{h_3}
\ar[dotted]{ld}{g_{24}}
&
\\
&&
&&
E_{2,3,4}
\ar[dashed]{lluu}[description]{h_{24}}
\ar[dotted]{ld}{g_{14}}
&&
\\
&&
&
E_{1,2,3,4}
\ar[dashed]{llluuu}{h_{14}}
&
&&
\end{tikzcd}
\begin{tikzcd}[column sep={1.0825cm,between origins}, row sep={1.25cm,between origins}] 
E_1
\ar{rr}{f_1}
&&
E_2
\ar{rr}{f_2}
\ar{ld}{g_1}
\ar[dotted]{drrr}[description]{f_{24}}
&&
E_3
\ar{rr}{f_3}
&&
E_4.
\ar{ld}{g_3}
\arrow[to=N24,phantom]{}[description, pos=0.275]{\star}
\\
&
E_{1,2}
\ar[dotted]{rrrr}{f_{14}}
\ar[dashed]{lu}[description]{h_1}
&
& |[alias=N24]| 
\arrow[ulll,phantom]{}[description,pos=0.725]{\star}
\arrow[dd,phantom]{}[description,pos=0.333]{\star}
&
& 
E_{3,4}
\ar[dashed]{lu}[description]{h_3}
\ar[dotted]{lldd}{g_{14}}
&
\\
&&
&&
&&
\\
&&
&
E_{1,2,3,4}
\ar[dashed]{lluu}{h_{14}}
&
&&
\end{tikzcd}
\end{equation}
Here, dashed and dotted arrows denote morphisms of degree $>0$
and $<1$, respectively, the triangles denoted by $\star$ are exact,
and the remaining triangles are commutative. 
\end{exmpl}

Suppose now we have a strongly pretriangulated DG-enhancement $\A$ 
of $\T$, i.e. a strongly pretriangulated $\A$ and a choice
of isomorphism $\iota\colon H^0(\A) \simeq \T$. 

\begin{defn}
Let $(E_i, f_i)$ be a length $n$ complex of objects in $\T$.  
A twisted complex $(\bar{E}_i, \bar{f}_{ij})$ over $\A$ is
a \em lift \rm of $(E_i, f_i)$ if $E_{i} = \iota(\bar{E}_{i-n})$
and $f_i = \iota(\bar{f}_{i-n,i-n+1}) $ for all 
$i \in \left\{1, \dots, n-1\right\}$. 

Conversely, for any twisted complex $(\bar{E}_i, \bar{f}_{ij})$ over $\A$
with $\bar{E}_i \neq 0$ only for $i \in \left\{-n+1, \dots, 0\right\}$, 
denote by $\iota(\bar{E}_i, \bar{f}_{ij})$ the complex 
$$ \iota(\bar{E}_{-n+1}) 
\xrightarrow{[\bar{f}_{-n+1,-n+2}]} 
\iota(\bar{E}_{-n+1})
\rightarrow 
\dots 
\rightarrow 
\iota(\bar{E}_{-1}) 
\xrightarrow{[\bar{f}_{-1,0}]} 
\iota(\bar{E}_{0}),$$
whose lift $(\bar{E}_i, \bar{f}_{ij})$ is. 
\end{defn}

The index shift in the definition above is introduced to ensure
that a lift of a complex $E_1 \rightarrow E_2 \rightarrow \dots
\rightarrow E_n$ lifts $E_n$ to the degree $0$ element of the twisted
complex. This is so that notions of convolution for 
twisted complexes and for ordinary complexes coincide, 
as per following two results.  

\begin{prps}
Let $\bar{E} = (\bar{E}_i, \bar{f}_{ij})$ be a twisted complex over $\A$. 
For any $k \in \mathbb{Z}$ there is a unique twisted complex
we denote by $\bar{E}(\hat{k})$ which has the same convolution as $\bar{E}$
and the object set:
$$
\bar{E}(\hat{k})_i = 
\begin{cases}
\bar{E}_{i}, & i > k + 1 \\
\left\{ \bar{E}_{i-1} \xrightarrow{\bar{f}_{i-1,i}}
\underset{\degzero}{\bar{E}_{i+1}} \right\}, & i = k+1, \\
\bar{E}_{i-1}[1], & i < k + 1 
\end{cases}
$$
\end{prps}
\begin{proof}
By the definition of the convolution functor, the convolution 
of $\bar{E}$ in $\modA$ is the $\A$-module $\bigoplus \bar{E}_i[-i]$ 
whose differential $d_{\bar{E}}$ is modified by adding to it the sum 
of the twisted differentials $\bar{f}_{ij}$. Conversely, the difference 
between $d_{\bar{E}}$ and the natural differential $\bigoplus \bar{E}_i[-i]$
is an endomorphism of the underlying graded module which decomposes uniquely 
into the set of twisted differentials $\bar{f}_{ij}$. 

Since $\bigoplus_i \bar{E}(\hat{k})_i[-i]$ has the same underlying 
graded module as $\bigoplus \bar{E}_i[-i]$ and $\bar{E}$, it follows
that there exists a unique set of twisted differentials which modifies 
the natural differential of $\bigoplus_i \bar{E}(\hat{k})_i[-i]$
to obtain $d_{\bar{E}}$. 
\end{proof}
\begin{cor}
\label{cor-postnikov-systems-from-twisted-complex}
Let $E = (E_i, f_i)$ be a length $n > 1$ complex of objects in $\T$. 
Let $\bar{E} = (\bar{E}_i, \bar{f}_{ij})$ be a twisted complex over $\A$ 
lifting $(E_i,f_i)$. Then the following is a Postnikov
system for $(E_i,f_i)$:
\begin{itemize}
\item Any choice of $i \in \left\{1, \dots, n-1\right\}$. 
\item The following choices of an exact triangle completing $f_i$
and of the lifts of $f_{i-1}$ and $f_{i+1}$:
\begin{equation*}
\begin{tikzcd}[column sep=1.5cm]
E_{i-1}
\ar{rr}{f_{i-1}}
\ar[dotted]{drrr}[description]{f_{i-1,i+1}}
&
&
E_i 
\ar{rr}{f_i}
& 
& 
E_{i+1}
\ar{rr}{f_{i+1}}
\ar{dl}[description]{g_i}
&
&
E_{i+2}, 
\\
&
&
&
\left\{\bar{E}_{i-n} \xrightarrow{\bar{f}_{i-n,i-n+1}}
\underset{\degzero}{\bar{E}_{i-n+1}} \right\}
\ar[dashed]{ul}[description]{h_i}
\ar{urrr}[description]{f_{i,i+2}}
&
&
&
\end{tikzcd}
\end{equation*} 
where $g_i$ and $h_i$ are the images
under $\iota$ of the convolutions of the tautological morphisms
from $\bar{E}_{i-n+1}$ to 
$\bar{E}_{i-n} \xrightarrow{\bar{f}_{i-n,i-n+1}}
\underset{\degzero}{\bar{E}_{i-n+1}}$ and from the latter
to $\bar{E}_{i-n}$ given by the identity maps, while $f_{i-1,i+1}$
and $f_{i, i+2}$ are the images under $\iota$ of the corresponding
differentials in the twisted complex $\bar{E}(\widehat{i-n})$. 
\item Any choice of a Postnikov system on the resulting 
complex $\iota\left(\bar{E}(\widehat{i-n})\right)$. 
\end{itemize}
\end{cor}

By applying Corollary \ref{cor-postnikov-systems-from-twisted-complex}
first to $E = \iota(\bar{E})$, then to $\iota(\bar{E}(\widehat{i-n}))$, 
and so on, one can produce for any permutation of $1, \dots, n$
a Postnikov system on $E$ whose convolution would be
the image under $\iota$ of the convolution of $\bar{E}$. 

Even if the DG-enhancement $\A$ of $\mathcal{T}$ is not strongly
pretriangulated, then $\pretriag \A$ is a strongly pretriangulated
enhancement of $\mathcal{T}$. Any twisted complex $\bar{E}$ over $\A$ can 
be viewed as a twisted complex over $\pretriag \A$, and thus 
Corollary \ref{cor-postnikov-systems-from-twisted-complex}
applies to produce Postnikov systems on the complex
$\iota(\bar{E})$ in $\mathcal{T}$. 

\subsection{Rectangle and Extraction lemmas}

We also need the two following useful technical facts. 
Let $\A$ be a DG-category. All twisted complexes in this section 
are considered to be over $\A$. We say that a map $(f_{ij})$ of 
twisted complexes is \em one-sided \rm if $f_{ij} = 0$ for any $j < i$. 

\begin{lemma}[Rectangle Lemma]
\label{lemma-rectangle-lemma}
Let $E=(E_i, \alpha_{ij})$ and $F=(F_i,\beta_{ij})$ 
be one-sided twisted complexes. 
Let $f=(f_{ij})$ be a one-sided closed map $E \rightarrow F$ of degree $0$.
 
There exists a twisted complex $G=(G_i, \gamma_{ij})$ over
$\pretriag \A$ with each 
$$ G_i = \big( E_i \xrightarrow{(-1)^i f_{ii}} \underset{\degzero}{F_i} \big) $$
such that  
$$ \tot \left( G_i, \gamma_{ij} \right) \simeq 
\tot\left( E \xrightarrow{f} \underset{\degzero}{F} \right)
$$
in $\pretriag \A$. 
\end{lemma}

\begin{proof}
Since $f$ is closed of degree $0$ and one-sided, we have
$(df)_{ii}=df_{ii}=0$, and 
\begin{align*}
(df)_{ij}=(-1)^{j}df_{ij}+\sum\limits_{k=i}^{j-1} \beta_{kj} f_{ik} -
\sum\limits_{k=i+1}^j f_{kj} \alpha_{ik} = 0 \quad\quad j > i. 
\end{align*}
Define the twisted differentials $G_i \xrightarrow{\gamma_{ij}} G_j$ by the
following diagram:
\begin{align*}
\xymatrix@C=2cm{
E_i
\ar[r]^{(-1)^i f_{ii}}
\ar[d]_{-\alpha_{ij}}
\ar[rd]^{f_{ij}}
&
\underset{\degzero}{F_i}
\ar[d]^{\beta_{ij}}
\\
E_j
\ar[r]^{(-1)^j f_{jj}}
&
\underset{\degzero}{F_j}.
}
\end{align*}
The degree of this map is $i-j+1$. Note also that 
\begin{align*}
d(\gamma_{ij})=\qquad
\xymatrix@C=8cm{
E_i
\ar[r]^{(-1)^i f_{ii}}
\ar[d]_{d\alpha_{ij}}
\ar[rd]^<<<<<<<<<<<<<<<<<<<<<<<<<<<<<<{\quad\quad\quad\quad\quad\quad\quad\quad\quad\quad df_{ij}-(-1)^jf_{jj}\alpha_{ij}-(-1)^{i-j+1}(-1)^i\beta_{ij}f_{ii}}
&
F_i
\ar[d]^{d\beta_{ij}}
\\
E_j
\ar[r]^{(-1)^j f_{jj}}
&
F_j.
}
\end{align*}
We claim that $G=(G_i, \gamma_{ij})$ is a twisted complex. For this we need 
to have for all $i$ and $j$ 
\begin{align*}
(-1)^j d\gamma_{ij} + \sum\limits_{k=i+1}^{j-1} \gamma_{kj}\gamma_{ik} = 0. 
\end{align*}
The map on the LHS has three components: $E_i \rightarrow E_j$, 
$F_i \rightarrow F_j$ and $E_i \rightarrow F_j$.  
The first two components vanish since $E$ and $F$ are twisted complexes. 
Computing the $E_i \rightarrow F_j$ 
component we get
\begin{multline*}
(-1)^j \left( df_{ij}-(-1)^jf_{jj}\alpha_{ij}+(-1)^{j}\beta_{ij}f_{ii} \right)
+ \sum\limits_{k=i+1}^{j-1} \beta_{kj} f_{ik} - \sum\limits_{k=i+1}^{j-1} f_{kj} \alpha_{ik}=\\
= (-1)^j df_{ij} +\beta_{ij}f_{ii} + \sum\limits_{k=i+1}^{j-1} \beta_{kj} f_{ik} - f_{jj}\alpha_{ij} - \sum\limits_{k=i+1}^{j-1} f_{kj} \alpha_{ik} = (df)_{ij}=0.
\end{multline*}

Thus $(G_i, \gamma_{ij})$ is a twisted complex. 
Its total complex and that of $\left(E \xrightarrow{f} F\right)$ 
both have the $i$-th term 
$$ E_{i+1} \oplus  F_i$$ 
and the $ij$-th differential 
$$ - \alpha_{i+1,j+1} +  f_{i+1,j} + \beta_{ij}, $$
as desired.
\end{proof}

Recall \cite[\S3.2]{AnnoLogvinenko-SphericalDGFunctors} 
\cite[\S1]{BondalKapranov-EnhancedTriangulatedCategories} 
that there exists the convolution functor $\pretriag \modA \rightarrow \modA$
which is a category equivalence and which we denote by $\big\{-\big\}$. 

\begin{cor}
\label{cor-the-cone-is-convolution-of-cones-of-vertical-arrows}
Let $\A$ be a DG-category and let $E = (E_i, \alpha_{ij})$ and 
$F = (F_i,\beta_{ij})$ be one-sided twisted complexes
over $\modA$ and let $f=(f_{ij})$ be a one-sided closed map 
$E \rightarrow F$ of degree $0$. 

The cone of the induced map 
$$\left\{E\right\} \xrightarrow{f} \left\{F \right\}$$
is isomorphic to the convolution of a twisted complex whose
objects are isomorphic to 
\begin{align}
\label{eqn-the-cone-of-the-convolutions-of-twisted-complexes}
\cone\left(E_i \xrightarrow{f_{ii}} F_i\right).
\end{align}
\end{cor}
\begin{proof}
The convolution of the twisted complex 
$\tot \left(G_i, \gamma_{ij}\right)$ constructed in 
Lemma \ref{lemma-rectangle-lemma} is isomorphic to the convolution of the twisted complex $\left( \{G_i\}, \gamma_{ij}\right)$, cf.~diagram 
3.1 in \cite[\S3]{BondalKapranov-EnhancedTriangulatedCategories}. Each $\left\{G_i\right\}$ is isomorphic to  $\cone\left(E_i \xrightarrow{f_{ii}} F_i\right)$. Similarly, the convolution of $\tot\left( E \xrightarrow{f} F \right)$ is 
isomorphic to $\cone\left( \{E\} \xrightarrow{f} \{F\}\right)$. 
The claim now follows from Lemma \ref{lemma-rectangle-lemma}. 
\end{proof}

\begin{cor}
\label{cor-every-vertical-comp-is-quasiiso-then-map-is-itself-quasiiso}
Let $\A$ be a DG-category and let $(E_i, \alpha_{ij})$ and 
$(F_i,\beta_{ij})$ be one-sided bounded above twisted complexes
over $\modA$. Let $f=(f_{ij})$ be a one-sided closed map 
$(E_i, \alpha_{ij}) \rightarrow (F_j, \beta_{ij})$ of degree $0$. 

If each component $f_{ii}: E_i \to F_i$ is 
a quasi-isomorphism, then so is the induced map 
$$\left\{E_i, \alpha_{ij}\right\} \xrightarrow{f} \left\{F_i,
\beta_{ij}\right\}.$$
\end{cor}
\begin{proof}
By Cor.~\ref{cor-the-cone-is-convolution-of-cones-of-vertical-arrows}
the cone of $f$ is isomorphic in $D(\A)$ to 
the convolution of the bounded above twisted complex whose objects are 
isomorphic to $\cone\left(E_i \xrightarrow{f_{ii}} F_i\right)$ in $D(\A)$. By
assumption $f_{ii}$ are quasi-isomorphisms, thus all the objects of 
this twisted complex are acyclic. To show that $f$ itself is a
quasi-isomoprhism, it remains to show that the convolution
of a bounded above twisted complex of acyclic modules is itself acyclic.
To see this, first note that the cone of any two acyclic modules is acyclic and, by induction, so is the convolution of any finite one-sided twisted complex of acyclic modules. Finally, the convolution of any bounded above twisted complex has an exhaustive filtration by the convolutions of its finite subcomplexes. The induced exhaustive filtration on cohomologies  
proves the claim. 
\end{proof}

\begin{lemma}[Extraction Lemma]
\label{lemma-extraction-lemma} 

Let $E \in \modA$ and let $Q \subset E$ be a null-homotopic
submodule. Let $\nu$ be a contracting homotopy of $Q$, that is
$\nu \in \homm^{-1}_\A(Q,Q)$ such that $d\nu = \id$. 
Suppose that on the level of the underlying
graded modules $E$ splits as $Q \oplus F$ for some graded $\A$-submodule $F$. 
Let the differential of $E$ with respect to that splitting be
\begin{equation}
\label{eqn-differential-of-the-module-with-an-acyclic-submodule}
\left(
\begin{matrix}
d_Q & \alpha
\\
\beta & \delta 
\end{matrix}
\right).
\end{equation}

Then $E$ is homotopy equivalent to $F$ equipped with the differential 
$d_F = \delta - \beta \circ \nu \circ \alpha$. 
\end{lemma}
\begin{proof}
Since $Q$ is a sub-DG-module of $E$, we have $d_Q^2 = 0$.
Since \eqref{eqn-differential-of-the-module-with-an-acyclic-submodule}
is a differential, it is a derivation and of square zero.
The fact \eqref{eqn-differential-of-the-module-with-an-acyclic-submodule} is 
a derivation together with the splitting of $E = Q \oplus F$
respecting $\A$-action implies that $\delta$ is also a derivation, 
while $\alpha$ and $\beta$ are maps of graded $\A$-modules. 
The fact that \eqref{eqn-differential-of-the-module-with-an-acyclic-submodule}
squares to zero implies that 
$$ \alpha \circ \beta = 0, $$
$$ \delta^2 + \beta \circ \alpha = 0, $$
$$ \delta \circ \beta + \beta \circ d_Q = 0, $$
$$ \alpha \circ \delta + d_Q \circ \alpha = 0. $$

The map $d_F$ is a derivation as it is the sum of the derivation $\delta$
and the graded $\A$-module map $- \beta \circ \nu \circ \alpha$. 
Moreover, we have 
\begin{align*}
d_F^2 & = \left(\delta - \beta \circ \nu \circ \alpha \right)^2 = 
\delta^2 - \delta \circ \beta \circ \nu \circ \alpha - 
\beta \circ \nu \circ \alpha \circ \delta + 
\beta \circ \nu \circ \alpha \circ \beta \circ \nu \circ \alpha = \\
& = \delta^2 + \beta \circ d_Q \circ \nu \circ \alpha + \beta \circ \nu
\circ d_Q \circ \alpha + 0 = 
\delta^2 + \beta \circ \left(d_Q \circ \nu + \nu \circ d_Q\right)
\circ \alpha = \\
& = \delta^2 + \beta \circ d\nu \circ \alpha = 
\delta^2 + \beta \circ \alpha  = 0. 
\end{align*}
Thus $d_F$ does indeed define a differential and hence a structure 
of a DG $\A$-module on $F$. It can be readily checked that 
with respect to that structure the maps $\alpha$ and $\beta$ are 
both closed.  

Consider the following maps of graded $\mathcal{A}$-modules:
\begin{align}
\label{eqn-extraction-lemma-E-to-F-homotopy-equivalence}
Q \oplus F 
\xrightarrow{ 
\left(
\begin{smallmatrix}
- \beta \circ \nu & \id 
\end{smallmatrix}
\right)
}
F
\\
F 
\xrightarrow{ 
\left(
\begin{smallmatrix}
-\nu \circ \alpha \\ \id 
\end{smallmatrix}
\right)
}
Q \oplus F
\label{eqn-extraction-lemma-F-to-E-homotopy-equivalence}. 
\end{align}
We claim that they
define mutually inverse homotopy equivalences $E \xrightarrow{\sim} F$
and $F \xrightarrow{\sim} E$ in $\modA$. Indeed, 
$$ 
\eqref{eqn-extraction-lemma-E-to-F-homotopy-equivalence}
\circ 
\eqref{eqn-extraction-lemma-F-to-E-homotopy-equivalence}
= \id + \beta \circ \nu^2 \circ  \alpha = 
\id - d(\beta \circ \nu^3 \circ \alpha), $$
while
$$
\eqref{eqn-extraction-lemma-F-to-E-homotopy-equivalence}
\circ 
\eqref{eqn-extraction-lemma-E-to-F-homotopy-equivalence}
= 
\left(
\begin{matrix}
0 & - \nu \circ \alpha \\
- \beta \circ \nu & \id
\end{matrix}
\right)
= 
\left(
\begin{matrix}
\id & 0 \\
0 & \id
\end{matrix}
\right)
- 
\left(
\begin{matrix}
\id & \nu \circ \alpha \\
\beta \circ \nu & 0 
\end{matrix}
\right)
= \id - 
d 
\left(
\begin{matrix}
\nu & 0 \\
0 & 0  
\end{matrix}
\right)
$$
as required. 
\end{proof}

Since the convolution functor $\pretriag \modA \rightarrow \modA$
is an equivalence, the Extraction Lemma can be applied to any twisted 
complex over $\modA$ with a null-homotopic subcomplex. For example: 

\begin{exmpl}
Let $\A$ be a DG-category and let 
\begin{equation}
\label{eqn-extraction-lemma-example}
\begin{tikzcd}[column sep={3cm},row sep={1.5cm}] 
\underset{\degzero}{E_0}
\ar{r}{
\left(
\begin{smallampmatrix}
\alpha_{01} \\ 
\delta_{01}
\end{smallampmatrix}
\right)
}
\ar[bend left=20,dashed]{rr}{
\left(
\begin{smallampmatrix}
\alpha_{02} \\ 
\delta_{02} 
\end{smallampmatrix}
\right)
}
\ar[bend left=30,dotted]{rrr}{\delta_{03}}
&
X \oplus E_1 
\ar{r}{
\left(
\begin{smallampmatrix}
\id & \alpha_{12} \\
\beta_{12} & \delta_{12}
\end{smallampmatrix}
\right)
}
\ar[bend right=20,dashed]{rr}{
\left(
\begin{smallampmatrix}
\beta_{13} & \delta_{13}
\end{smallampmatrix}
\right)
}
&
X \oplus E_2 
\ar{r}{
\left(
\begin{smallampmatrix}
\beta_{23} & \delta_{23} 
\end{smallampmatrix}
\right)
}
&
E_3
\end{tikzcd}
\end{equation}
be a twisted complex over $\modA$. Then the following is a twisted
complex homotopy equivalent to \eqref{eqn-extraction-lemma-example}: 
\begin{equation}
\label{eqn-extraction-lemma-example-excised}
\begin{tikzcd}[column sep={3cm},row sep={1.5cm}] 
\underset{\degzero}{E_0}
\ar{r}{
\delta_{01}
}
\ar[bend left=20,dashed]{rr}{
\delta_{02} - \beta_{12} \circ \alpha_{02}
}
\ar[bend left=30,dotted]{rrr}{
\delta_{03} - \beta_{13} \circ \alpha_{02}
}
&
E_1 
\ar{r}{
\delta_{12} - \beta_{12} \circ \alpha_{12}
}
\ar[bend right=20,dashed]{rr}{
\delta_{13} - \beta_{13} \circ \alpha_{12}
}
&
E_2 
\ar{r}{
\delta_{23}
}
&
E_3.
\end{tikzcd}
\end{equation}
\end{exmpl}
\begin{proof}
In Lemma \ref{lemma-extraction-lemma} set $Q$ to be the convolution 
of $\left( \underset{\degone}{X} \xrightarrow{\id} X \right)$ and 
set $F$ to be the graded module underlying $\bigoplus E_i[-i]$. 
Set the contracting homotopy $\nu$ to be given by
\begin{equation}
\begin{tikzcd}
\underset{\degone}{X} 
\ar{r}{\id}
& X 
\ar{dl}[description]{\id}
\\
\underset{\degone}{X} 
\ar{r}{\id}
& 
X.
\end{tikzcd}
\end{equation}

Then in 
\eqref{eqn-differential-of-the-module-with-an-acyclic-submodule}
the map $\delta$ is $\sum d_{E_i} + \sum \delta_{ij}$, $\alpha$ is $\sum
\alpha_{ij}$ and $\beta$ is $\sum \beta_{ij}$. The map $\beta \circ
\nu \circ \alpha$ is therefore 
\begin{equation*}
\begin{tikzcd}[column sep={3cm},row sep={0.5cm}] 
\underset{\degzero}{E_0}
\ar{dr}[description]{\alpha_{01}}
\ar{drr}[description]{\alpha_{02}}
&
E_1 
\ar{dr}[description]{\alpha_{12}}
&
E_2 
&
E_3
\\
& X & X \ar{dl}[description]{\id} & 
\\
& 
X 
\ar{dr}[description]{\beta_{12}} 
\ar{drr}[description]{\beta_{13}} 
& 
X \ar{dr}[description]{\beta_{23}} & 
\\
\underset{\degzero}{E_0} & E_1 & E_2 & E_3,
\end{tikzcd}
\end{equation*}
that is 
$ (\beta_{12} + \beta_{13}) \circ (\alpha_{02} + \alpha_{12})$. 
Thus the differential $\delta - \beta \circ \nu \circ \alpha$ on the graded 
module $\bigoplus E_i[-i]$ equals
\begin{align}
\label{eqn-excised-complex-differential} 
\sum d_{E_i} + \sum \delta_{ij} - 
(\beta_{12} + \beta_{13}) \circ (\alpha_{02} + \alpha_{12}). 
\end{align}
The convolution functor $\pretriag \modA \rightarrow \modA$
is an equivalence. Every differential on the graded 
complex $\bigoplus E_i[-i]$ corresponds to a twisted complex 
whose objects are $E_i$. By the definition of the convolution 
functor, the differential \eqref{eqn-excised-complex-differential}
corresponds to the twisted complex 
\eqref{eqn-extraction-lemma-example-excised}. 
\end{proof}

\subsection{$\Ainfty$-algebras, modules, and bimodules}
\label{section-Ainfty-algebras-modules-bimodules}

For an introduction to $\Ainfty$-categories we recommend
\cite{Keller-AInfinityAlgebrasModulesAndFunctorCategories}, and 
for a comprehensive technical text -- \cite{Lefevre-SurLesAInftyCategories}.
We refer the reader to the latter for the definitions and the notation 
we employ. Below, we summarise some of it and prove several minor 
new results. 

An \em $\Ainfty$-algebra \rm over $k$ is a graded $k$-bimodule $A$ 
together with graded maps
\begin{align}
\label{eqn-A-infty-algebra-maps}
m_i\colon A^{\otimes i} \rightarrow A \quad \quad i \geq 1
\end{align}
of degree $2 - i$ which lift to a differential which
gives the structure of an coaugmented DG coalgebra to 
the coaugmented tensor coalgebra 
$$ T^c(A[1]) = \bigoplus_{n \geq 0} A^{\otimes n}[n]$$ 
generated by $A[1]$. Such differential is necessarily
a sum of the zero map on the coaugmented part $k$ 
and a differential making
the reduced tensor coalgebra $\bigoplus_{n \geq 1} A^{\otimes n}[n]$
into a DG coalgebra. 
The resulting coaugmented DG-coalgebra $\bigoplus_{n \geq 0} A^{\otimes n}[n]$
is the \em bar construction $\infbar A $ \rm of $A$, whose counit we
denote by $\tau\colon \infbar A \rightarrow k$. 
The resulting DG-coalgebra $\bigoplus_{n \geq 1} A^{\otimes n}[n]$
is the \em non-augmented bar construction \rm $\infbarnaug A$. 

Let $A$ and $B$ be $\Ainfty$-algebras. An \em $\Ainfty$-algebra morphism \rm 
$f\colon A \rightarrow B$ is a collection of graded maps 
\begin{align}
\label{eqn-A-infty-morphism-maps}
f_i\colon A^{\otimes i} \rightarrow B \quad \quad i \geq 1
\end{align}
of degree $1 - i$ which lift to a morphism of augmented coalgebras
$\infbar A \rightarrow \infbar B$.

Let $A$ be an $\Ainfty$-algebra. A \em (right) $A$-module \rm is 
a graded $k$-module $E$ together with a collection of graded maps
\begin{align}
\label{eqn-A-infty-module-maps}
m_i\colon E \otimes A^{\otimes i - 1} \rightarrow E \quad \quad i \geq 1
\end{align}
of degree $2-i$ which lift to a differential on the free graded
$\infbar A$-comodule $E[1] \otimes_k \infbar A$ generated by $E[1]$. 
The resulting DG $\infbar A$-comodule is the \em bar construction 
$\infbar E$ \rm of $E$.  

Let $E$ and $F$ be two $A$-modules. An \em $\Ainfty$-module morphism \rm
$f\colon E \rightarrow F$ is a collection of graded maps 
\begin{align}
\label{eqn-A-infty-module-morphism-maps}
f_i \colon E \otimes_k A^{\otimes i - 1} \rightarrow F \quad \quad i \geq 1
\end{align}
of degree $1 - i$ which lift to a $\infbar A$-comodule morphism
$\infbar E \rightarrow \infbar F$. An $\Ainfty$-module morphism 
is \em strict \rm if $f_i = 0$ for $i \geq 2$. 

Let $A$ and $B$ be $\Ainfty$-algebras. An \em $A$-$B$-bimodule \rm is 
a graded $k$-bimodule $M$
together with a differential $m_{0,0}$ and the action maps
\begin{align}
m_{i,j}\colon A^{\otimes i} \otimes_k M \otimes_k B^{\otimes j} \rightarrow M  
\quad \quad i \geq 0, j \geq 0
\end{align}
of degree $1 - i - j$ which together lift to a differential 
\begin{align}
(\infbar A) \otimes_k M[1] \otimes_k (\infbar B) 
\longrightarrow 
(\infbar A) \otimes_k M[1] \otimes_k (\infbar B),
\end{align}
cf. \cite[\S2.5.1]{Lefevre-SurLesAInftyCategories}. 
The resulting DG $\infbar A$-$\infbar B$-bicomodule is
the \em bar construction $\infbar M$ \rm of $M$. Whenever it  
is necessary to avoid confusion, e.g. in the case of the diagonal
bimodule, we will denote the bimodule bar construction 
as $\infbarbim M$. 

We also consider the partial bar constructions. 
The \em $B$-bar construction $B^B_\infty M$ \rm is the 
right DG $(\infbar B)$-comodule 
whose underlying graded $\infbar B$-comodule is $M \otimes_k \infbar B$ 
and whose differential is constructed from $m_{0,j}$ for 
$j \geq 0$. It also carries the structure of a left $\Ainfty$
$A$-module, defined by $m_{i,j}$ for $i \geq 1, j \geq 0$. 
Likewise, the \em $A$-bar construction $B^A_\infty M$ \rm is 
the left DG $(\infbar A)$-comodule and right $\Ainfty$ $B$-module 
defined similarly. Consider now the DG-algebras $\eend_{\infbar B}(B^B_\infty M)$
and $\eend_{(\infbar A)^\opp}(B^A_\infty M)$.
Specifying the structure of $A$-$B$-bimodule on $M$ is 
equivalent to specifying the natural \em $A$-action \rm $\Ainfty$-morphism
\begin{align}
\label{eqn-A-infty-bimodule-left-action-map}
A \xrightarrow{\action_M} \eend_{\infbar B}(B^B_\infty M).
\end{align} 
Similarly, it is equivalent to specifying the natural \em $B$-action 
$\Ainfty$-morphism \rm
\begin{align}
\label{eqn-A-infty-bimodule-right-action-map}
B^{\opp} \xrightarrow{\action_M} \eend_{(\infbar A)^\opp}(B^A_\infty M), 
\end{align}
cf. the ``lemme clef'' of \cite[\S5.3]{Lefevre-SurLesAInftyCategories}. 
Explicitly, we define e.g. the $A$-action morphism by setting
for each $a_1 \otimes  \dots \otimes a_n \in A^{\otimes n}$ 
the endomorphism 
$\action_M(a_1 \otimes \dots \otimes a_n) \in \eend_{\infbar B}(B^B_\infty M)$ 
to be
\begin{align}
m \otimes b_1 \otimes \dots \otimes b_m \mapsto
\sum_{l = 0}^m (-1)^? 
m_{n,l}(a_1 \otimes \dots \otimes a_n 
\otimes m 
\otimes b_1 \otimes \dots \otimes b_l) \otimes
b_{l+1} \otimes \dots \otimes b_m
\end{align}
where the signs are dictated by 
the definitions in \cite[\S5.3]{Lefevre-SurLesAInftyCategories}. 

The \em diagonal bimodule \rm $A$ is defined by the graded maps
\begin{align}
\label{eqn-A-infty-diagonal-bimodule}
m_{i,j}\colon A^{\otimes i} \otimes_k A \otimes_k A^{\otimes j} 
\xrightarrow{m_{i+j+1}} A \quad \quad i \geq 0, j \geq 0
\end{align}
where $m_i$ are the maps which define the $\Ainfty$-algebra structure on $A$. 

Let $M$ and $N$ be two $A$-$B$-bimodules. An \em $\Ainfty$-bimodule 
morphism \rm $f\colon M \rightarrow N$ is a collection of graded maps 
\begin{align}
\label{eqn-A-infty-bimodule-morphism-maps}
f_{i,j} \colon A^{\otimes i} \otimes_k M \otimes_k B^{\otimes j}
\rightarrow N \quad \quad i \geq 0, j \geq 0
\end{align}
of degree $- i - j$ which lift to a $\infbar A$-$\infbar B$-bicomodule 
morphism $\infbar M \rightarrow \infbar N$. An $\Ainfty$-module morphism 
is \em strict \rm if $f_{i,j} = 0$ for $i \geq 1$ or $j \geq 1$. 

The DG $k$-module $\homm_{\infbar B}(B^B_\infty M, B^B_\infty N)$ has
a natural structure of an $A$-$A$-bimodule defined via the $A$-action maps 
for $B^B_\infty M$ and $B^B_\infty N$. Similar to
\eqref{eqn-A-infty-bimodule-left-action-map} and 
\eqref{eqn-A-infty-bimodule-right-action-map}, specifying 
an $\Ainfty$ $A$-$B$-bimodule morphism $M \xrightarrow{f} N$ is then 
equivalent to specifying a DG $\infbar A$-$\infbar A$-bicomodule morphism 
\begin{align}
\label{eqn-A-infty-map-of-bimodules-as-left-action-map}
\infbar A \xrightarrow{f_A}
\infbar \Big( \homm_{\infbar B}\left(B^B_\infty M, B^B_\infty N\right)\Big).
\end{align}
Similarly, it is equivalent to specifying a DG
$\infbar B$-$\infbar B$-bicomodule morphism 
\begin{align}
\label{eqn-A-infty-map-of-bimodules-as-right-action-map}
\infbar B \xrightarrow{f_B}
\infbar \Big(
\homm_{\infbar A}\left(B^A_\infty M, B^A_\infty N\right)
\Big).
\end{align}

\subsection{$\Ainfty$-categories}

Let $\mathbb{A}$ be a set. We define $k_\mathbb{A}$ to be the category 
whose set of objects is $\mathbb{A}$ and whose morphisms spaces are
\begin{align}
\homm_{k_\mathbb{A}}(a,b) = 
\begin{cases}
k & \quad \text{ if } a = b \\
0 & \quad \text{ if } a \neq b. 
\end{cases}
\end{align}
For any sets $\mathbb{A}$, $\mathbb{B}$ and $\mathbb{C}$, 
any graded $k_\mathbb{A}$-$k_{\mathbb{B}}$-bimodule $M$, and 
$k_\mathbb{B}$-$k_\mathbb{C}$-bimodule $N$
we denote by $\otimes_k$ their tensor product over 
$k_\mathbb{B}$:
\begin{align}
\leftidx{_a}(M \otimes_k N)_c  =  \bigoplus_{b \in \mathbb{B}}
\leftidx{_a}M_b \otimes_k  \leftidx{_b}N_c.
\end{align}
We give the category of graded $k_\mathbb{A}$-$k_\mathbb{A}$-bimodules 
a monoidal structure by equipping it with the multiplication given 
by $\otimes_k$ and the identity element given by the diagonal bimodule
$k_\mathbb{A}$. 

Given a graded $k_\mathbb{A}$-$k_\mathbb{B}$-bimodule $M$ and 
two maps of sets $\mathbb{A}' \xrightarrow{f} \mathbb{A}$ and
$\mathbb{B}' \xrightarrow{g} \mathbb{B}$, we write $\leftidx{_f}M_g$
for the graded $k_{\mathbb{A}'}$-$k_{\mathbb{B}'}$-bimodule obtained 
by pulling back along $f$ and $g$, i.e. 
\begin{align}
\leftidx{_{a'}}{\bigl(\leftidx{_f}M_g\bigr)}_{b'} = 
\leftidx{_{f(a')}}M_{g(b')} 
\quad\quad\quad
\forall\; a' \in \mathbb{A}', b' \in \mathbb{B}'.
\end{align}

An \em $\Ainfty$-category $\A$ \rm is an object set $\mathbb{A}$ and 
an $\Ainfty$-algebra $\A$ over $k_\mathbb{A}$, i.e. in the monoidal
category of graded $k_\mathbb{A}$-$k_\mathbb{A}$-bimodules. See 
\cite[\S1.1-1.2]{Lefevre-SurLesAInftyCategories} for an explanation 
of how the usual notions of ordinary, DG, and $\Ainfty$-algebras
over a ring generalise to those over an arbitrary monoidal category. 
We abuse the notation by also using $\A$ to denote 
the object set $\mathbb{A}$ where it doesn't cause confusion, 
e.g. we write $k_\A$ for $k_\mathbb{A}$. 

Given two $\Ainfty$-categories $\A$ and $\B$  an \em $\Ainfty$-functor \rm 
$\A \xrightarrow{F} \B$ is a map $\mathbb{A} \xrightarrow{\dot{F}}
\mathbb{B}$ of their object sets and a morphism 
$\A \rightarrow \leftidx{_{\dot{F}}}{\B}_{\dot{F}}$
of $\Ainfty$-algebras in the category of graded 
$k_\mathcal{A}$-$k_\mathcal{A}$-bimodules. 

The definitions of modules, bimodules, etc. for 
$\Ainfty$-algebras in \S\ref{section-Ainfty-algebras-modules-bimodules}
generalise similarly to $\Ainfty$-categories by considering 
the latter as $\Ainfty$-algebras in approriate categories of graded 
bimodules, cf.~\cite[\S5.1]{Lefevre-SurLesAInftyCategories}. 

Let $\A$ be an $\Ainfty$-category. We denote by $\noddinfA$ 
the DG-category of all right $\Ainfty$-modules over $\A$. Its objects
are right $\A$-modules and for any two objects $E$ and $F$ we have
\begin{align}
\homm_{\noddinfA}(E,F) \simeq \homm_{\infbar \A}(\infbar E, \infbar F), 
\end{align}
cf. \cite[\S5.2]{Lefevre-SurLesAInftyCategories}. It follows that the
elements of $\homm_{\noddinfA}(E,F)$ can be identified with
arbitrary collections of graded $k_\A$-module morphisms
$\left\{ E \otimes_k \A^{\otimes i} \rightarrow F \right\}_{i \geq
0}$. Note that such collection defines a morphism of
$\Ainfty$-modules, as per
\S\ref{section-Ainfty-algebras-modules-bimodules}, if and only if
the corresponding element is closed of degree $0$. 

Let $\A$ and $\B$ be $\Ainfty$-categories. The DG-category
$\noddinfAB$ of $\Ainfty$ $\AbimB$-bimodules is defined similarly 
to the above. 

Let $M$ be an $\AbimB$ bimodule. The notions of 
partial bar constructions and action maps defined 
in \S\ref{section-Ainfty-algebras-modules-bimodules} extend
to $\Ainfty$-functors $\A \rightarrow \noddinfB$ and $\Bopp
\rightarrow \noddinf \Aopp$, cf. 
\cite[Cor. 5.3.0.2]{Lefevre-SurLesAInftyCategories}. 
Given $a \in \A$ and $b \in \B$ we write $\aM$ and $M_b$ for 
their images under these functors. 
When $\A$ and $\B$ are $\Ainfty$-algebras, i.e. $\Ainfty$-categories
with a single object $\bullet$, $\leftidx{_\bullet}M$
is e.g. the $\B$-module which corresponds to $B^\B_\infty M$, 
and the functor $\A \rightarrow \noddinfB$ acts on the morphism
spaces by the $\Ainfty$-morphism 
$\A \rightarrow \homm_{\noddinfB}(\leftidx{_\bullet}M,
\leftidx{_\bullet}M)$ which
corresponds to the action map 
$A \xrightarrow{\action_M} \eend_{\infbar B}(B^B_\infty M)$. 

In case when $M$ is the diagonal bimodule $\A$, this yields
the \em Yoneda embedding \rm $\A \hookrightarrow \noddinfA$.
The modules $\bigl\{\leftidx{_a}\A\bigr\}_{a \in \A}$ 
are the \em representable \rm $\A$-modules. Explicitly, 
the graded $k_\A$-module underlying $\leftidx{_a}\A$ 
is $\homm_\A(-,a)$ and its $\Ainfty$-module structure
is given by the $\Ainfty$-operations $m_i$ of $\A$. 

An $\A$-module $E$ is \em free \rm if 
it is isomorphic to a direct sum of shifts of representable modules. 
An $\A$-module $E$ is \em semi-free \rm if it admits an ascending
filtration whose quotients are free modules.  

\subsection{The derived category of an $\Ainfty$-category} 
\label{section-the-derived-category-of-an-Ainfty-category}

Let $\A$ be an $\Ainfty$-category. A morphism of $\Ainfty$-modules
is a quasi-isomorphism if and only if it is a homotopy equivalence
\cite[Prop. 2.4.1.1]{Lefevre-SurLesAInftyCategories}. Thus
all quasi-isomorphisms are already invertible in $H^0(\noddinfA)$, 
and there is no need to formally invert them to construct 
the derived category of $\A$. Instead, however, we have 
to throw away certain $\A$-modules. For example, given  
an associative unital algebra, for its derived category 
as an $\Ainfty$-algebra to coincide with its usual derived category
we only want to consider its $\Ainfty$-modules which are 
quasi-isomorphic to its ordinary, unital modules. It turns out 
that we need to impose a similar sort of condition 
even when dealing with an arbitrary $\Ainfty$-category. 

We recall various notions of unitality for
$\Ainfty$-categories and their modules. Let $\A$ be an
$\Ainfty$-category. We say that $\A$ is \em strictly unital \rm if it
is equipped with a unit $\eta: k_\A \rightarrow \A$ such that $m_2(\eta
\otimes \id_\A) = m_2(\id_\A \otimes \eta) = \id_\A$ and $m_i(\id_\A
\otimes \dots \otimes \id_\A \otimes \eta \otimes \id_\A \otimes \dots
\otimes \id_\A) = 0$ for all $i \neq 2$. We say that $\A$ 
is \em homologically unital \rm if $H^*(\A)$ is a unital graded
category. Finally, we say that $\A$ is \em $H$-unital \rm 
if its non-augmented bar construction $\infbarnaug \A$ is acyclic. 
These notions are related as 
\begin{equation}
\label{eqn-strictly-unitary-to-homolog-unitary-to-H-unitary}
\text{ strictly unitary } 
\subset
\text{ homologically unital }
\subset
\text{ $H$-unital }. 
\end{equation}
The first inclusion is clear, and the second is due to 
\cite[Prop. 4.1.2.7]{Lefevre-SurLesAInftyCategories}.

Let $E \in \noddinfA$. If $\A$ is strictly unital, we say 
that $E$ is \em strictly unital \rm if $m_2(\id_M \otimes \eta) =
\id_M$ and $m_i(\id_M \otimes \id_\A
\otimes \dots \otimes \id_\A \otimes \eta \otimes \id_\A \otimes \dots
\otimes \id_\A) = 0$ for all $i \geq 3$. If $\A$ is homologically
unital, we say that $E$ is \em homologically unital \rm if $H^*(E)$
is a unital graded module over $H^*(\A)$. Finally, for any $\A$
we say that $E$ is \em $H$-unital \rm if its bar construction 
$\infbar E$ is acyclic. Note that the bar construction of $\A$
considered as a right module over itself coincides with its
non-augmented bar construction as an $\Ainfty$-category. Thus 
$\A$ is an $H$-unital $\Ainfty$-category if and only if $\A$
is an $H$-unital $\A$-module. 

Let $\noddinfhuA$ be the full subcategory of $\noddinfA$ 
consisting of $H$-unital modules. 
We define the \em derived category $D(\A)$ \rm of $\A$ to be $H^0(\noddinfhuA)$. 
When $\A$ is a DG-category, we denote by $D_\infty(\A)$ the derived
category of right $\Ainfty$ $\A$-modules as defined above. 
This is to distinguish it from the derived category $D(\A)$ of right DG
$\A$-modules as defined in \S\ref{section-preliminaries-on-DG-categories}. 

If $\A$ is $H$-unital, then for any 
$E \in \noddinfA$ the following conditions are equivalent:
\begin{enumerate}
\item 
\label{item-E-is-homotopic-to-semi-free}
$E$ is homotopic to a semi-free module.

\item
\label{item-E-is-in-triaA}
$E$ lies in the smallest full subcategory of $H^0(\noddinfA)$ which 
is triangulated, cocomplete, closed under isomorphisms, and contains 
the representable modules.

\item
\label{item-E-is-H-unital}
$E$ is \em $H$-unital\rm, that is --- its bar-construction 
$\infbar E$ is acyclic. 
\end{enumerate}
The equivalence of \eqref{item-E-is-homotopic-to-semi-free} 
and \eqref{item-E-is-in-triaA} is straightforward, 
while that of \eqref{item-E-is-in-triaA} and 
\eqref{item-E-is-H-unital} 
is due to \cite[Prop. 4.1.2.10]{Lefevre-SurLesAInftyCategories}. 

If $\A$ is strictly unital, 
by \cite[Prop. 4.1.3.7]{Lefevre-SurLesAInftyCategories} 
the equivalent conditions above are further equivalent to:
\begin{enumerate}
\setcounter{enumi}{3}
\item 
\label{item-E-is-homologically-unital}
$E$ is \em homologically unital\rm, that is --- 
$H^*(E)$ is a unital graded $H^*(\A)$-module. 
\end{enumerate}

Thus, for a strictly unital $\A$ we have a chain of inclusions
\begin{align*}
\moddinfA \hookrightarrow \noddinfuA \hookrightarrow 
\noddinfhuA \hookrightarrow \noddinfA
\end{align*}
where $\noddinfuA$ is the full subcategory consisting of strictly
unital modules, and $\moddinfA$ is its non-full subcategory of
strictly unital modules and strictly unital morphisms between them. 
The first two inclusions are quasi-equivalences, and thus 
$\moddinfA$ and $\noddinfuA$ are alternative DG-enhancements of 
$D(\A)$.

The derived categories of $\Ainfty$-bimodules are defined similarly and similar 
considerations apply.

\subsection{Tensor and Hom functors for bimodules}
\label{section-tensor-and-hom-functors-for-Ainfty-bimodules}

Let $\A$, $\B$, and $\C$ be $\Ainfty$-categories. 
Let $M \in \noddinf \AbimB$, and $N \in \noddinf \BbimC$. 
We define the $\Ainfty$-tensor product
$M \inftimes_\B N$ to be the $\Ainfty$ $\AbimC$-bimodule
whose bar construction is the (shifted) cotensor product of DG-comodules
\begin{align}
\label{eqn-bar-construction-of-ainfty-tensor-product}
\infbar M \otimes_{\infbar \B} \infbar N[-1]. 
\end{align}

Explicitly, the underlying graded $k_\A$-$k_\C$-bimodule is 
\begin{align}
M \otimes_k (\infbar \B) \otimes_k N 
\end{align}
and its $\Ainfty$ $\AbimC$-bimodule structure 
consists of the differential 
$$ m_{0,0} =  - d_{\infbar M} \otimes \id - \id \otimes d_{\infbar N} +
\id \otimes d_{\infbar \B} \otimes \id $$
and of commuting $\A$ and $\C$ actions 
induced from those on $M$ and $N$ respectively: 
\begin{align*}
m_{p,r} \bigl((a_1 \otimes \dots \otimes a_p) \otimes  m \otimes ( b_1 \otimes \dots \otimes b_q) \otimes n \otimes (c_1 \otimes \dots \otimes c_r) \bigr)
\end{align*}
equals
\begin{align}
\label{eqn-Ainfty-structure-on-Ainfty-tensor-product-explicit}
\begin{cases}
0 
&\quad \text{ if }  p, r \neq 0 \\
\bigoplus_{i = 0}^{q} (-1)^? m^M_{p,i}\bigl(a_1 \otimes \dots \otimes
a_p \otimes  m \otimes b_1 \otimes \dots \otimes b_i \bigr) \otimes
(b_{i+1} \otimes \dots \otimes b_q) \otimes n
&\quad \text{ if }  p \neq 0, r = 0
\\
\bigoplus_{i=0}^{q} (-1)^? m \otimes (b_1 \otimes \dots \otimes b_i) \otimes 
m^N_{q-i,r} \bigl(b_{i+1} \otimes \dots \otimes  b_{q} \otimes n \otimes c_1 \otimes \dots \otimes c_r \bigr)
&\quad \text{ if }  p = 0, r \neq 0 
\end{cases}
\end{align}
with the signs dictated by \eqref{eqn-bar-construction-of-ainfty-tensor-product}.

Let now $M \xrightarrow{f} M'$ be a morphism in $\noddinf \AbimB$ 
and $N \xrightarrow{g} N'$ to be a morphism in $\noddinf \BbimC$. 
Define the morphism 
$$ M \inftimes_\B N \xrightarrow{ f \otimes g} M' \inftimes_\B N' $$
to be the morphism in $\noddinf \AbimC$ which corresponds to the
DG-bicomodule morphism 
$$ \infbar M \otimes_{\infbar \B} \infbar N \xrightarrow{f \otimes g} 
\infbar M' \otimes_{\infbar \B} \infbar N'.$$
Explicitly, $f \otimes g$ sends
\begin{align*}
(a_1 \otimes \dots \otimes a_p) \otimes  m \otimes ( b_1 \otimes \dots \otimes b_q) \otimes n \otimes (c_1 \otimes \dots \otimes c_r) 
\end{align*}
to
\begin{small}
\begin{align*}
\bigoplus_{ 0  \leq i \leq j \leq q} (-1)^? f_{p,i}\bigl(a_1 \otimes \dots \otimes
a_p \otimes  m \otimes b_1 \otimes \dots \otimes b_i \bigr) \otimes
(b_{i+1} \otimes \dots \otimes b_j) \otimes
g_{q-j, r} 
\bigl(b_{j+1} \otimes \dots \otimes b_q \otimes  n \otimes c_1 \otimes \dots \otimes c_r \bigr).
\end{align*}
\end{small}

We thus obtain a DG-functor:
\begin{align}
(-) \inftimes_\B (-)\colon 
\noddinf \AbimB \otimes_k \noddinf \BbimC \longrightarrow \noddinf \AbimC.
\end{align}
Note that $f \otimes \id$ is $\C$-strict: 
$(f \otimes \id)_{i,j} = 0$ if $j > 0$. 
It follows that for any $M \in \noddinf \AbimB$ the functor
$$ (-) \inftimes_\A M\colon \noddinf \A \longrightarrow \noddinf \B $$
filters through the non-full subcategory 
$\noddinfstr \B \subset \noddinf \B$ consisting of all
$\B$-modules and strict $\Ainfty$-morphisms between them. 

If the category $\B$ is a DG-category, then the above defined
$\Ainfty$-tensor product over $\B$ is different from the usual
DG-tensor product over $\B$ and we need to differentiate the two notions: 

\begin{defn}
\label{defn-dg-tensor-product-of-some-Ainfty-bimodules}
Let $\A$ and $\C$ be $\Ainfty$-categories and let $\B$ be a DG-category. 
Let $M \in \noddinf \AbimB$ and $N \in \noddinf \BbimC$ be such that their
partial bar-constructions $B^\A_\infty M$ and $B^\C_\infty N$  are
DG-modules over $\B$. In other words, $m^M_{i,j} = 0$ if $j \geq 2$ and $m^N_{i,j} = 0$ if $i \geq 2$.

The \em DG-tensor product \rm $M \otimes_\B N$ is the $\Ainfty$ 
$\AbimC$-bimodule which corresponds to the free DG 
$B_\infty \A$-$B_\infty \C$ bicomodule 
obtained as the DG-tensor product of the partial bar constructions 
of $M$ and $N$:
\begin{align}
\label{eqn-dg-tensor-product-of-some-Ainfty-bimodules}
B^\A_\infty M \otimes_\B B^\C_\infty M. 
\end{align}
Explicitly, the underlying DG 
$k_\A$-$k_\C$-bimodule 
is $M \otimes_\B N$
and the commuting $\A$ and $\C$ $\Ainfty$-actions 
given by  
\begin{align*}
m^{M \otimes_\B N}_{p,r}(a_1, \dots, a_p, m \otimes n, c_1, \dots, c_r ) = 
\begin{cases}
0 & p,r \neq 0 \\
(-1)^? m \otimes m^N_{0,r}(n,c_1, \dots, c_r) 
& p = 0  \\
(-1)^? m^M_{p,0}(a_1,\dots, a_p,m) \otimes n
& r = 0 
\end{cases}
\end{align*}
with the signs dictated by 
\eqref{eqn-dg-tensor-product-of-some-Ainfty-bimodules}. 
\end{defn}

In particular, for any $\Ainfty$-categories $\A$, $\B$, and $\C$ and
any $M \in \noddinf \AbimB$, and $N \in \noddinf \BbimC$ denote 
by $M \otimes_k N$ the above construction applied to $M$ and $N$ 
considered as $\A$-$k$ and $k$-$\C$ bimodules, respectively. In other
words, we simply forget the $\B$-module structure on $M$ and $N$,
tensor them as DG $k_\B$ modules, and then define the commuting $\A$ 
and $\C$ $\Ainfty$-actions on the result. 

A particularly useful application of this construction is to tensor 
with the diagonal bimodule. 
Let $\A$ be an $\Ainfty$-category and $N$ be a DG $k_{\A}$-module. 
The DG-tensor product $N \otimes_k \A$ can be considered as 
the $\A$-module generated by $N$ over $\A$. Explicitly,  
it has $N \otimes_k \A$ as the underlying DG $k$-module and 
for each $p \geq 2$ we have
$$ m^{N \otimes_k \A}_{p}(n \otimes a, a_1, \dots, a_{p-1} ) = 
(-1)^? n \otimes m_{p}(a,a_1, \dots, a_{p-1}) \quad \quad 
n \otimes a \in N \otimes_k \A,\; a_i \in \A $$
with appropriate signs. 

\begin{lemma}
\label{lemma-over-a-field-N-otimes_k-A-is-semi-free}
Let $\A$ be an $\Ainfty$-category and let $N$ be a DG $k_{\A}$-module. 
The $\A$-module $N \otimes_k \A$ admits a filtration of length two whose 
quotients are free modules. In particular, $N \otimes_k \A$ is semi-free. 
\end{lemma}
\begin{proof}
Suppose first that $N$ is a graded $k_{\A}$-module considered 
as a DG-module with zero differential. 
Then $N \otimes_k \A$ is a free $\A$-module, as it is isomorphic to
$$ \bigoplus_{a \in \A, i \in \mathbb{Z}} (N_a)_i \otimes_k \aA [-i].$$ 
On the other hand, if $N$ is a DG-module bounded from above, then
$N \otimes_k \A$ is semi-free as 
it admits a filtration whose quotients are $N_i \otimes_k \A$. In
particular, if all $N_i$ vanish for $i \notin [a,b]$ for some 
$a,b \in \mathbb{Z}$, then $N \otimes_k \A$ admits a filtration 
of length $b - a + 1$ whose factors are free. 

Finally, since $k$ is a field, we can (non-canonically) decompose 
DG $k_\A$-module $N$ as a direct sum of its graded 
cohomology module $H^*(N)$ and 
acyclic DG-modules $\img d_i \rightarrow \img d_i$ 
concentrated in degrees $i$ and $i+1$.  Therefore, the $\A$-module $N \otimes_k \A$ splits into a direct sum of a free module and the modules which each admit a filtration of length $2$ whose quotients are free. The desired assertion follows. 
\end{proof}

Now let $L \in \noddinf \DbimB$, and $M \in \noddinf \AbimB$. 
We define the $\Ainfty$-Hom bimodule $\infhom_{\B}(L,M)$  
as follows. The underlying graded $k_\A$-$k_\D$-bimodule is
\begin{align}
\label{eqn-graded-bimodule-underlying-Ainfty-Hom}
\homm_{\infbar \B}(\infbarB L, \infbarB M). 
\end{align}
It has a natural structure of a DG-bimodule over DG-categories
$\eend_{\infbar \B}(\infbarB M)$
and $\eend_{\infbar \B}(\infbarB L)$. 
Using the $\A$- and $\D$-action functors
we restrict this to an $\Ainfty$ $\AbimD$-bimodule structure, cf. 
\cite[\S6.2]{Keller-IntroductionToAInfinityAlgebrasAndModules}. 

Explicitly, this bimodule structure consists of
the standard differential 
\begin{align}
m_{0,0}(\alpha) = d_{\infbarB F} \circ \alpha - (-1)^{|\alpha|} \alpha
\circ d_{\infbarB E}
\end{align}
and of commuting $\A$ and $\D$ actions: for any $\alpha \in
\homm_{\infbar \B}(\infbarB E, \infbarB F)$ we have  
\begin{align*}
m_{p,r} \bigl((a_1 \otimes \dots \otimes a_p) \otimes 
\alpha \otimes (d_1 \otimes \dots \otimes d_r) \bigr)
= 
\begin{cases}
0 
&\quad \text{ if }  p, r \neq 0 \\
(-1)^? \action_M(a_1 \otimes \dots \otimes a_p) \circ \alpha
&\quad \text{ if }  r = 0
\\
(-1)^? \alpha \circ \action_M(d_1 \otimes \dots \otimes d_r) \bigr)
&\quad \text{ if }  p = 0.
\end{cases}
\end{align*}
where the signs are dictated by the definition of the restriction 
functor in 
\cite[\S6.2]{Keller-IntroductionToAInfinityAlgebrasAndModules}. 

Let now further $N \in \noddinf \CbimB$. It can be readily 
checked that the \em composition map \rm 
\begin{align}
\label{eqn-Ainfty-DG-cobimodule-composition-map}
\infbar \big(\infhom_{\B}\left(M, N\right)\big)
\otimes_{\infbar \A} 
\infbar \big(\infhom_{\B}\left(L,M\right)\big)  
\xrightarrow{\composition} 
\infbar \big(\infhom_{\infbar \B}\left(L,N\right)\big)
\end{align}
defined by 
\begin{small}
\begin{align*}
\quad &
\infbar \C \otimes_k 
\homm_{\infbar \B}\left(\infbarB M, \infbarB N\right)
\otimes_k \infbar \A \otimes_k 
\homm_{\infbar \B}\left(\infbarB L, \infbarB M\right)
\otimes_k \infbar \D 
\xrightarrow{\id^{\otimes 2} \otimes \tau \otimes \id^{\otimes 2}}
\\
\rightarrow \quad &
\infbar \C \otimes_k 
\homm_{\infbar \B}\left(\infbarB M, \infbarB N\right)
\otimes_k
\homm_{\infbar \B}\left(\infbarB L, \infbarB M\right)
\otimes_k \infbar \D 
\xrightarrow{\id \otimes \composition \otimes \id}
\\
\rightarrow \quad &
\infbar \C \otimes_k 
\homm_{\infbar \B}\left(\infbarB L, \infbarB N\right)
\otimes_k \infbar \D 
\end{align*}
commutes with the differentials. It defines therefore 
in $\noddinfCD$ the \em composition map \rm  
\begin{align}
\label{eqn-Ainfty-composition-map}
\infhom_B(M,N) \inftimes_\A \infhom(L,M) \xrightarrow{\composition}
\infhom_B(L,N). 
\end{align}
\end{small}

Let now $L' \xrightarrow{f} L$ and $M \xrightarrow{g} M'$
be morphisms in $\noddinf \DbimB$ and $\noddinf \AbimB$, respectively. 
Define the morphism 
$$\infhom_\B(L,M) \xrightarrow{g \circ (-) \circ f} \infhom_\B(L',M')$$ 
in $\noddinf \AbimD$ by the DG bicomodule morphism 
\begin{align*} 
& \infbar \big(  \infhom_\B\left(L,M\right) \big) \simeq 
\infbar \A
\otimes_{\infbar \A}
\infbar \big(  \infhom_\B\left(L,M\right) \big)
\otimes_{\infbar \D}
\infbar \D 
\xrightarrow{
g_\A
\otimes \id \otimes
f_\D
} \\
\rightarrow \quad &
\infbar \big(  \infhom_\B\left(M,M'\right) \big)
\otimes_{\infbar \A}
\infbar \big(  \infhom_\B\left(L,M\right) \big)
\otimes_{\infbar \D}
\infbar \big(  \infhom_\B\left(L',L\right) \big)
\xrightarrow{\composition}
\infbar \big(  \infhom_\B\left(L',M'\right) \big).
\end{align*}
Explicitly, for any $\alpha \in
\homm_{\infbar \B}(\infbarB L, \infbarB M)$ the map  
$(f \circ - \circ g)_{p,r}$ sends
\begin{align*}
(a_1 \otimes \dots \otimes a_p) \otimes 
\alpha \otimes (d_1 \otimes \dots \otimes d_r) 
\end{align*}
to the map 
\begin{align*}
(-1)^?
\big(
g_{p, \bullet} \left(a_1 \otimes \dots \otimes a_p \otimes -\right)
\otimes \id
\big) 
\circ 
\Delta 
\circ 
\alpha
\circ  
\left( 
f_{r, \bullet}\left(d_1 \otimes \dots \otimes d_r \otimes -\right)
\otimes \id 
\right) 
\circ \Delta 
\end{align*}
in $\homm_{\infbar \B}(\infbarB L', \infbarB M')$. 
Here $\Delta$ denotes, as usual, the comodule comultiplications. 

We thus obtain a DG-functor
\begin{align}
\infhom_{\B}(-,-) \colon (\noddinf \CbimB)^\opp \otimes \noddinf \AbimB
\rightarrow \noddinf \AbimC
\end{align}
and, similar to the above, for any $M \in \noddinf \AbimB$ the functor
$$ \infhom_\B(M,-)\colon \noddinfB \longrightarrow \noddinfA $$ 
filters through
$\noddinfstr \A \subset \noddinf \A$.  

We then have the usual Tensor-Hom adjunction: for every 
$M \in \noddinf \AbimB$ the functors
\begin{align}
(-) \inftimes_\A M\colon \noddinf \CbimA \rightarrow \noddinf \CbimB \\
\infhom_\B(M,-) \colon \noddinf \CbimB \rightarrow \noddinf \CbimA 
\end{align}
are left and right adjoint to each other, respectively. Same holds
for the functors $M \inftimes_\B (-)$ and $\infhom_{\Aopp}(M,-)$. 

Let $M \in \noddinf \AbimB$. The DG $k_\A$-$k_\A$-bimodule underlying 
$\infhom_{\B}(M,M)$ has an algebra structure given by composition, and
thus defines a DG-category with the same object set as $\A$. 
By definition, this DG-category can be naturally identified 
with the DG-category $\eend_{\infbar \B}(\infbarB M)$.
On the other hand, it can be identified with the image
of the functor $\A \rightarrow \noddinfB$ defined by $M$. 
Indeed, the assignment $a \mapsto \leftidx{_a}M$ gives 
a fully faithful inclusion
$\infhom_{\B}(M,M) \hookrightarrow \noddinfB$, and 
the functor $\A \rightarrow \noddinfB$
decomposes as 
\begin{align}
\A \xrightarrow{\action_M} 
\infhom_{\B}(M,M) \hookrightarrow \noddinfB. 
\end{align}
Here we write $\action_M$ for the composition 
$\A \xrightarrow{\action_M} \eend_{\infbar \B}(\infbarB M) 
\simeq \infhom_{\B}(M,M)$. 

\subsection{A functorial semi-free resolution for $\noddinfhuA$}
\label{section-functorial-semi-free-resolution-for-noddinfhuA}
 
To our knowledge, the material presented in this section is original 
to this paper. 

Let $\A$ be an $\Ainfty$-category and let $E \in \noddinfA$. 
Consider the $\Ainfty$ $\A$-module $E \inftimes_\A \A$. The corresponding
DG $\infbar \A$-comodule is 
\begin{align}
\label{eqn-DG-comodule-corresponding-to-E-inftimes-A}
\infbar E \otimes_{\infbar \A} \infbarbim \A \; [-1].
\end{align}
As a graded $\infbar \A$-comodule  
\eqref{eqn-DG-comodule-corresponding-to-E-inftimes-A} is isomorphic to
$$ E \otimes_k \infbar \A \otimes_k \A[1] \otimes_k \infbar \A $$
which decomposes as 
$$ \bigoplus_{i \geq 0} E \otimes_k (\A[1])^{\otimes i} 
\otimes_k \A [1] \otimes_k \infbar \A.$$

Observe that  the component of the differential of 
\eqref{eqn-DG-comodule-corresponding-to-E-inftimes-A} which goes
from its $i$-th summand to its $j$-th summand is zero if $j > i$. It follows that this differential decomposes into:
\begin{enumerate}
\item For each $i \geq 0$ a 
degree $1$ square zero $\infbar \A$-coderivation
\begin{align}
\label{eqn-E-A^otimes-i-coderivation} 
E \otimes_k (\A[1])^{\otimes i} \otimes_k \A [1] \otimes_k \infbar \A
\longrightarrow 
E \otimes_k (\A[1])^{\otimes i} \otimes_k \A [1] \otimes_k \infbar \A
\end{align}
\item 
For each $i > j \geq 0$ a degree $1$ 
graded $\infbar \A$-comodule morphism
\begin{align}
\label{eqn-E-A^otimes-i-to-E-A^otimes-j-morphism} 
E \otimes_k (\A[1])^{\otimes i} \otimes_k \A [1] \otimes_k \infbar \A
\longrightarrow 
E \otimes_k (\A[1])^{\otimes j} \otimes_k \A [1] \otimes_k \infbar \A. 
\end{align}
\end{enumerate}

For any $i > 0$ write $E \otimes_k \A^{\otimes i}$ for the $\Ainfty$ 
$\A$-module 
$$ \bigl(E \otimes_k \A^{\otimes (i-1)}\bigr) \otimes_k \A $$
in the sense of Definition 
\ref{defn-dg-tensor-product-of-some-Ainfty-bimodules}. 
The corresponding DG $\infbar \A$-comodule is the graded 
$\infbar \A$-comodule 
$$ E \otimes_k \A^{\otimes (i-1)} \otimes_k \A[1] \otimes_k \infbar \A $$ 
whose differential is (the shift of) 
the coderivation \eqref{eqn-E-A^otimes-i-coderivation}. 

\begin{defn}
\label{defn-Ainfty-morphism-E-A^i-to-E-A^j}
For any $i > 0$ define a degree $1 - i$ morphism 
\begin{align}
\label{eqn-Ainfty-morphism-E-A^i-to-E} 
E \otimes_k \A^{\otimes i} 
\longrightarrow 
E  
\end{align}
in $\noddinf \A$ by the graded $k_\A$-module maps  
$$ E \otimes_k \A^{\otimes i} \otimes_k \A^{\otimes n}
\xrightarrow{m_{n+i+1}} E. $$ 

For any $i$ and $j$ with $i > j > 0$ define a degree $i-j+1$ morphism 
\begin{align}
\label{eqn-Ainfty-morphism-E-A^i-to-E-A^j} 
E \otimes_k \A^{\otimes i} 
\longrightarrow 
E \otimes_k \A^{\otimes j}
\end{align}
in $\noddinf \A$ by the (shift of the) graded $\infbar \A$-comodule morphism
$\eqref{eqn-E-A^otimes-i-to-E-A^otimes-j-morphism}$. 
\end{defn}

Explicitly, \eqref{eqn-Ainfty-morphism-E-A^i-to-E-A^j} is defined by
the maps
\begin{align*}
f_{n + 1} \colon 
E \otimes_k \A^{\otimes (i - 1)} \otimes_k \A \otimes_k \A^{\otimes n}
\longrightarrow 
E \otimes_k \A^{\otimes (j - 1)} \otimes_k \A 
\end{align*}
where 
\begin{align*}
f_1 = \sum_{\underset{r \geq 0, s \geq 0}{r + 1 + s = j + 1}} 
(-1)^{?} \; \id^{\otimes r} \otimes m_{i + 1 - r - s} \otimes 
\id^{\otimes s}. 
\end{align*}
and for any $n \geq 1$ 
\begin{align*}
f_{n+1} = (-1)^{?} \; \id^{\otimes j} \otimes m_{i - j + 1 + n}. 
\end{align*}

\begin{lemma}
\label{lemma-E-inftimesA-A-is-semi-free}
Let $\A$ be an $\Ainfty$-category. The functor 
$$ (-) \inftimes_{\A} \A \colon \noddinfA \rightarrow \noddinfA $$
filters through the full subcategory $\semi-free^{\strict}(\A)
\subset \noddinfA$ consisting of semi-free modules and strict
$\Ainfty$-morphisms between them. 
\end{lemma}
\begin{proof}
As explained in \S\ref{section-tensor-and-hom-functors-for-Ainfty-bimodules}
for any $M \in \noddinf \AbimB$ the functor 
$$(-) \inftimes_{\A} M \colon \noddinf \A \longrightarrow \noddinf \B $$ 
filters through $\noddinfstr \B$. It remains to show that for
any $E \in \noddinfA$ the module $E \inftimes_\A \A$ is semi-free. 

Recall the decomposition of the differential on the DG comodule 
corresponding to $E \inftimes_{\A} \A$ discussed prior to and employed in 
Definition \ref{defn-Ainfty-morphism-E-A^i-to-E-A^j}. It follows tautologically
that $E \inftimes_{\A} \A$ is isomorphic to the convolution of 
the twisted complex 
\begin{align}
\label{eqn-universal-semi-free-resolution}
\xymatrix{ 
\ar@{-->}@/^4ex/[rr]^<<<<<<<<<<<{}
\ar@{..>}@/^6ex/[rrr]^<<<<<<<<<<<{}
&
\dots 
\ar[r]
\ar@{..>}@/^6ex/[rrr]^<<<<<<<<<<<{}
\ar@{-->}@/^4ex/[rr]^<<<<<<<<<<<{}
&
{E \otimes_k \A^{\otimes 3}}
\ar[r]
\ar@{-->}@/^4ex/[rr]^<<<<<<<<<<<{}
&
{E \otimes_k \A^{\otimes 2}}
\ar[r]
&
\underset{\degzero}{E \otimes_k \A}
}
\end{align}
whose differentials are 
the $\Ainfty$-morphisms \eqref{eqn-Ainfty-morphism-E-A^i-to-E-A^j}.

Thus it suffices to show that the convolution of \eqref{eqn-universal-semi-free-resolution} is semi-free. As the twisted complex \eqref{eqn-universal-semi-free-resolution} is bounded from above and one-sided, its convolution admits an exhaustive filtration whose quotients are (the shifts of) its objects, the modules $E \otimes_k \A^{\otimes i}$.  On the other hand, since $k$ is a field, by Lemma \ref{lemma-over-a-field-N-otimes_k-A-is-semi-free} each of the modules $E \otimes_k \A^{\otimes i}$ in \eqref{eqn-universal-semi-free-resolution} admits a filtration of length $2$ whose quotients are free modules.  We thus obtain an exhaustive filtration on the convolution of $E \otimes_k \A^{\otimes i}$ whose quotients are free modules, as desired. 
\end{proof}

\begin{cor}
\label{cor-inftimes-factors-through-strictlyunital/DG-modules}
Let $\A$ be a strictly unital $\Ainfty$-category (resp. a DG-category). 
The functor 
$$ (-) \inftimes_{\A} \A \colon \noddinfA \rightarrow \noddinf \A $$
filters through the full subcategory of $\noddinfstr \A$ consisting
of strictly unital (resp. DG) modules. 
\end{cor}
\bf NB: \rm When $\A$ is a DG-category, strict $\Ainfty$-morphisms between 
DG-modules are simply the DG-morphisms, so 
the subcategory of $\noddinfstr \A$ consisting of DG-modules is
canically isomorphic to the usual DG-category $\modA$ of
DG-modules over $\A$. 

For any $E \in \noddinfA$ there is a map of twisted complexes from 
\eqref{eqn-universal-semi-free-resolution} to $E$ concentrated in
degree $0$ whose individual components are the maps 
$E \otimes_k \A^{\otimes k} \rightarrow E$ defined in 
\eqref{eqn-Ainfty-morphism-E-A^i-to-E}:
\begin{align}
\label{eqn-universal-semi-free-resolution-morphism-to-Id}
\vcenter{
\xymatrix{ 
\ar@{-->}@/^4ex/[rr]^<<<<<<<<<<<{}
\ar@{..>}@/^6ex/[rrr]^<<<<<<<<<<<{}
&
\dots 
\ar[r]
\ar@{..>}@/^6ex/[rrr]^<<<<<<<<<<<{}
\ar@{-->}@/^4ex/[rr]^<<<<<<<<<<<{}
\ar[drrr]
&
{E \otimes_k \A^{\otimes 3}}
\ar[r]
\ar@{-->}@/^4ex/[rr]^<<<<<<<<<<<{}
\ar[drr]
&
{E \otimes_k \A^{\otimes 2}}
\ar[r]
\ar[dr]
&
\underset{\degzero}{E \otimes_k \A} 
\ar[d]
\\
& & & & 
\underset{\degzero}{E}.
}
}
\end{align}
It can be readily checked that this map is closed of degree $0$.  
As per the proof of Lemma \ref{lemma-E-inftimesA-A-is-semi-free}, 
the convolution of the top complex is isomorphic to $E \inftimes_\A \A$. 
We can therefore define:
\begin{defn}
Let $\A$ be an $\Ainfty$-category. Define a natural transformation  
\begin{align}
\label{eqn-natural-transformation-E-inftimesA-A-to-E}
(-) \inftimes_{\A} \A \longrightarrow \id. 
\end{align}
by setting for each $E \in \noddinfA$ the corresponding 
morphism $E \inftimes_\A \A \rightarrow E$ to be the convolution 
of \eqref{eqn-universal-semi-free-resolution-morphism-to-Id}. 
\end{defn}
This was defined in different terms in 
\cite[Lemme 4.1.1.6]{Lefevre-SurLesAInftyCategories} for
strictly unital modules. 

\begin{prps}
\label{prps-E-inftimesA-A-to-E-is-quasiiso-iff-bar-complex-is-acyclic}
For any $E \in \noddinfA$ the morphism 
$E \inftimes \A 
\xrightarrow{\eqref{eqn-natural-transformation-E-inftimesA-A-to-E}}
E$ is a quasi-isomorphism if and only if $\infbar E$ is acyclic. 
\end{prps}
\begin{proof}
The morphism 
$E \inftimes \A 
\xrightarrow{\eqref{eqn-natural-transformation-E-inftimesA-A-to-E}}
E$
is induced by the twisted complex morphism 
\eqref{eqn-universal-semi-free-resolution-morphism-to-Id}. 
It can be readily checked that the convolution of the total complex 
of \eqref{eqn-universal-semi-free-resolution-morphism-to-Id}
is an $\Ainfty$ $\A$-module whose underlying DG $k_\A$-module
is the same as that of $\infbar E$. The claim now follows.  
\end{proof}

Recall, as discussed in 
\S\ref{section-the-derived-category-of-an-Ainfty-category}, 
the full subcategory $\noddinfhuA \subset \noddinfA$ consisting of
$H$-unital modules. These are, equivalently,
the modules whose bar construction is acyclic and the
modules homotopic to semifree modules. We therefore obtain:
\begin{cor}
\label{cor-functorial-semi-free-resolution-for-noddinfhu}
The natural transformation
\eqref{eqn-natural-transformation-E-inftimesA-A-to-E}
is a functorial semi-free resolution for $\noddinfhuA$. If, moreover,
$\A$ is strictly unital (resp. DG), then this resolution  
is also a strictly unital (resp. DG) resolution.  
\end{cor}

We note that 
Prop.~\ref{prps-E-inftimesA-A-to-E-is-quasiiso-iff-bar-complex-is-acyclic}
generalises and simplifies the proofs of several results in  
\cite{Lefevre-SurLesAInftyCategories}, \em 
Chapitre 4\rm, e.g. the proof that every module whose bar construction
is acyclic is homologically unital. 

\subsection{The bar complex} 
\label{section-the-bar-complex}

In this section, we give an account of the \em bar complex\rm, the
notion which lies at the technical heart of this paper. It is obtained
from the bar construction on a DG category $\A$, but the key point is
that the resulting object is considered in the monoidal category
$(\AmodA, \otimes_\A, \A)$ of $\AbimA$-bimodules, as opposed to the
monoidal category $(\kmodk$, $\otimes_k, k)$ of DG $k$-$k$-bimodules.
To this extent, we provide below an alternative construction which
works purely in terms of the former monoidal category and is an
instance of a more general notion of a \em twisted tensor algebra\rm. 

Let $\A$ be a DG category. As per
\S\ref{section-Ainfty-algebras-modules-bimodules}, the bar
construction on $\A$ is the graded $k$-$k$-bimodule $\bigoplus_{i \geq
1} \A^{i}[i]$ with the structure of a (non-unital) coalgebra in the
monoidal category $(\kmodk$, $\otimes_k, k)$ of DG $k$-$k$-bimodules.
This structure consists of a differential and a comultiplication. The
differential, together with the natural left and right actions of $\A$
by composition,  makes $\bigoplus_{i \geq 1} \A^{i}[i]$ into a DG
$\A$-$\A$-bimodule. The comultiplication map lifts to define on this
bimodule a non-unital coalgebra structure in the monoidal category
$(\AmodA, \otimes_\A, \A)$.

Our first point of interest is its shift by one to the right. The resulting DG $\AmodA$-bimodule has the following natural description in the language of the twisted complexes:

\begin{defn}
Let $\A$ be a DG category. Define the \em extended bar complex \rm  $\tildeA \in \AmodA$  to be the convolution of the following twisted complex of $\AbimA$ bimodules
\begin{equation}
\label{eqn-extended-bar-complex} 
\begin{tikzcd}[column sep={2.5cm}]
\dots 
\ar{r}
&
\A \otimes_k \A \otimes_k \A
\ar{r}{-\id \otimes \eqref{eqn-canonical-map-Aotimes_kA-to-A} + \eqref{eqn-canonical-map-Aotimes_kA-to-A}  \otimes \id}
&
\A \otimes_k \A
\ar{r}{-\eqref{eqn-canonical-map-Aotimes_kA-to-A}}
& 
\underset{\degzero}{\A}
\end{tikzcd}
\end{equation}
whose differentials $\A^{\otimes(n+1)} \rightarrow \A^{\otimes n}$
are given by 
$$\sum_{i=0}^{n-1} (-1)^{i+1} \id^{\otimes(i)} \otimes \eqref{eqn-canonical-map-Aotimes_kA-to-A} \otimes\id^{\otimes(n-i-1)}$$ 
and all the higher differentials are zero. 
\end{defn}

The $\AbimA$-bimodule $\tildeA$ is well-known to be acyclic, since as
a $k$-$k$ DG-bimodule it admits a contracting homotopy of degree $-1$ whose components are the maps $\A^{\otimes n} \rightarrow \A^{\otimes (n+1)}$
defined by
$$ a_1 \otimes \dots \otimes a_n \mapsto 1 \otimes a_1 \otimes \dots \otimes a_n.$$

Thus the twisted complex \eqref{eqn-extended-bar-complex} yields a resolution of the diagonal bimodule $\A$ by what is known as the \em bar complex \rm:
\begin{defn}
Let $\A$ be a DG category. The \em bar complex \rm  $\barA \in \AmodA$ 
is the convolution of the twisted complex of free $\AbimA$ bimodules
\begin{equation}
\label{eqn-bar-complex} 
\begin{tikzcd}[column sep={2.5cm}]
\dots 
\ar{r}
&
\A \otimes_k \A \otimes_k \A
\ar{r}{\id \otimes \eqref{eqn-canonical-map-Aotimes_kA-to-A} - \eqref{eqn-canonical-map-Aotimes_kA-to-A}  \otimes \id}
&
\underset{\degzero}{\A \otimes_k \A}
\end{tikzcd}
\end{equation}
whose differentials $\A^{\otimes(n+1)} \rightarrow \A^{\otimes n}$
are given by 
\begin{align}
\label{eqn-differentials-in-barA-twisted-complex}
\sum_{i=0}^{n-1} (-1)^i \id^{\otimes(i)} \otimes \eqref{eqn-canonical-map-Aotimes_kA-to-A} \otimes\id^{\otimes(n-i-1)}
\end{align}
and all the higher differentials are zero. 

Explicitly, the underlying graded $\AbimA$-bimodule of $\barA$ is 
$\bigoplus_{n \geq 2} \A^{\otimes n}[n-2]$ and its differential 
$d_{\barA}$ sends any 
$$ a_1 \otimes \dots \otimes a_n \in \A^{\otimes n}[n-2]$$
to the sum of 
\begin{align}
\label{eqn-natural-part-of-the-differential-on-barA}
\sum_{i = 1}^n (-1)^{n+\sum^{i-1}_{j=1} \deg(a_j)} a_1 \otimes \dots \otimes da_i \otimes \dots \otimes a_n,
\end{align}
which comes from the natural differential on $\A^{\otimes n}$ and 
\begin{align}
\label{twisted-natural-part-of-the-differential-on-barA}
\sum_{i = 1}^n (-1)^{i-1} a_1 \otimes \dots \otimes a_i a_{i+1} \otimes \dots \otimes a_n, 
\end{align}
which comes from the differential in the twisted complex \eqref{eqn-bar-complex}. 
\end{defn}

Since the complex \eqref{eqn-bar-complex} is bounded from above and each of its terms is free, $\barA$ is semi-free. It also comes with a canonical projection to $\A$:
\begin{defn}
Define the canonical projection
\begin{align}
	\tau \colon \barA \rightarrow \A
\end{align}
to be the convolution of the following map of the twisted complexes:
\begin{align}
\label{eqn-bar-complex-natural-map-tau} 
\vcenter{
\xymatrix{ 
\dots 
\ar[r]
& 
\A^{\otimes (n+1)} 
\ar[r]
& 
\A^{\otimes n} 
\ar[r]
& 
\dots 
\ar[r]
&
\A \otimes_k \A \otimes_k \A
\ar[r]
&
\underset{\degzero}{\A \otimes_k \A}
\ar[d]^{\eqref{eqn-canonical-map-Aotimes_kA-to-A}}
\\
& & & & & 
\underset{\degzero}{\A}.
}
}
\end{align} 
Explicitly, it is the map 
\begin{align*}
	a_1 \otimes \dots \otimes a_n \mapsto 
	\begin{cases}
		a_1 a_2 & \quad \quad n = 2 \\
		0 & \quad \quad \text{ otherwise.}
	\end{cases}
\end{align*}
\end{defn}

By definition of $\tau$, the convolution of 
\begin{align}
\label{eqn-twisted-complex-barA-into-A}
\barA \xrightarrow{\tau} \underset{\degzero}{\A}. 
\end{align}
is equal to the convolution of the total complex of \eqref{eqn-bar-complex-natural-map-tau}, that is --- to $\tildeA$. Since the latter is acyclic, $\tau$ is a quasi-isomorphism, and thus $\barA$ is a canonical semi-free resolution of the diagonal bimodule $\A$. 
 
The extended bar complex $\tildeA$ admits a structure of an algebra in the monoidal category  $(\AmodA, \otimes_\A, \A)$. It comes from a general construction which we now describe. This construction itself is an instance of the cobar construction on a curved $\Ainfty$-coalgebra, cf. \cite[\S7.4]{Positselski-TwoKindsOfDerivedCategoriesKoszulDualityAndComoduleContramoduleCorrespondence}. However,  it is a degenerate case where the comultiplication and the higher operations are all all zero, leaving only the cocurvature and the differential. It is worth it therefore to give a direct definition:
 
 \begin{defn}
 Let $\A$ be a DG category. Let $H \in \AmodA$ and $\sigma: H \rightarrow \A$ be a closed map of degree $0$. The \em $\sigma$-twisted tensor algebra \rm $T_\sigma(H)$ of $H$ is the convolution of the twisted complex
 \begin{equation}
\label{eqn-twisted-tensor-algebra} 
\begin{tikzcd}[column sep={1.5cm}]
\dots 
\ar{r}
&
H \otimes_\A H 
\ar{r}{\id \otimes \sigma - \sigma \otimes \id}
&
H 
\ar{r}{\sigma}
&
\underset{\degzero}{\A}
\end{tikzcd}
\end{equation}
whose differentials $H^{\otimes n} \rightarrow H^{\otimes (n-1)}$
are given by 
$$\sum_{i=0}^{n-1} (-1)^i \id^{\otimes(i)} \otimes \sigma \otimes\id^{\otimes(n-i-1)}$$ 
and all the higher differentials are zero. 

In other words, as a graded $\A$-$\A$ bimodule $(T_\sigma, m,e)$ is just the tensor
algebra $\oplus_{i \geq 0} H^{\otimes i}[i]$, but the natural
differential on the latter is modified using the map $\sigma$, whence
the word ``twisted'' in our choice of the name. 

Define further
\begin{align*}
e\colon \A \rightarrow T_\sigma(H)	
\end{align*}
to be the canonical inclusion, and the map
\begin{align*}
m\colon T_\sigma(H) \otimes_\A T_\sigma(H) \mapsto T_\sigma(H)
\end{align*}
by the natural left and right actions of $\A$ on each $H^{\otimes}[n]$ and by the sign-twisting isomorphisms
\begin{align*}
H^{\otimes p}[p] \otimes_\A H^{\otimes q}[q] &\rightarrow H^{\otimes(p+q)}[p+q]\\
\left(h_1 \otimes \dots \otimes h_p\right) \otimes \left(h_{p+1} \otimes \dots h_{p+q}\right) 
& \mapsto 
(-1)^{q \sum_{i=1}^p \deg(h_i)}
h_1 \otimes \dots \otimes h_p \otimes h_{p+1} \otimes \dots h_{p+q}.
\end{align*}
The latter come from the signless associativity isomorphisms  of $\otimes_\A$ using the sign-twisting identifications $H^{\otimes p}[p] \simeq (H[1])^{\otimes p}$, cf. \cite[\S1.1.1]{Lefevre-SurLesAInftyCategories}.
\end{defn} 

\begin{lemma}
\label{lemma-twisted-tensor-algebra-is-an-algebra}
The triple $(T_\sigma(H), m, e)$ is a unital algebra in the monoidal category
$(\AmodA, \otimes_\A, \A)$. 
\end{lemma}
\begin{proof}
It is easy to check that it is precisely the unital algebra obtained
via the cobar construction from the curved $\Ainfty$-coalgebra
structure given on $H[2]$ by the cocurvature $\sigma$, the natural
differential, and with the comultiplication and all the higher
operations being zero. Indeed, the data of a curved
$\Ainfty$-coalgebra is the most general way to define the differential
on the tensor algebra of a graded module in order to obtain a DG
algebra, see \cite[\S7.4]{Positselski-TwoKindsOfDerivedCategoriesKoszulDualityAndComoduleContramoduleCorrespondence}
for further details. 
\end{proof}
The extended bar complex can be viewed as a twisted tensor algebra in the following way:
\begin{defn}
Define the degree $0$ map 
\begin{align}
\label{eqn-extended-bar-complex-multiplication}
m\colon \tildeA \otimes_\A \tildeA \rightarrow \tildeA
\end{align}
by 
$$ \left( a_1 \otimes \dots \otimes a_p \right) 
\otimes_\A (a_{p+1} \otimes \dots \otimes a_{p+q})
\mapsto 
(-1)^{(q-1) \sum_{i=1}^p \deg(a_i)} 
a_1 \otimes \dots \otimes a_p a_{p+1} \otimes \dots \otimes a_{p+q}
$$
and let 
$$ e\colon \A \hookrightarrow \tildeA $$ 
be the canonical inclusion. 
\end{defn}
\begin{cor}
The triple $(\tildeA, m, e)$ is a unital algebra in the monoidal category
$(\AmodA, \otimes_\A, \A)$. 
\end{cor}
\begin{proof}
Follows from the Lemma \ref{lemma-twisted-tensor-algebra-is-an-algebra} 
by setting $H = \A \otimes_k \A$ and $\sigma$ to be the map 
$\A \otimes_k \A \xrightarrow{\eqref{eqn-canonical-map-Aotimes_kA-to-A}} \A$. 
\end{proof}

We now decompose the mulplication map on the extended bar complex $\tildeA$ into components pertaining to $\barA$ and to $\A$. 
The bimodule $\tildeA$ is not just isomorphic but equal to 
the convolution of the twisted complex
$(\barA \xrightarrow{\tau} \underset{\degzero}{\A})$ , these are merely
two different descriptions of the same differential 
on $\bigoplus_{i \geq 1} \A^{i}[i-1]$. 
Therefore, by the formula for the tensor product of twisted complexes
\cite[Lemma 3.4]{AnnoLogvinenko-SphericalDGFunctors},
the convolution of the twisted complex 
\begin{equation}
\label{eqn-tensor-square-of-barA-to-A-as-a-complex}
\barA \otimes_\A \barA 
\xrightarrow{
\left(\begin{smallmatrix} 
\id \otimes \tau \\ - \tau \otimes \id 
\end{smallmatrix}\right)
}
\barA \otimes_\A \A \oplus \A \otimes_\A \barA 
\xrightarrow{
\left(\begin{smallmatrix} 
- \tau & - \tau 
\end{smallmatrix}\right)
}
\A \otimes_\A \A
\end{equation}
is isomorphic to $\tildeA \otimes_\A \tildeA$ via
a sign-twisting isomorphism 
$$
\left(a_1 \otimes \dots \otimes a_p \right) \otimes_\A
\left(a_{p+1} \otimes \dots \otimes a_{p+q}\right) 
\mapsto 
\begin{cases}
\left(a_1 \otimes \dots \otimes a_p \right) \otimes_\A
\left(a_{p+1} \right) 
& 
\quad q = 1,
\\
(-1)^{\deg(a_1)}
\left(a_1 \right) \otimes_\A
\left(a_{p+1} \otimes \dots \otimes a_{p+q} \right) 
& 
\quad p = 1,
\\
-1^{p + \sum_{i=1}^p \deg(a_i)}
\left(a_1 \otimes \dots \otimes a_p \right) \otimes_\A
\left(a_{p+1} \otimes \dots \otimes a_{p+q}\right) 
& 
\quad p, q > 1.
\end{cases}
$$

The multiplication map \eqref {eqn-extended-bar-complex-multiplication} is a closed, degree zero map. Composing it with the isomorphism above
and applying the natural isomorphisms \eqref{eqn-DG-A-otimes-M-to-M-isomorphism} and  \eqref{eqn-DG-M-otimes-B-to-M-isomorphism},  we obtain the closed, degree zero map of twisted complexes 
\begin{equation}
\label{eqn-extended-bar-complex-multiplication-as-map-of-twisted-complexes} 
\begin{tikzcd}[column sep={2.5cm},ampersand replacement=\&]
\barA \otimes_\A \barA
\ar{dr}{\mu}
\ar{r}{\left(\begin{smallmatrix} 
\id \otimes \tau \\ - \tau \otimes \id 
\end{smallmatrix}\right)}
\&
\barA \oplus \barA
\ar{d}{
\left(\begin{smallmatrix} 
\id & \id 
\end{smallmatrix}\right)
}
\ar{r}{
\left(\begin{smallmatrix} 
- \tau & - \tau 
\end{smallmatrix}\right)
}
\&
\underset{\degzero}{\A}
\ar{d}{\id}
\\
\&
\barA 
\ar{r}{- \tau}
\&
\underset{\degzero}{\A}
\end{tikzcd}
\end{equation}
where the map $\mu$ is defined as follows. By 
\cite[Lemma 3.4(1)]{AnnoLogvinenko-SphericalDGFunctors} we 
can identify $\barA \otimes_\A \barA$ with the convolution of
the twisted complex 
\begin{align}
\label{eqn-twisted-complex-barA-otimes-barA}
\dots
\rightarrow
\left(
\A^{\otimes 3} \otimes_\A \A^{\otimes 2} \bigoplus
\A^{\otimes 2} \otimes_\A \A^{\otimes 3}
\right)
\rightarrow
\underset{\degzero}{\A^{\otimes 2} \otimes_\A \A^{\otimes 2}}
\end{align}
whose degree zero differentials are defined on each  
$\A^{\otimes p} \otimes_\A \A^{\otimes q}$
by
$$ \sum
\eqref{eqn-differentials-in-barA-twisted-complex} \otimes \id + 
(-1)^p \id \otimes \eqref{eqn-differentials-in-barA-twisted-complex}
$$
and whose higher differentials are all zero. 

\begin{defn}
\label{defn-degree-minus-one-map-mu}
Let $\A$ be a DG-category. Define the degree $-1$ map 
\begin{align}
\label{eqn-degree-minus-one-multiplication-in-bar-complex}
\mu\colon \barA \otimes_\A \barA  \rightarrow \barA
\end{align}
in $\AmodA$ to be the map induced by the degree $-1$ map from the twisted complex
\eqref{eqn-twisted-complex-barA-otimes-barA} 
to the twisted complex
\eqref{eqn-bar-complex}
whose only components are the degree zero maps 
$\bigoplus_{n = p+q} \A^{\otimes p} \otimes_\A \A^{\otimes q} \rightarrow \A^{n-1}$ given by
\begin{align}
\label{eqn-mu-defined-as-a-map-of-twisted-complexes}
\left(a_1 \otimes \dots \otimes a_p \right) \otimes_\A
\left(a_{p+1} \otimes \dots \otimes a_{p+q}\right)
\mapsto (-1)^{p} a_1 \otimes \dots \otimes a_p a_{p+1} \otimes \dots \otimes a_{p+q}. 
\end{align}
\end{defn}
We note that the identification of 
$\barA \otimes_\A \barA$ with the convolution of \eqref{eqn-twisted-complex-barA-otimes-barA}
given in  \cite[Lemma 3.4(1)]{AnnoLogvinenko-SphericalDGFunctors} involves a sign-twisting isomorphism. 
Consequently, the explicit formula for $\mu$ as a map in $\AmodA$ is the formula 
\eqref{eqn-mu-defined-as-a-map-of-twisted-complexes} with an extra sign twist $q\sum_{i=1}^p \deg(a_i)$. 

We have immediately:
\begin{lemma}
\label{lemma-d-mu-equals-tau-id-minus-id-tau}
Let $\A$ be a DG-category. Then 
\begin{enumerate}
\item 
\label{item-d-mu-equals-tau-id-minus-id-tau}
$d\mu = \tau \otimes \id - \id \otimes \tau$, 
\item 
\label{item-composition-tau-circ-mu}
$\tau \circ \mu = 0$. 
\end{enumerate}	
in $\AmodA$. 
\end{lemma}
\begin{proof}
This follows immediately from \eqref {eqn-extended-bar-complex-multiplication-as-map-of-twisted-complexes} being a closed, degree zero map of twisted complexes. 
\end{proof}

The bar-complex $\barA$ has a  natural coalgebra structure in $(\AmodA, \otimes_\A, \A)$ which is defined as follows: 

\begin{defn}
Define the comultiplication 
\begin{align}
\barA \xrightarrow{\Delta} \barA \otimes_\A  \barA,
\end{align}	
to be the map induced by the degree $0$ map from the twisted complex
\eqref{eqn-bar-complex}
to the twisted complex 
\eqref{eqn-twisted-complex-barA-otimes-barA} whose only components
are the degree zero maps 
$\A^{\otimes n} \rightarrow \bigoplus_{n = p+q} \A^{\otimes(p+1)} \otimes_\A \A^{\otimes(q+1)}$
given by
\begin{align}
\label{eqn-formula-for-delta-as-a-map-of-twisted-complexes}
a_1 \otimes \dots \otimes a_n & \mapsto
\sum_{p =1}^{n-1} (a_1 \otimes \dots \otimes a_p \otimes 1) 
\otimes_{\A} 
(1 \otimes a_{p+1} \otimes \dots \otimes a_{p+q}). 
\end{align}
\end{defn}

As explained for the map $\mu$, the explicit formula for $\Delta$ as a map in $\AmodA$ 
is the formula \eqref{eqn-formula-for-delta-as-a-map-of-twisted-complexes}  
with the extra  sign twist $(-1)^{(n-i+1)\sum_{k=1}^{i-1} \deg(a_k)}$. 

\begin{prps}
\label{prps-coalgebra-structure-on-the-bar-complex}
The triple $(\barA, \Delta, \tau)$ is a unital coalgebra in the monoidal category $(\AmodA, \otimes_\A, \A)$. 
\end{prps}
\begin{proof}
With the definitions above it is a straightforward verification on the level of  twisted complexes over $\AmodA$.
\end{proof}

\section{Bar category of modules $\modbarA$}
\label{section-bar-category-of-modules}

Let $X$ be a scheme of finite type over $k$. By \cite{BondalVanDenBergh-GeneratorsAndRepresentabilityOfFunctorsInCommutativeAndNoncommutativeGeometry}
the category $D_{qc}(X)$ admits a compact generator. Hence 
$D_{qc}(X) \simeq D(\A)$ for a DG-algebra $\A$ which is
the endomorphism algebra of (an $h$-injective resolution of) such 
a generator. Similarly, by \cite{Lunts-CategoricalResolutionOfSingularities}
the category $D(X)$ admits a classical generator and we have 
$D(X) \simeq D_c(\B)$ for the endomorphism DG-algebra $\B$ of such
generator. Moreover, the generator can be chosen in such a way that
$\B$ is smooth. See \cite[\S 4]{AnnoLogvinenko-SphericalDGFunctors} for
 a detailed exposition, as well as generalities on DG-enhancements.

This reduces DG-enhancing derived categories of algebraic varieties 
to DG-enhancing derived categories of DG-modules over DG-algebras or, 
more generally, DG-categories.
Let $\A$ and $\B$ be DG-categories and let 
$$ D(\A) \xrightarrow{f} D(\B) $$ be a DG-enhanceable functor. 
Recall that $f$ is said to be \em continuous \rm if it commutes
with infinite direct sums.  
By 
\cite[Theorem 7.2]{Toen-TheHomotopyTheoryOfDGCategoriesAndDerivedMoritaTheory}
every DG-enhanceable continuous functor $D(\A) \rightarrow D(\B)$ 
is of the form $(-) \ldertimes M$ for some bimodule $M \in D(\AbimB)$. 
It follows that $D(\AbimB)$ can be identified with the 
triangulated category of DG-enhanceable continuous functors 
$D(\A) \rightarrow D(\B)$. This furthermore identifies 
the subcategory $D^{\Bperf}(\AbimB)$ 
with the triangulated category of DG-enhanceable functors
$D_c(\A) \rightarrow D_c(\B)$. This reduces DG-enhancing exact
functors between the derived categories of algebraic varieties to 
DG-enhancing the derived categories of bimodules, cf. \cite{LuntsSchnurer-NewEnhancementsOfDerivedCategoriesOfCoherentSheavesAndApplications}. 

Let $\A$ be a DG-category and let $\modA$ be the DG-category 
of $\A$-modules. There are two enhancements commonly used in 
the literature for $D(\A)$: the full subcategory $\hprojA$ of
the $h$-projective modules in $\modA$, and the Drinfield quotient 
$\modA / \acyc(\A)$ where $\acyc(\A)$ is the full subcategory of 
acyclic modules. Neither turned out
to be suitable for our purposes. The problem with the Drinfeld quotient 
is that its morphisms are inconvenient to work with explicitly.  
The problem with $\hprojA$ manifests itself when working with
bimodules. The diagonal bimodule $\A$, which corresponds to the
identity functor $D(\A) \rightarrow D(\A)$, is not in general $h$-projective.
Hence every construction involving the identity functor has to be
$h$-projectively resolved by e.g.~tensoring with the bar complex. This
leads to many formulas becoming vastly more complicated than they
should be, cf.~\cite{AnnoLogvinenko-SphericalDGFunctors}.

We propose a different DG-enhancement framework for the derived
categories of DG-categories. We think it is more suitable for 
identifying the derived categories of DG-bimodules with triangulated
categories of DG-enhanceable functors as described above. Let $\A$
be a DG-category. The proposed enhancement of $D(\A)$ admits 
two different descriptions. 

\subsection{DG-modules with $\Ainfty$-morphisms between them}
\label{section-DG-modules-with-Ainfty-morphisms-between-them}

The first one is in the language of $\Ainfty$-categories and modules. 
The enhancement we want is the full subcategory 
of the DG-category $\noddinfA$ of $\Ainfty$ $\A$-modules 
which consists of DG $\A$-modules. We denote this subcategory 
by $\noddinfdgA$. Note that 
the subcategory $(\noddinfstr \A)_{dg} \subset \noddinfA$
which consists of DG $\A$-modules and strict $\Ainfty$-morphisms 
between them can be canonically identified with the usual DG-category 
$\modA$ of DG $\A$-modules. Consider the chain of subcategory inclusions  
\begin{align}
\label{eqn-Ainfty-module-subcategory-inclusions}
\sfA \hookrightarrow \modA \hookrightarrow 
\noddinfdgA \hookrightarrow \noddinfhuA.
\end{align}
In$\noddinfhuA$ all quasi-isomorphisms are homotopy equivalences, 
and thus the functorial resolution $(-) \inftimes_\A \A$ of $\noddinfhuA$ into 
$\sfA$ established 
in Cor.~\ref{cor-functorial-semi-free-resolution-for-noddinfhu}
ensures that every full subcategory of $\noddinfhuA$ which 
contains $\sfA$ is quasi-equivalent to $\sfA$.  We thus obtain:

\begin{prps}
Let $\A$ be a DG category.  The natural inclusions 
$$ \sfA \hookrightarrow \noddinfdgA \hookrightarrow \noddinfhuA $$ are  quasi-equivalences. In particular, the induced equivalences $$ 
D(\A) \simeq H^0(\noddinfdgA) \simeq D_\infty(\A) $$
make $\noddinfdgA$ and $\noddinfhuA$ into DG-enhancements of $D(\A)$. 
\end{prps}

For any DG-bimodule $M \in \AmodB$ the adjoint functors 
$(-) \inftimes_\A M$ and $\infhom_\B(M,-)$ restrict
from $\noddinfA \leftrightarrow \noddinfB$ to 
$\noddinfdgA \leftrightarrow \noddinfdgB$. We thus
have the usual Tensor-Hom adjunction for the categories
$\noddinfdg$. 

\subsection{The category $\modbarA$}
\label{section-category-modbar}

The second description is a direct one in the language of DG-modules. 
While less conceptual, it significantly simplifies the computations
involved and allows one to avoid having to deal with the sign conventions
for $\Ainfty$-categories and modules. It builds on the ideas  
introduced in \cite[\S6.6]{Keller-DerivingDGCategories} where it was
applied to a set of compact generators of $D(\A)$ to obtain 
a Morita enhancement.  We apply it to the whole of $\modA$ instead:

\begin{defn}
\label{defn-bar-category-of-modules}
Let $\A$ be a DG-category. Define the \em bar category of modules \rm
$\modbarA$ as follows:
\begin{itemize}
\item 
The object set of $\modbarA$ is the same as that of $\modA$:
DG-modules over $\A$.  
\item For any $E,F \in \modA$ set
\begin{align}
\label{eqn-hom-sets-in-modbar-A} 
\homm_{\modbarA}(E,F) = \homm_\A(E \otimes_\A \barA, F)
\end{align}
and write $\barhom_\A(E,F)$ to denote this $\homm$-complex. 
\item For any $E \in \modA$ set $\id_E \in \barhom_\A(E,E)$
to be the element given by 
\begin{align}
\label{eqn-identity-in-modbar-A} 
E \otimes_\A \barA \xrightarrow{\id \otimes \tau} E \otimes_\A \A
\xrightarrow{\sim} E
\end{align}
where $\tau \colon \barA \rightarrow \A$ is the counit of $\barA$ as the coalgebra in $\AmodA$,  cf.~\S\ref{section-preliminaries-on-DG-categories}. 
\item For any $E,F,G \in \modA$ define the composition map 
\begin{align}
\label{eqn-modbar-composition-map}
\barhom_\A(F,G) \otimes_k \barhom_\A(E,F)
\longrightarrow 
\barhom_\A(E,G)
\end{align}
by setting for any $E \otimes_\A \barA \xrightarrow{\alpha} F$
and $F \otimes_\A \barA \xrightarrow{\beta} G$ 
the composition of the corresponding elements 
to be the element given by
\begin{align}
\label{eqn-composition-in-modbar-A} 
E \otimes_\A \barA \xrightarrow{\id \otimes \Delta}
E \otimes_\A \barA \otimes_\A \barA
\xrightarrow{\alpha \otimes \id}
F \otimes_\A \barA 
\xrightarrow{\beta}
G.
\end{align}
where $\Delta\colon \barA \rightarrow \barA \otimes_\A \barA$ is the comultiplication of $\barA$ as the coalgebra in $\AmodA$,  cf.~\S\ref{section-preliminaries-on-DG-categories}.
\end{itemize}

Let $\A$ and $\B$ be DG-categories. We define the bimodule category 
$\AmodbarB$ similarly, but with 
\begin{align}
\barhom_{\AbimB}(M,N) = \homm_{\AbimB}(\barA \otimes_\A M \otimes_\B
\barB, N) \quad \quad \forall M,N \in \AmodB, 
\end{align}
with $\id_M \in \barhom_{\AbimB}(M,M)$ being the element given by
\begin{equation*}
\barA \otimes_\A M \otimes_\B \barB \xrightarrow{\tau \otimes \id
\otimes \tau} M,
\end{equation*}
and with the composition of the $\AmodbarB$ morphisms 
$M \rightarrow N$ and $L \rightarrow M$  corresponding to 
$$ \beta \colon \barA \otimes_\A M \otimes_\B \barB \rightarrow N
\quad \text{ and } \quad
 \alpha \colon \barA \otimes_\A L \otimes_\B \barB \rightarrow M $$
being the $\AmodbarB$ morphism $L \rightarrow N$ corresponding to
$$ 
\barA \otimes_\A L \otimes_\B \barB
\xrightarrow{\Delta \otimes \id \otimes \Delta}
\barA \otimes_\A \barA \otimes_\A L \otimes_\B \barB \otimes_\B \barB
\xrightarrow{\id \otimes \alpha \otimes \id }
\barA \otimes_\A M \otimes_\B \barB 
\xrightarrow{\beta}
N.
$$
\end{defn}

\begin{prps}
\label{prps-embedding-of-modA-into-modbarA}
Let $\A$ be a DG-category. We have a (non-full) inclusion
\begin{align}
\label{eqn-embedding-of-modA-into-modbarA}
\Upsilon\colon \modA \hookrightarrow \modbarA
\end{align}
which is the identity on objects. 
\end{prps}
\begin{proof}
Define $\Upsilon$ to be the identity on
objects and for any $E,F \in \modA$ define
$$ \Upsilon_{E,F}\colon 
\homm_\A(E,F) \rightarrow
\barhom_\A(E,F)$$
to be the map which sends any $\alpha \in \homm_\A(E,F)$ to 
the morphism in $\modbarA$ defined by 
\begin{align*}
E \otimes_\A \barA  \xrightarrow{\alpha \otimes \tau} 
F \otimes_\A \A
\xrightarrow{\sim} F. 
\end{align*}

The map $\Upsilon_{E,F}$ is injective as it can be rewritten 
as the pre-composition of $\alpha$ with the map 
\eqref{eqn-identity-in-modbar-A} which is surjective. 
This also shows that it sends $\id$ in 
$\modA$ to $\id$ in $\modbarA$. 
It remains to check that $\Upsilon_{E,F}$ is compatible with compositions. 
Let $G \in \modA$ and let $\beta \in \homm_\A(F,G)$.  
By definition, the composition $\Upsilon_{F,G}(\beta) \circ
\Upsilon_{E,F}(\alpha)$ in $\modbarA$ 
is the element of $\barhom_{\A}(E,G)$ defined by
\begin{align*}
E \otimes_\A \barA \xrightarrow{\id \otimes \Delta}
E \otimes_\A \barA \otimes_\A \barA 
\xrightarrow{\alpha \otimes \tau \otimes \id}
F \otimes_\A \barA 
\xrightarrow{\beta \otimes \tau}
G 
\end{align*}
where we suppress the isomorphisms $(-)\otimes_\A \A \simeq \id_\A$. 
By functoriality of tensor product, this simplifies to 
$$ E \otimes_\A \bar\A \xrightarrow{\left(\beta \circ \alpha\right)
\otimes \left((\tau \otimes \tau) \circ \Delta\right)} G. $$
Since $(\tau \otimes \tau) \circ \Delta = \tau$ the above equals
$\Upsilon_{E,G}(\beta \circ \alpha)$, as desired. 
\end{proof}

The inclusion $\Upsilon$ is 
a special case of a more general identification which relates 
this section to \S\ref{section-DG-modules-with-Ainfty-morphisms-between-them}:

\begin{defn}
\label{defn-modbarA-to-noddinfdgA-isomorphism}
Let $\A$ be a DG-category. Define a DG-functor 
\begin{align}
\label{eqn-modbarA-to-noddinfdgA-isomorphism}
\Psi\colon \modbarA \rightarrow \noddinfdgA 
\end{align}
by setting it to be the identity on the objects, and
for any $E,F \in \modbarA$ setting 
\begin{align}
\label{eqn-hombar-to-hominf-isomorphism}
\Psi(-)\colon \barhom_\A(E,F) \rightarrow \infhom_\A(E,F)
\end{align}
to be the following map. By definition 
$\barhom_\A(E,F) = \homm_\A(E \otimes_\A \barA, F)$, 
and the module $E \otimes_\A \barA$ is isomorphic
to the convolution of the twisted complex 
\begin{align}
\label{eqn-E-otimes-barA-as-a-twisted-complex}
\xymatrix{ 
\dots 
\ar[r]
&
E \otimes_k \A^{\otimes 3}
\ar[r]
&
E \otimes_k \A^{\otimes 2}
\ar[r]
&
\underset{\degzero}{E \otimes_k \A} 
}
\end{align}
with the degree $0$ differentials 
$$
E \otimes \A^{\otimes n + 1} 
\xrightarrow{\sum_{i = 0}^n (-1)^{i} \id^{\otimes i} \otimes m_2
\otimes \id^{ \otimes (n - i)}}
E \otimes \A^{\otimes n}
$$
where $m_2$ denotes either the composition map 
$\A \otimes_k \A \rightarrow \A$ 
or the action map $E \otimes_k \A \rightarrow E$, as appropriate.
Since $\modA$ is strongly pre-triangulated, the DG complex 
$\homm_\A(E \otimes_\A \barA, F)$ is isomorphic to 
the DG complex of twisted complex morphisms 
\begin{equation}
\label{eqn-map-E-otimes-barA-to-F-as-a-twisted-complex-map}
\begin{tikzcd}
\dots 
\ar{r}
\ar{drrr}[description]{\dots}
&
E \otimes_k \A^{\otimes 3}
\ar{r}
\ar{drr}[description]{\alpha_2}
&
E \otimes_k \A^{\otimes 2}
\ar{r}
\ar{dr}[description]{\alpha_1}
&
\underset{\degzero}{E \otimes_k \A} 
\ar{d}[description]{\alpha_0}
\\
& & &
F.
\end{tikzcd}
\end{equation}
The degree $i$ part of this DG complex comprises all 
$\left\{\alpha_n \right\}_{n \geq 0}$ with $\alpha_n \in
\homm^{i-n}_\A(E\otimes_k \A^{\otimes(n + 1)}, F)$.  
Since each $E \otimes_k \A^{\otimes(n+1)}$ is a free $\A$-module
generated by the $k_\A$-module $E \otimes_k \A^{\otimes
n}$, such collections $\left\{\alpha_n \right\}_{n \geq 0}$
are in bijection with 
$\left\{\alpha'_n \right\}_{n \geq 0}$ with 
$\alpha'_n \in \homm^{i-n}_k(E\otimes_k \A^{\otimes n}, F)$, 
and hence with the elements of $\infhom_\A(E,F)$. 

Thus, define the action \eqref{eqn-hombar-to-hominf-isomorphism}
of $\Psi$ on morphism complexes to be the composition of 
the bijective map which sends the elements of $\barhom_\A(E,F)$ to 
the twisted complex morphisms 
\eqref{eqn-map-E-otimes-barA-to-F-as-a-twisted-complex-map} 
with the bijective map which sends the latter to the elements of 
of $\infhom_\A(E,F)$. The resulting map clearly respects the 
degrees, and checking that it commutes with the differentials
is a straightforward verification, comparing the definition
of the differential for morphisms of twisted complexes given in e.g. 
\cite[\S3.1]{AnnoLogvinenko-SphericalDGFunctors} or
\cite[\S1]{BondalKapranov-EnhancedTriangulatedCategories}
with that of the differential for $\A_\infty$-module morphisms given 
in e.g.~\cite[\S5.2]{Lefevre-SurLesAInftyCategories}. 
\end{defn}

\begin{prps}
\label{prps-modbarA-to-noddinfdgA-isomorphism}
Let $\A$ be a DG-category. The DG-functor $\Psi$ 
of Definition \ref{defn-modbarA-to-noddinfdgA-isomorphism}
is an isomorphism  
$$ \modbarA \simeq \noddinfdgA $$
of DG-categories which identifies 
$\Upsilon\colon \modA \hookrightarrow \modbarA$ with 
$\modA = \left(\noddinfA\right)_{\text{dg}}^{\text{strict}} 
\hookrightarrow \noddinfdgA$. 
\end{prps}
\begin{proof}
By construction, $\Psi$ is identity on objects and bijective  
on morphism complexes. Hence it is an isomorphism of DG categories. 
For the last assertion, let $\alpha \in \homm_{\A}(E,F)$. The
corresponding element of $\barhom_\A(E,F)$ is the composition 
of $E \otimes_\A \barA \xrightarrow{\id \otimes \tau} E$ with 
$E \xrightarrow{\alpha} F$. On the level of twisted complexes, 
the former map consists of a single component 
$E \otimes_k \A \xrightarrow{\action} E$. The composition 
consists therefore of a single component $E \otimes_k \A
\xrightarrow{\alpha \circ \action} F$. The
corresponding $k_\A$-module morphism is $E \xrightarrow{\alpha} F$. 
We conclude that the resulting collection of $k_\A$-module morphisms 
$\left\{E \otimes_k \A^{\otimes n} \rightarrow F\right\}_{n \geq 0}$
consists of a single non-zero component: $E \xrightarrow{\alpha} F$. 
This defines the strict $\Ainfty$-morphism $E \rightarrow F$ 
corresponding to $\alpha$, as required. 
\end{proof}

\begin{cor}
\label{cor-canonical-category-isomorphism-D(A)-to-H^0-modbar-A}
Let $\A$ be a DG category. There is a canonical category isomorphism 
\begin{align}
\label{eqn-D(A)-to-H^0-modbar-A-isomorphism}
\Theta: D(\A) \xrightarrow{\sim} H^0\left(\modbarA\right)	
\end{align}
giving $\modbarA$ the structure of a DG-enhancement of $D(\A)$. 
\end{cor}
\begin{proof}
In $H^0\noddinfdgA$ every acyclic module is isomorphic
to zero. By Prop.~\ref{prps-modbarA-to-noddinfdgA-isomorphism} it 
is also true of $H^0(\modbarA)$. As 
$D(\A) = H^0(\modA)/H^0\left(\acyc\A\right)$ the universal property of Verdier
quotient ensures that the inclusion
$$ H^0(\Upsilon)\colon H^0(\modA) \rightarrow H^0(\modbarA) $$ 
factors uniquely into the canonical projection 
$H^0(\modA) \rightarrow D(\A)$ and the functor we define to be $\Theta$:
\begin{align}
\label{eqn-modbarA-iso-DG-enhancement-of-DA}
H^0(\modA) 
\rightarrow 
D(\A)
\xrightarrow{\Theta}
H^0(\modbarA).
\end{align}

To see that $\Theta$ is an isomorphism of categories, 
precompose $H^0(\Upsilon)$ with $H^0(\hprojA) \hookrightarrow H^0(\modA)$, 
where $\hprojA \subset \modA$ is the full subcategory of 
$h$-projective modules. On the morphism complexes $\Upsilon$ is the 
pre-composition with the quasi-isomorphism $E \otimes_\A \barA
\rightarrow E$ defined in \eqref{eqn-identity-in-modbar-A}. In both 
$H^0(\hprojA)$ and $H^0(\modbarA)$ quasi-isomorphisms become
isomorphisms. Thus the resulting functor 
$H^0(\hprojA) \rightarrow H^0(\modbarA)$ is fully faithful. 
It is furthermore an equivalence, since any module $E$ is isomorphic
in $H^0(\modbarA)$ to the $h$-projective module $E \otimes_\A \barA$ . 
Since the composition 
$$ H^0(\hprojA) \hookrightarrow H^0(\modA) \rightarrow D(\A) $$
is well-known to be an equivalence, we conclude that $\Theta$ is also
one. 
\end{proof}

\begin{lemma}
Let $\A$ be a DG-category and let $E \xrightarrow{\alpha} F$ be the $\modbarA$ 
morphism defined by a $\modA$ morphism $E \otimes_\A \barA
\xrightarrow{\alpha'} F$.
The category isomorphism $\Theta$ of 
Cor.~\ref{cor-canonical-category-isomorphism-D(A)-to-H^0-modbar-A}
identifies $\alpha$ with the $D(\A)$ morphism
\begin{align}
\label{eqn-E-to-F-in-D(A)-which-corresponds-to-a-modbar-morphism}
E \xrightarrow{\id \otimes \tau^{-1}} E \otimes_\A \barA 
\xrightarrow{\alpha'} F.
\end{align}
Here $\tau^{-1}$ is the formal inverse of 
the quasi-isomorphism $\barA \xrightarrow{\tau} \A$. 
\end{lemma}
\begin{proof}
It suffices to show that $\Theta(\alpha')$ is 
the composition $\alpha \circ \Theta(\id \otimes \tau)$ in $H^0(\modbarA)$. 
As $\Theta(\alpha') = H^0(\Upsilon)(\alpha')$, we conclude
it is the morphism defined by the $\modA$ morphism 
$$ E \otimes_\A \barA \otimes_\A \barA \xrightarrow{\alpha' \circ (\id \otimes \tau)} F. $$

On the other hand,  the image of $E \otimes_\A \barA \xrightarrow{\id \otimes \tau} E$
in $\modbarA$ is defined by the $\modA$ morphism 
$$ E \otimes_\A \barA \otimes_\A \barA \xrightarrow{\id \otimes \tau \otimes \tau} E.$$
Its $\modbarA$ composition with $\alpha$ is therefore defined by the $\modA$ morphism
$$ E \otimes_\A \barA \otimes_\A \barA \xrightarrow{\alpha' \circ (\id \otimes \tau \otimes \id)} F.$$

The claim now follows, since the map $\barA \otimes_\A \barA \xrightarrow{\tau \otimes \id - \id \otimes \tau} \barA$
is null-homotopic.  One choice for the contracting homotopy can be
found in Lemma \ref{lemma-d-mu-equals-tau-id-minus-id-tau}.
\end{proof}

\begin{cor}
\label{cor-criterion-for-modbar-morphism-to-be-homotopy-equivalence}
Let $\A$ be a DG-category.
A morphism $E \rightarrow F$ in $\modbarA$ 
is a homotopy equivalence if and only if
the corresponding $\modA$ morphism $E \otimes_\A \barA \rightarrow F$ is a quasi-isomorphism.
\end{cor}

We next furnish the categories $\modbar$ with 
adjoint bifunctors which are enhancements of 
the derived bifunctors $(-) \ldertimes_\B (-)$ and
$\rder\homm_B(-,-)$.

\begin{defn}
\label{defn-bar-tensor-bifunctor}
Let $\A$, $\B$, and $\C$ be DG-categories. Define the functor 
\begin{align}
\bartimes_\B \colon 
\AmodbarB \otimes_k \BmodbarC \rightarrow \AmodbarC 
\end{align}
by setting
\begin{align*}
M \bartimes_\B N = M \otimes_\B \barB \otimes_\B N 
\quad\quad
\forall\; M \in \AmodbarB, \; N \in \BmodbarC. 
\end{align*}
Furthermore, for any 
$\alpha \in \barhom_{\AbimB}(M,M')$ and 
$\beta \in \barhom_{\BbimC}(N,N')$ define 
$$ M \bartimes_\B N \xrightarrow{ \alpha \bartimes \beta }
M' \bartimes_\B N' $$
to be the morphism corresponding to
\begin{align*}
\barA \otimes_\A M \otimes_\B \barB \otimes_\B N \otimes_\C \barC
\xrightarrow{\id^{\otimes 2} \otimes \Delta^2 \otimes \id^{\otimes 2}}
\barA \otimes_\A M \otimes_\B \barB \otimes_\B \barB \otimes_\B \barB \otimes_\B 
N \otimes_\C \barC
\xrightarrow{\alpha \otimes \id \otimes \beta}
M' \otimes_\B \barB \otimes_\B N'.
\end{align*}

Explicitly, we have 
$$ a \otimes m \otimes b \otimes n \otimes c \mapsto
\sum
(-1)^{\deg(\beta) \deg(a \otimes m \otimes b_{(1)} \otimes b_{(2)})}
\alpha(a \otimes m \otimes b_{(1)}) \otimes b_{(2)} \otimes
\beta(b_{(3)} \otimes n \otimes c). 
$$
Here and below we use Sweedler's notation for comultiplications and 
coactions: the sum above runs over all the summands 
$b_{(1)} \otimes b_{(2)} \otimes b_{(3)}$ of $\Delta^2(b)$. More 
generally, given a coalgebra element denoted by e.g. letter $b$
a sum involving expressions $b_{(1)}, \dots, b_{(k)}$ means that
the sum is taken over all summands of $\Delta^k(b)$ with each $b_{(i)}$ 
denoting the $i$-th factor of each summand. 
\end{defn}

\begin{defn}
\label{defn-bar-hom-bifunctor}
Let $\A$, $\B$, and $\C$ be DG-categories. Define the functor 
\begin{align}
\barhom_\B(-,-)\colon 
\AmodbarB \otimes_k (\CmodbarB)^{\opp} \rightarrow \AmodbarC 
\end{align}
by setting
\begin{align*}
\barhom_\B\left(M,N\right) = \homm_\B\left(M \otimes_\B \barB, N\right) 
\quad\quad
\forall\; M \in \CmodbarB, \; N \in \AmodbarB. 
\end{align*}
Furthermore, for any 
$\alpha \in \barhom_{\CbimB}(M',M)$ and 
$\beta \in \barhom_{\AbimB}(N,N')$ define 
$$ \barhom_{\B}\left(M,N\right) 
\xrightarrow{\beta \circ (-) \circ \alpha}
\barhom_{\B}\left(M', N'\right)$$
by the $\AmodC$ map 
\begin{align}
\label{eqn-A-and-C-action-on-infhom_B}	
& \barA \otimes_\A  
\homm_{\B}\left(M \otimes_\B \barB ,N\right) 
\otimes_\C \barC \xrightarrow{\beta_\A \otimes \id \otimes \alpha_\C} 
\\
\nonumber
\rightarrow \quad 
& \homm_{\B}\left(N \otimes_\B \barB, N'\right) 
\otimes_\A  
\homm_{\B}\left(M \otimes_\B \barB, N\right) 
\otimes_\C
\homm_{\B}\left(M' \otimes_\B \barB, M\right) 
\xrightarrow{\eqref{eqn-modbar-composition-map}} 
\\
\nonumber
\rightarrow \quad 
& 
\homm_{\B}\left(M' \otimes_\B \barB, N'\right). 
\end{align}
Here 
$\barA \xrightarrow{\beta_\B} \homm(N \otimes_\B \barB, N')$
and 
$\barC \xrightarrow{\alpha_\C} \homm(M' \otimes_\B \barB, M)$
are the right adjoints of the $\AmodB$ and $\CmodB$ morphisms 
$$\barA \otimes_\A N \otimes_\B \barB \rightarrow  N', $$
$$\barC \otimes_\C M'\otimes_\B \barB \rightarrow M $$
which correspond to $\beta$ and $\alpha$.  

Explicitly, the map $\eqref{eqn-A-and-C-action-on-infhom_B}$ takes any
$a \otimes \gamma \otimes c$ to the map 
$$ 
m' \otimes b 
\mapsto 
\sum
(-1)^{\deg(\alpha) (\deg(a) + \deg(\gamma))}
\beta\left(a \otimes \gamma\left(\alpha\left(c \otimes m' \otimes
b_{(1)}\right) \otimes b_{(2)} \right) \otimes b_{(3)} \right).
$$

We define similarly the functor
\begin{align}
\barhom_{\Bopp}(-,-)\colon 
\BmodbarC \otimes_k (\BmodbarA)^{\opp} \rightarrow \AmodbarC. 
\end{align}
\end{defn}

\begin{defn}
Let $\A$, $\B$, $\C$, and $\D$ be DG-categories. For any 
$M \in \AmodbarB$, $L \in \CmodbarB$, and $N \in \DmodbarB$
define the \em composition map \rm in $\DmodbarC$ 
\begin{align}
\label{eqn-composition-map-in-modbar} 
{\composition}\colon 
\barhom_\B(M,N) \bartimes_\A \barhom_\B(L,M)
\rightarrow \barhom_\B(L,N) 
\end{align}
by the corresponding $\DmodC$ map 
\begin{align}
\nonumber
\barD \otimes_\D \homm_\B(M \otimes_\B \barB,N) \otimes_\A \barA \otimes_\A 
\homm_\B(L \otimes_\B \barB,M) \otimes_\C \barC
\xrightarrow{\tau \otimes \id \otimes \tau \otimes \id \otimes \tau}
\\
\label{eqn-composition-map-in-modbar-underlying-mod-map} 
\rightarrow 
\homm_\B(M \otimes_\B \barB,N) \otimes_\A \homm_\B(L \otimes_\B \barB,M)
\xrightarrow{ \eqref{eqn-modbar-composition-map} }
\homm_\B(L \otimes_\B \barB,N).
\end{align}

For  any 
$M \in \BmodbarA$, $L \in \BmodbarC$, and $N \in \BmodbarD$
define similarly $\CmodbarD$ map 
\begin{align}
\label{eqn-opp-composition-map-in-modbar} 
{\composition}\colon \barhom_{\Bopp}(L,M) \bartimes_\A \barhom_{\Bopp}(M,N)
\xrightarrow{\composition} \barhom_{\Bopp}(L,N). 
\end{align}
\end{defn}
When working with bimodules where two different categories 
act on the left and on the right, there can be no confusion 
over whether we mean the left or the right action of either category.
In such case, we write $\homm_\A$ for $\homm_{\Aopp}$, for 
the purposes of brevity. 

\begin{prps}
\label{prps-bartimes-and-barhom-are-identified-with-derived-versions}

Let $\A$ and $\B$ be DG-categories and let $M \in \AmodB$. The isomorphism
$\Theta$ of Cor.~\ref{cor-canonical-category-isomorphism-D(A)-to-H^0-modbar-A}
identifies the functors $(-) \ldertimes_\A M$ and
$\rder\homm_\B\left(M,-\right)$ with the functors 
$H^0\left((-) \bartimes_\A M\right)$ and
$H^0\left(\barhom_\B\left(M,-\right)\right)$. Similarly, 
$\Theta$ identifies $M \ldertimes_\B (-)$ and
$\rder\homm_\A\left(M,-\right)$
with $H^0\left(M \bartimes_\B (-) \right)$ and
$H^0\left(\barhom_\A\left(M,-\right)\right)$. 
\end{prps}

\begin{proof}
We only prove the assertion for $(-) \ldertimes_\A M$ and 
$(-) \bartimes_\A M$, the others are proved similarly. 

For any DG-category $\C$ the following square commutes:
\begin{align*}
\xymatrix{
\CmodA  
\ar[rr]^{(-) \otimes_\A \barA \otimes_\A M}
\ar[d]_{\Upsilon}
& &
\CmodB 
\ar[d]^{\Upsilon}
\\
\CmodbarA 
\ar[rr]^{(-) \bartimes_\A M}
& &
\CmodbarB. 
}
\end{align*}
Since $\barA \otimes_\A M$ is an $\A$-$h$-projective resolution of $M$, the functor 
$H^0\left((-) \otimes_\A \barA \otimes_\A M\right)$  descends to the functor 
$(-) \ldertimes_\A M\colon D(\CbimA) \rightarrow D(\CbimB)$. 
The factorisation 
\eqref{eqn-modbarA-iso-DG-enhancement-of-DA}
then implies that this functor is identified by $\Theta$ with 
the functor $H^0\left((-) \bartimes_\A M\right)$ . 
\end{proof}

\begin{prps}
\label{prps-bartensor-and-barhom-are-identified-with-Ainfty-tensor-and-hom}
Let $\A$ and $\B$ be DG-categories.   
The isomorphism $\Psi$
of Prop.~\ref{prps-modbarA-to-noddinfdgA-isomorphism} 
identifies the bifunctors $(-) \bartimes_\B (-)$, $\barhom_\A(-,-)$,
and $\barhom_\B(-,-)$ with the bifunctors 
$(-) \inftimes_\B (-)$, $\infhom_\B(-,-)$, and $\infhom_\A(-,-)$. 
\end{prps}
\begin{proof}
	Straightforward verification.
\end{proof}
In view of
Propositions \ref{prps-bartensor-and-barhom-are-identified-with-Ainfty-tensor-and-hom} and \ref{prps-modbarA-to-noddinfdgA-isomorphism}
we could deduce the Tensor-Hom adjunction for $\modbar$
from the Tensor-Hom adjunction for $\Ainfty$-modules
\cite[\em Lemme 4.1.1.4]{Lefevre-SurLesAInftyCategories}. 
However, it is more convenient to prove this adjunction directly 
in $\modbar$ by exhibiting explicit formulas for its unit and counit:

\begin{prps}
\label{prps-bartimes-and-barhom-are-dg-adjoint}
Let $\A$, $\B$, and $\C$ be DG-categories and let $M \in \AmodbarB$.
\begin{enumerate}
\item 
\label{item-bartimesA-M-and-HomB-M-adjunction}
The functors $(-) \bartimes_\A M$ and $\barhom_\B(M,-)$
are left and right adjoint functors $\CmodbarA \leftrightarrow
\CmodbarB$. The unit and the counit of the adjunction are the maps
\begin{align}
\label{eqn-bartimes-M-adjunction-unit}
E \xrightarrow{\mlt} 
\barhom_\B(M,E \bartimes_\A M) \quad \quad \forall\;
E \in \CmodbarA, \\
\label{eqn-bartimes-M-adjunction-counit}
\barhom_\B(M,F) \bartimes_\A M 
\xrightarrow{\ev} F 
\quad \quad \forall\; F \in \CmodbarB 
\end{align}
in $\CmodbarA$ and $\CmodbarB$ which correspond to the $\CmodA$ and
$\CmodB$ maps
\begin{small}
\begin{align}
\label{eqn-bartimes-M-adjunction-unit-underlying-map-in-Mod}
\barC \otimes_\C E \otimes_\A \barA 
\xrightarrow{\tau \otimes \mlt}
\homm_\B(M \otimes_\B \barB, E \otimes_\A \barA \otimes_\A M \otimes_\B \barB) 
\xrightarrow{(\id^{\otimes 3} \otimes \tau) \circ (-) }
\homm_\B(M \otimes_\B \barB, E \otimes_\A \barA \otimes_\A M), 
\\
\label{eqn-bartimes-M-adjunction-counit-underlying-map-in-Mod}
\barC \otimes_\C \homm_\B(M \otimes_\B \barB,F) \otimes_\A \barA \otimes_\A M \otimes_\B \barB
\xrightarrow{\tau \otimes \id \otimes \tau \otimes \id^{\otimes 2}}
\homm_\B(M \otimes_\B \barB ,F) \otimes_\A M \otimes_\B \barB
\xrightarrow{\ev}
F.
\end{align} 
\end{small}

\item 
\label{item-M-bartimesB-and-HomA-M-adjunction}
The functors $M \bartimes_\B (-)$ and $\barhom_\A(M,-)$ are 
left and right adjoint functors $\BmodbarC \leftrightarrow \AmodbarC$. 
The unit and the counit of the adjunction are given by the maps 
\begin{small}
\begin{align}
\label{eqn-M-bartimes-adjunction-unit}
E \xrightarrow{\mlt} \barhom_\A(M,M \bartimes_\B E) \quad \quad \forall\;
E \in \BmodbarC, \\
\label{eqn-M-bartimes-adjunction-counit}
M \bartimes_\B \barhom_\A(M,F) \xrightarrow{\ev} F \quad \quad \forall\;
F \in \AmodbarC
\end{align}
in $\BmodbarC$ and $\AmodbarC$ which correspond to the 
in $\BmodC$ and $\AmodC$ maps
\begin{align}
\label{eqn-M-bartimes-adjunction-unit-underlying-map-in-Mod}
\barB \otimes_\B E \otimes_\C \barC
\xrightarrow{\mlt \otimes \tau}
\homm_\A(\barA \otimes_\A M, \barA \otimes_\A M \otimes_\B \barB \otimes_\B E), 
\xrightarrow{(\tau \otimes \id^{\otimes 3}) \circ (-) }
\homm_\A(\barA \otimes_\A M, M \otimes_\B \barB \otimes_\B E), 
\\
\label{eqn-M-bartimes-adjunction-counit-underlying-map-in-Mod}
\barA \otimes_\A M \otimes_\B \barB \otimes_\B \homm_\A(\barA \otimes_\A M, F) 
\otimes_\barC \barC
\xrightarrow{\id^{\otimes 2} \otimes \tau \otimes \id \otimes \tau} 
\barA \otimes_\A M \otimes_\B \homm_\A(\barA \otimes_\A M, F) 
\xrightarrow{\ev} 
F.
\end{align}
\end{small}
\end{enumerate}
\end{prps}
\begin{proof}
To prove the assertion \eqref{item-bartimesA-M-and-HomB-M-adjunction} it
suffices to show that for any $E \in \CmodbarA$ and $F \in \CmodbarB$ 
\begin{align}
\label{item-modbar-Tensor-Hom-adjucntion-F-FRF-F}
E \bartimes_\A M 
\xrightarrow{\mlt \bartimes \id}
\barhom_\B(M,E \bartimes_\A M) \bartimes_\A M 
\xrightarrow{\ev}
E \bartimes_\A M 
\\
\label{item-modbar-Tensor-Hom-adjucntion-R-RFR-R}
\barhom_\B(M,F) 
\xrightarrow{\mlt}
\barhom_\B(M, \barhom_\B(M,F) \bartimes_\A M) 
\xrightarrow{\ev \circ (-)}
\barhom_\B(M,F) 
\end{align}
are identity morphisms. We only demonstrate this for 
\eqref{item-modbar-Tensor-Hom-adjucntion-F-FRF-F}, as 
\eqref{item-modbar-Tensor-Hom-adjucntion-R-RFR-R} works out very
similarly. 

By definition of the composition in $\modbarB$, 
\eqref{item-modbar-Tensor-Hom-adjucntion-F-FRF-F} corresponds
to the $\modB$ map
\begin{small}
\begin{align*}
& \barC \otimes_\C E \otimes_\A \barA \otimes_\A M \otimes_\B \barB 
\xrightarrow{\Delta \otimes \id \otimes \Delta \otimes \id^{\otimes 2}} 
\barC \otimes_\C \barC \otimes_\C E \otimes_\A \barA \otimes_\A \barA \otimes_\A M \otimes_\B \barB  
\xrightarrow{\id \otimes \tau \otimes \mlt \otimes \id^{ \otimes 3}} 
\\
\rightarrow 
&\barC \otimes_\C \homm_\B(M \otimes_\B \barB, E \otimes_\A \barA \otimes_\A M
\otimes_\B \barB) \otimes_\A \barA \otimes_\A M \otimes_\B \barB  
\xrightarrow{\id \otimes \left( (\tau \otimes \id^{\otimes 3}) \circ (-) \right)
\otimes \id^{\otimes 3}}
\\
\rightarrow 
& \barC \otimes_\C \homm_\B(M \otimes_\B \barB, E \otimes_\A \barA \otimes_\A M) 
\otimes_\A \barA \otimes_\A M \otimes_\B \barB
\xrightarrow{\tau \otimes \id \otimes \tau \otimes \id^{\otimes 2}}
\\
\rightarrow 
& \homm_\B(M \otimes_\B \barB, E \otimes_\A \barA \otimes_\A M) \otimes_\A M \otimes_\B \barB
\xrightarrow{\ev}
E \otimes_\A \barA \otimes_\A M. 
\end{align*}
\end{small}
By functoriality of the tensor product the above composition equals
\begin{small}
\begin{align*}
& \barC \otimes_\C E \otimes_\A \barA \otimes_\A M \otimes_\B \barB 
\xrightarrow{\Delta \otimes \id \otimes \Delta \otimes \id^{\otimes 2}} 
\barC \otimes_\C \barC \otimes_\C E \otimes_\A \barA \otimes_\A \barA \otimes_\A M \otimes_\B \barB  
\xrightarrow{\tau \otimes \id^{\otimes 3} \otimes \tau \otimes \id^{\otimes 2}} 
\\
\rightarrow 
& \barC \otimes_\C E \otimes_\A \barA \otimes_\A M \otimes_\B \barB 
\xrightarrow{\tau \otimes \mlt \otimes \id^{ \otimes 2}} 
\homm_\B(M \otimes_\B \barB, E \otimes_\A \barA \otimes_\A M
\otimes_\B \barB) \otimes_\A M \otimes_\B \barB  
\xrightarrow{\left( (\id^{\otimes 3} \otimes \tau) \circ (-) \right) \otimes \id^{\otimes 2}}
\\
\rightarrow 
& \homm_\B(M \otimes_\B \barB, E \otimes_\A \barA \otimes_\A M) \otimes_\A M \otimes_\B \barB
\xrightarrow{\ev}
E \otimes_\A \barA \otimes_\A M. 
\end{align*}
\end{small}
By Prop.~\ref{prps-coalgebra-structure-on-the-bar-complex}  the maps 
$\tau \otimes \id$ and $\id \otimes \tau$ are left inverses to $\Delta$, thus the first two maps compose to $\id$. 
On the other hand, the last two maps compose to $(\id^{\otimes 3}
\otimes \tau) \circ \ev$. By the Tensor-Hom adjunction for 
$M \otimes_\B \barB$ the total composition is therefore 
$\tau \otimes \id^{\otimes 3} \otimes \tau$. 
The corresponding map in $\modbarB$ is $\id$, as desired. 

The assertion \eqref{item-M-bartimesB-and-HomA-M-adjunction} is
settled similarly. 
\end{proof}

It is worth writing out the maps defining the units and the counits of 
these two adjunctions explicitly. The compositions  
\eqref{eqn-bartimes-M-adjunction-unit-underlying-map-in-Mod}, 
\eqref{eqn-bartimes-M-adjunction-counit-underlying-map-in-Mod}, 
\eqref{eqn-M-bartimes-adjunction-unit-underlying-map-in-Mod}, 
and \eqref{eqn-M-bartimes-adjunction-counit-underlying-map-in-Mod}
are the maps 
\begin{align*}
c \otimes e \otimes a & \mapsto \left 
(\tau(c) e \otimes a \otimes - \right) \circ \left( \id \otimes \tau \right)
\\
c \otimes \alpha \otimes a \otimes m \otimes b 
&\mapsto 
\tau(c) \alpha(\tau(a).m \otimes b)
\\
b \otimes e \otimes c
& \mapsto 
\left 
((-1)^{\deg(-)\left(\deg(b) + \deg(e)\right)} (-) \otimes b \otimes e
\tau(c) \right) \circ \left( \tau \otimes \id \right)
\\
a \otimes m \otimes b \otimes \alpha \otimes c
&\mapsto 
(-1)^{\deg(\alpha)\left(\deg(a) + \deg(m) + \deg(b)\right)}\alpha(a
\otimes m.\tau(b))\tau(c).
\end{align*}

To sum up, we have a DG-enhancement framework which to every DG-category
$\A$ associates an enhancement $\modbarA$ of its derived category
$D(\A)$. These enhancements $\modbar$ admit genuinely adjoint 
(in each argument) bifunctors $(-) \bartimes_{\bullet} (-)$ and 
$\barhom_{\bullet}(-,-)$.

%

\subsection{On non-invertibility of the semi-free resolution $\A\bartimes M \rightarrow M$}
\label{section-on-non-invertibility-of-A-bartimes-M-to-M}

Recall the semi-free resolution 
$$ \A \inftimes_\A M
\xrightarrow{\eqref{eqn-natural-transformation-E-inftimesA-A-to-E}}
M
$$
discussed in
\S\ref{section-functorial-semi-free-resolution-for-noddinfhuA}. 
Consider moreover its right adjoint 
with respect to $\A \inftimes_\A (-)$:
$$ M
\longrightarrow 
\infhom_\A(\A,M).
$$

The category isomorphism $\Psi$ of 
Prop.~\ref{prps-modbarA-to-noddinfdgA-isomorphism} 
identifies these with the maps
\begin{align}
\label{eqn-A-bartimes-M-to-M}
\A \bartimes_\A M \longrightarrow M 
\\
\label{eqn-M-to-barhom-A-M}
M \longrightarrow \barhom_\A(\A,M) 
\end{align}
defined by the $\AmodB$ maps
\begin{align*}
\barA \otimes_\A \barA \otimes_\A M \otimes_\B \barB 
\xrightarrow{\tau \otimes \tau \otimes \id \otimes \tau} 
M 
\\
\barA \otimes_\A M \otimes_\B \barB \xrightarrow{\tau \otimes \id \otimes \tau}
M \xrightarrow{\quad\eqref{eqn-DG-M-to-hom-A-M-isomorphism}\quad} \homm_\A(\A,M) \xrightarrow{(-)\circ \tau}
\homm_\A(\barA,M).
\end{align*}
We therefore see that the $\AmodbarB$ maps 
$\eqref{eqn-A-bartimes-M-to-M}$
and 
$\eqref{eqn-M-to-barhom-A-M}$
are the analogues of 
the canonical $\AmodB$ isomorphisms 
$\A \otimes_\A M \xrightarrow{\eqref{eqn-DG-A-otimes-M-to-M-isomorphism}} M$ 
and 
$M \xrightarrow{\eqref{eqn-DG-M-to-hom-A-M-isomorphism}} \homm_\A(\A,M)$. 
Indeed, they induce the same isomorphisms $\A \ldertimes_\A M \simeq M$ and
$M \simeq \rder\homm_\A(\A,M)$ in the derived category $D(\AbimB)$
as \eqref{eqn-DG-A-otimes-M-to-M-isomorphism} and 
\eqref{eqn-DG-M-to-hom-A-M-isomorphism}. 

\begin{center}
\fbox{
\begin{minipage}{0.9\textwidth}
\em The biggest drawback of the categories $\modbar$ is that 
the maps $\eqref{eqn-A-bartimes-M-to-M}$ and $\eqref{eqn-M-to-barhom-A-M}$
are not themselves isomorphisms, 
like \eqref{eqn-DG-A-otimes-M-to-M-isomorphism} and 
\eqref{eqn-DG-M-to-hom-A-M-isomorphism}, but merely homotopy equivalences. 
\rm 
\end{minipage}
}
\end{center}
\vspace{0.1cm}

In this section, we show that this can be controlled. 
The maps $\eqref{eqn-A-bartimes-M-to-M}$ and $\eqref{eqn-M-to-barhom-A-M}$
have natural semi-inverses. These are genuine inverses on one side,
but only homotopy inverses on the other. However, the arising
higher homotopies are induced by endomorphisms of the bar complex
and thus independent of $M$. 

To put this into context, recall 
\cite[\S 3.7]{Drinfeld-DGQuotientsOfDGCategories}, 
\cite{Tabuada-UneStructureDeCategorieDeModelesDeQuillenSurLaCategorieDesDG-Categories}, 
\cite[Appendix A]{AnnoLogvinenko-SphericalDGFunctors} that for any
DG-category $\A$ and any objects $x,y \in \A$ we can (non-canonically) 
complete any homotopy equivalence 
\begin{equation}
\label{eqn-quiver-presentation-of-single-morphism-category}
\begin{tikzcd}
x
\ar[bend left=20]{rr}{\beta}
& &
y
\end{tikzcd}
\end{equation}
to the following system of morphisms and relations between them. 
The dotted arrows denote the morphisms of degree $-1$ and the dashed arrow
the morphism of degree $-2$:
\begin{equation}
\label{eqn-quiver-presentation-of-generating-cofibration-category}
\begin{minipage}[c][1in][c]{1.5in}
$d \theta_x = \alpha \circ \beta - \id_x,$ \\
$d \theta_y = \id_y - \beta \circ \alpha,$ \\
$d \alpha = d \beta = 0,$ \\
$d \phi = - \beta \circ \theta_x - \theta_y \circ \beta.$ 
\end{minipage}
\quad\quad
\begin{tikzcd}[column sep={2cm},row sep={1.5cm}] 
x
\ar[bend left=20]{rr}{\beta}
\ar[bend left=50, dashed]{rr}{\phi}
\ar[out=-150, in=150,loop,distance=6em, dotted]{}{\theta_x}
& &
y
\ar[bend left=20]{ll}{\alpha}
\ar[out=30, in=-30,loop,distance=6em, dotted]{}{\theta_y}
\end{tikzcd}
\end{equation}
In other words, we can find:
\begin{itemize}
\item a homotopy inverse $\alpha$ of $\beta$, 
\item a degree $-1$ homotopy $\theta_x$ from $\alpha \circ \beta$ to $\id_x$ 
\item a degree $-1$ homotopy $\theta_y$ from $\beta \circ \alpha$ to $\id_y$ 
\item a degree $-2$ homotopy $\phi$ from $\beta \circ \theta_x$ to 
$\theta_y \circ \beta$. 
\end{itemize}
The key assertion here is that we can choose 
$\theta_x$ and $\theta_y$ so that $\phi$ exists.  

It turns out that in the case of homotopy equivalences 
$\eqref{eqn-A-bartimes-M-to-M}$ and $\eqref{eqn-M-to-barhom-A-M}$
we can do quite a bit better than
\eqref{eqn-quiver-presentation-of-generating-cofibration-category}. 
Firstly, they admit natural one-sided inverses:

\begin{defn}
Let $M \in \AmodbarB$. Define the maps
\begin{align}
\label{eqn-M-to-A-bartimes-M}
M \longrightarrow \A \bartimes_\A M \\
\label{eqn-barhom-A-M-to-M}
\barhom_\A(\A,M) \longrightarrow M 
\end{align}
in $\AmodbarB$ by the $\AmodB$ maps
\begin{align*}
\barA \otimes_\A M \otimes_\B \barB \xrightarrow{\id \otimes \tau} \barA
\otimes_\A M
\\ 
\barA \otimes_\A \homm_\A(\barA,M) \otimes_\B \barB 
\xrightarrow{\id \otimes \tau}
\barA \otimes_\A \homm_\A(\barA,M)
\xrightarrow{\ev}
M.
\end{align*}
\end{defn}

It can be readily checked that $\eqref{eqn-M-to-A-bartimes-M}$ is
a right inverse to $\eqref{eqn-A-bartimes-M-to-M}$, while
$\eqref{eqn-barhom-A-M-to-M}$ is a left inverse
to $\eqref{eqn-M-to-barhom-A-M}$. We can apply these to give 
a more natural description of Tensor-Hom adjunction counits, 
and to show action maps to be instances of Tensor-Hom adjunction units: 

\begin{lemma}
\label{lemma-explicit-description-of-units-counits-of-Tensor-Hom-adjunction}
Let $\A$, $\B$, and $\C$ be DG-categories and let $M \in \AmodbarB$.  

The composition 
\begin{small}
\begin{align}
\barhom_\B(M,-) \bartimes_\A M 
\xrightarrow{\id \bartimes \eqref{eqn-M-to-barhom-A-M}} 
\barhom_\B(M,-) \bartimes_\A \barhom_\B(\A,M)
\xrightarrow{\composition} 
\barhom_\B(\A,-) 
\xrightarrow{\eqref{eqn-barhom-A-M-to-M}}
\id_{\CmodbarB}
\end{align}
\end{small}
is the counit of the $\left( (-) \bartimes_\A M, \ 
\barhom_\B(M,-)\right)$ adjunction. The counit of 
the $\left( M \otimes_\B (-), \ \barhom_\A(M,-) \right)$ adjunction 
admits an analogous description.  
\end{lemma}
\begin{proof}
Direct verification. 
\end{proof}

\begin{lemma} Let $\A$ and $\B$ be DG-categories and let 
$M \in \AmodbarB$. The compositions 
\begin{align*}
\A \xrightarrow{\mlt} \barhom_\B(M, \A \bartimes_\A M ) 
\xrightarrow{\eqref{eqn-A-bartimes-M-to-M}}
\barhom_\B(M, M) 
\\
\B \xrightarrow{\mlt} \barhom_\A(M,M \bartimes_\B \B)
\xrightarrow{\eqref{eqn-A-bartimes-M-to-M}}
\barhom_\A(M,M) 
\end{align*}
are the maps
\begin{align}
\A \xrightarrow{\action} \barhom_\B(M,M) \\
\B \xrightarrow{\action} \barhom_\A(M,M)
\end{align}
in $\AmodbarA$ and $\BmodbarB$ induced 
by the corresponding action maps. 
\end{lemma}
\begin{proof}
Direct verification. 
\end{proof}

In order to define degree $-1$ and $-2$ homotopies as per 
\eqref{eqn-quiver-presentation-of-generating-cofibration-category} 
we need to introduce certain natural endomorphisms of the bar complex. 
In Prop.~\ref{prps-coalgebra-structure-on-the-bar-complex} we have shown that
 $\tau \otimes \id$ and $\id \otimes \tau$ are left inverses to $\Delta$.  Since $\tau$
 is a quasi-isomorphism, the maps $\tau \otimes \id$ and $\id \otimes \tau$ are 
 also quasi-isomorphisms and thus homotopy equivalences. Hence so is $\Delta$, 
 and the following morphism, which composes to zero with $\Delta$, is a boundary:
\begin{align}
\barA \otimes_\A \barA \xrightarrow{ \tau \otimes \id - \id \otimes \tau } 
\barA.
\end{align}
In \S\ref{section-the-bar-complex} we have produced, out of a natural algebra structure on the extended bar complex, a degree $-1$ map 
$$ \mu\colon \barA \otimes_\A \barA \mapsto \barA $$
which lifts this boundary, i.e. $d \mu = \tau \otimes \id - \id \otimes \tau$, and which satisfies $ \tau \circ \mu = 0$.  We next look at the compositions of this lift $\mu$ with the comultiplication $\Delta$.  For this, we need the following definition: 

\begin{defn}
\label{defn-lambda_k-the-insertion-of-k-1s-map}
Let $\A$ be a DG-category and let $k \in \mathbb{Z}_{\geq 0}$. 
Define the \em insertion of $k$ 1s \rm map  
\begin{align}
\label{eqn-lambda_k-the-insertion-of-k-1s-map}
\lambda_{k}\colon \barA \rightarrow \barA 
\end{align}
by the degree $-k$ map from the twisted complex 
\eqref{eqn-bar-complex} to itself which sends any 
$a_1 \otimes \dots \otimes a_n \in \A^{\otimes n}$ to 
\begin{align*}
\sum_{i_1 + i_2 + \dots + i_{k+1} = n} \;
&(-1)^{k i_1 + (k-1) i_2 + \dots + 1 i_{k} + k} \\ 
& a_1 \otimes \dots \otimes a_{i_1} \otimes 1 \otimes 
a_{i_1 + 1} \otimes \dots \otimes a_{i_1 + i_2} \otimes 1 \otimes 
\cdots\cdots\cdots
\otimes 1 \otimes a_{i_1 + \dots + i_k + 1} \otimes \dots \otimes a_n.
\end{align*}
\end{defn}

We have established in \S\ref{section-the-bar-complex}
that $\barA$ has a natural structure of coalgebra.
In particular, its comultiplication map $\Delta$ is coassociative. 
We therefore  write $\Delta^k$ for the unique
map $\barA \rightarrow \barA^{\otimes(k+1)}$ which 
is a composition of $k$ applications of $\Delta$. 

\begin{prps}
\label{prps-properties-of-mu-and-lambda_k}
\begin{enumerate} 
\item 
\label{item-composition-Delta-circ-mu}
The composition $\barA \otimes_\A \barA \xrightarrow{\mu} 
\barA \xrightarrow{\Delta} \barA \otimes_\A \barA$ equals the sum
$$
(\id\otimes\mu)\circ(\Delta\otimes\id) +
(\mu\otimes\id)\circ(\id\otimes\Delta). $$
\item 
\label{item-composition-mu-circ-Delta}
The composition $\barA \xrightarrow{\Delta} \barA \otimes_\A \barA
\xrightarrow{\mu} \barA$ is the map $\lambda_1$. 
\item 
\label{item-composition-mu^k-circ-Delta^k-left}
For any $k \geq 0$ the map $\lambda_k$ equals the composition 
\begin{align}
\barA \xrightarrow{\Delta^k} 
\barA^{\otimes(k + 1)}
\xrightarrow{\id^{\otimes(k-1)} \otimes \mu} 
\barA^{\otimes k}
\xrightarrow{\id^{\otimes(k-2)} \otimes \mu} 
\barA^{\otimes (k-1)}
\rightarrow 
\dots
\rightarrow
\barA \otimes_\A \barA 
\xrightarrow{\mu} 
\barA. 
\end{align}
\item 
\label{item-composition-mu^k-circ-Delta^k-right}
For any $k \geq 0$ the map $(-1)^{\frac{k(k-1)}{2}}\lambda_k$ equals the composition 
\begin{align}
\barA \xrightarrow{\Delta^k} 
\barA^{\otimes(k + 1)}
\xrightarrow{\mu\otimes\id^{\otimes(k-1)}} 
\barA^{\otimes k}
\xrightarrow{\mu\otimes\id^{\otimes(k-2)}} 
\barA^{\otimes (k-1)}
\rightarrow 
\dots
\rightarrow
\barA \otimes_\A \barA 
\xrightarrow{\mu} 
\barA.
\end{align}
\item
\label{item-recursive-formula-for-lambda-k}
For any $k \geq 1$ the map $\lambda_k$ equals the compositions 
$\lambda_k = \mu \circ (\id \otimes \lambda_{k-1}) \circ \Delta=(-1)^{k+1} \mu \circ (\lambda_{k-1} \otimes \id) \circ \Delta$. 
\item 
\label{item-differential-of-the-map-lambda-k}
For any $k \geq 1$ the map $d\lambda_k$ equals 
$\lambda_{k-1}$ if $k$ is even, and $0$ if $k$ is odd. 
\end{enumerate}
\end{prps}
\begin{proof}
As explained in \S\ref{section-the-bar-complex} the bimodules $\barA$ and $\barA \otimes_\A \barA$ 
can be identified with the convolutions of the twisted complexes 
\eqref{eqn-bar-complex}
and
\eqref{eqn-twisted-complex-barA-otimes-barA}. 
The maps $\tau$, $\mu$ and $\delta$ were then defined on the level of these twisted complexes in 
$\pretriag(\AmodA)$. We therefore perform all the computations in this proof with these maps  
in $\pretriag(\AmodA)$. The reason for doing this, as explained in  \S\ref{section-the-bar-complex}, 
is that the signs in the formulas become significantly simpler on the level of twisted complexes. 

\eqref{item-composition-Delta-circ-mu}: \;
For any $\left(a_1\otimes\ldots\otimes a_n\right)
\otimes_\A \left( b_1\otimes\ldots\otimes b_m \right)$ 
in the twisted complex \eqref{eqn-twisted-complex-barA-otimes-barA}
its image under $\Delta \circ \mu$ is:
\begin{align*}
&(-1)^n\sum\limits_{i=1}^{n-1} (a_1\otimes\ldots\otimes a_i\otimes 1)\otimes (1\otimes\ldots\otimes a_nb_1 \otimes\ldots\otimes b_m) +\\
+ &(-1)^n\sum\limits_{i=1}^{m-1} (a_1\otimes\ldots\otimes a_nb_1\otimes\ldots\otimes b_i\otimes 1)\otimes (1\otimes b_{i+1}\otimes\ldots\otimes b_m).
\end{align*}
Its image under $(\id \otimes \mu) \circ (\Delta \otimes \id)$ is 
\begin{align*}
\sum\limits_{i=1}^{n-1} (-1)^{i+1}(-1)^{n-i+1}(a_1\otimes\ldots\otimes a_i\otimes 1)\otimes (1\otimes\ldots\otimes a_nb_1 \otimes\ldots\otimes b_m)  
\end{align*}
where the sign $(-1)^{i+1}$ comes from the definition of a tensor
product of two maps applied to $\id \otimes \mu$, while the 
sign $(-1)^{n-i+1}$ from comes the definition of $\mu$. 
Finally, its image under 
$(\mu \otimes \id) \circ (\id \otimes \Delta)$ is
\begin{align*}
\sum\limits_{i=1}^{m-1} (-1)^n (a_1\otimes\ldots\otimes a_nb_1\otimes\ldots\otimes b_i\otimes 1)\otimes (1\otimes b_{i+1}\otimes\ldots\otimes b_m).
\end{align*}
where the sign $(-1)^n$ comes from the definition of $\mu$. 
The desired result now follows. 

\eqref{item-composition-mu-circ-Delta}: \; For any $a_1 \otimes
\dots \otimes a_n \in \barA$ we have 
\begin{align*}
& \mu \left(\Delta\left(a_1 \otimes \dots \otimes a_n\right)\right) = \\
= \quad &\mu\left(\sum_{i=1}^{n-1} 
\left(a_1 \otimes \dots \otimes a_i \otimes 1\right)
\otimes_\A 
\left(1 \otimes a_{i+1} \otimes \dots \otimes a_n\right)
\right) = \\
=\quad  &\sum_{i=1}^{n-1} (-1)^{i + 1}
a_1 \otimes \dots \otimes a_i \otimes 1
\otimes a_{i+1} \otimes \dots \otimes a_n = \lambda_1(a_1 \otimes
\dots \otimes a_n). 
\end{align*}

\eqref{item-composition-mu^k-circ-Delta^k-left}, \eqref{item-composition-mu^k-circ-Delta^k-right}: \; 
This follows by a direct computation analogous 
to the one for \eqref{item-composition-mu-circ-Delta}. 

\eqref{item-recursive-formula-for-lambda-k}: \;
By the assertion \eqref{item-composition-mu^k-circ-Delta^k-left} we have 
\begin{align*}
\lambda_k  
= & \mu \circ (\id \otimes \mu) \circ \dots \circ
(\id^{\otimes(k-1)} \otimes \mu) \circ \Delta^k  =  \\
= & 
\mu 
\circ 
\left(\id \otimes \mu\right)
\circ 
\dots 
\circ
\left(\id^{\otimes(k-1)} \otimes \mu\right) 
\circ 
\left(\id^{\otimes(k-1)} \otimes \Delta\right)
\circ 
\dots 
\circ 
\left(\id \otimes \Delta\right)
\circ 
\Delta = 
\\
= 
&
\mu 
\circ 
\Big(
\id \otimes 
\Bigg(
\mu
\circ 
\dots 
\circ
\left(\id^{\otimes(k-2)} \otimes \mu\right) 
\circ 
\left(\id^{\otimes(k-2)} \otimes \Delta\right)
\circ 
\dots 
\circ 
\Delta
\Bigg)
\Big)
\circ 
\Delta = \\
= & \mu \circ \Big( \id \otimes \lambda_{k-1} \Big) \circ \Delta. 
\end{align*}
The second part is proven similarly using the assertion \eqref{item-composition-mu^k-circ-Delta^k-right}.

\eqref{item-differential-of-the-map-lambda-k}:\;
By the assertion 
\eqref{item-composition-mu-circ-Delta}
we have 
\begin{align*}
d\lambda_1 = d(\mu \circ \Delta) = (d\mu) \circ \Delta 
= (\tau \otimes \id - \id \otimes \tau ) \circ \Delta = \id - \id =
0. 
\end{align*}
Suppose now we have proved our claim for $k = n > 1$. Then 
by the assertion 
\eqref{item-recursive-formula-for-lambda-k}
we have 
\begin{align*}
d\lambda_{n+1} = &d\Big(\mu \circ (\id \otimes \lambda_n) \circ \Delta\Big) 
= d\mu \circ (\id \otimes \lambda_n) \circ \Delta - 
\mu \circ (\id \otimes d\lambda_n) \circ \Delta = \\
= &\Big(\tau \otimes \id - \id \otimes \tau\Big) 
\circ 
\Big(\id \otimes \left( \mu \circ (\id \otimes \lambda_{n-1}) \circ \Delta
\right)\Big)
\circ \Delta
- \mu \circ \Big(\id \otimes d\lambda_n\Big) \circ \Delta = 
\\ 
=
& \Big( \mu \circ (\id \otimes \lambda_{n-1}) \circ \Delta \Big)
\circ \Big(\tau \otimes \id\Big) \circ \Delta
-
\Big(\id \otimes \left(\tau \circ \mu \circ (\id \otimes
\lambda_{n-1}) \circ \Delta \right)\Big) \circ \Delta
- \mu \circ \Big(\id \otimes d\lambda_n\Big) \circ \Delta. 
\end{align*}
Since $\tau \circ \mu = 0$ by 
the Lemma \ref{lemma-d-mu-equals-tau-id-minus-id-tau}\eqref{item-composition-tau-circ-mu} and since 
$(\tau \otimes \id) \circ \Delta = \id$ the above further equals
\begin{align*}
\mu \circ (\id \otimes \lambda_{n-1}) \circ \Delta 
- \mu \circ (\id \otimes d\lambda_n) \circ \Delta. 
\end{align*}
Since this is zero when $d\lambda_n = \lambda_{n-1}$ and $\lambda_n$
when $d\lambda_n = 0$, the desired assertion for $k = n+1$ follows. 
\end{proof}

Having established these properties of the maps $\mu$ and $\lambda_k$, 
we can now proceed to our main objective: 

\begin{defn}
\label{defn-alpha-beta_k-theta}
Let $\A$ and $\B$ be DG-categories and let $M \in \AmodbarB$. 
Define the maps 
\begin{align}
\label{eqn-beta_k-M-to-A-bartimes-M-map}
M &\xrightarrow{\beta_k} \A \bartimes_\A M 
\\
\label{eqn-theta-degree-minus-1-A-bartimes-M-to-A-bartimes-M-map}
\A \bartimes_\A M & \xrightarrow{\theta} \A \bartimes_\A M 
\\
\label{eqn-alpha-A-bartimes-M-to-A-map}
\A \bartimes_\A M &\xrightarrow{\alpha} M
\end{align}
of degree $-k$, $-1$, and $0$, respectively, in $\AmodbarB$ 
by the corresponding $\AmodB$ maps
\begin{align*}
\barA \otimes_\A \left( M \right) \otimes_\B \barB
\;&\xrightarrow{ \lambda_k \otimes \id \otimes \tau }\;
\barA \otimes_\A M \\
\barA \otimes_\A \left( \barA \otimes_\A M \right) \otimes_\B \barB
\;&\xrightarrow{\mu \otimes \id \otimes \tau}\;
\barA \otimes_\A M
\\
\barA \otimes_\A \left( \barA \otimes_\A M \right) \otimes_\B \barB
\;&\xrightarrow{\tau \otimes \tau \otimes \id \otimes \tau}\;
M.
\end{align*}
\bf NB: \rm  
The map $\alpha$ is the canonical map 
$\eqref{eqn-A-bartimes-M-to-M}$, while $\beta_0$ is its left inverse
$\eqref{eqn-M-to-A-bartimes-M}$. 
\end{defn}

\begin{prps}
\label{prop-alpha-beta_k-theta-relations}
Let $\A$ and $\B$ be DG-categories and let $M \in \AmodbarB$. 
We have:
\begin{enumerate}
\item \label{item-d(theta)} 
$d \theta = \id - \beta_0 \circ \alpha$. 
\item \label{item-alpha-circ-beta_0}
$ 0 = \alpha \circ \beta_0 - \id$. 
\item \label{item-beta_k}
$\beta_k = \theta^k \circ \beta_0$. 
\item \label{item-alpha-circ-beta_k}
For any $k \geq 1$ we have $\alpha \circ \beta_k = 0$. 
\item \label{item-d(beta_k)}
For any $k \geq 1$ the map $d \beta_{k}$ equals $0$ when $k$ is odd,
and $\beta_{k-1}$ when $k$ is even. 
\end{enumerate}
\end{prps}
\begin{proof}
These follow immediately from the properties of the maps 
$\mu$ and $\lambda_k$ established in 
Lemma \ref{lemma-d-mu-equals-tau-id-minus-id-tau}
and
Prop.~\ref{prps-properties-of-mu-and-lambda_k}. 
For example, by the definition of the map $\theta$ and by  
Lemma \ref{lemma-d-mu-equals-tau-id-minus-id-tau} the map $d\theta$ 
corresponds to 
\begin{align*}
d\left(\mu \otimes \id \otimes \tau\right) = \left(d\mu\right) \otimes \id \otimes \tau =  \tau \otimes \id  \otimes \id \otimes \tau - \id \otimes \tau \otimes \id \otimes \tau.
\end{align*}

The map 
\begin{align*}
	\barA \otimes_\A \left( \barA \otimes_\A M \right) \otimes_\B \barB
\;\xrightarrow{\tau \otimes \id \otimes \id \otimes \tau}\;
\barA \otimes_\A M 
\end{align*}
is the map in $\AmodB$ which corresponds to the 
$\AmodbarB$ identity map
$\A \bartimes_\A M \xrightarrow{\;\id\;} \A \bartimes_\A M$.
On the other hand, 
the map $\id \otimes \tau \otimes \id \otimes \tau$ is readily checked 
to be the map in $\AmodB$ which corresponds to 
$\beta_0 \circ \alpha$ in $\AmodbarB$. 
The assertion \eqref{item-d(theta)} follows.
\end{proof}

\begin{prps}
\label{prps-alpha-beta-theta}
Let $\A$ and $\B$ be DG-categories and let $M \in \AmodbarB$. 
Let $\alpha$, $\beta_0$, and $\theta$ be the maps, respectively, 
introduced in Definition \ref{defn-alpha-beta_k-theta}: 
\begin{equation}
\label{eqn-quiver-presentation-of-alpha-beta-theta}
\begin{tikzcd}[column sep={2cm},row sep={1.5cm}] 
M
\ar[bend left=20]{rr}{\beta_0}
& &
\A \bartimes_\A M
\ar[bend left=20]{ll}{\alpha}
\ar[out=30, in=-30,loop,distance=6em, dotted]{}{\theta}
\end{tikzcd}
\end{equation}

The sub-DG-category of $\AmodbarB$ generated by $\alpha, \beta_0$ and 
$\theta$ is the free DG-category generated by those elements modulo the
following relations:
\begin{enumerate}
\item \label{item-d(alpha)-and-d(beta)-are-zero}
$d\alpha = d\beta_0 = 0$, 
\item $d\theta = \id - \beta_0 \circ \alpha$, 
\item $0 =  \alpha \circ \beta_0 - \id$, 
\item \label{item-alpha-theta-is-zero}
$\alpha \circ \theta = 0$. 
\end{enumerate}
\end{prps}

\bf NB: \rm The relations in Prop.~\ref{prps-alpha-beta-theta}
can be obtained by taking those in 
\eqref{eqn-quiver-presentation-of-generating-cofibration-category}
and demanding further that $\theta_x = 0$, $\alpha \circ \theta_y = 0$, 
and $\phi = - \theta_y^2 \circ \beta$.  
\begin{proof}
It follows from 
Prop.~\ref{prop-alpha-beta_k-theta-relations}
that the relations
\eqref{item-d(alpha)-and-d(beta)-are-zero}-\eqref{item-alpha-theta-is-zero}
do hold. It remains to show that no other relations are necessary. 
This is equivalent to showing that:
\begin{enumerate}
\item \label{item-homs-from-M-to-A-bartimes-M}
$\theta^k \beta_0$ for $k \geq 0$ are linearly independent elements of $\barhom_{\AbimB}(M, \A \bartimes_\A M)$. 
\item \label{item-homs-from-A-bartimes-M-to-A-bartimes-M}
$\id$, $\theta^k$, $\theta^k \circ \beta_0 \circ \alpha$ for all $k \geq 0$ 
are linearly independent elements of 
$\barhom_{\AbimB}(\A \bartimes_\A M, \A \bartimes_\A M)$. 
\end{enumerate}
For \eqref{item-homs-from-M-to-A-bartimes-M} first note that 
each $\theta^k \circ \beta_0$ is of degree $-k$. For degree reasons, it is
enough therefore to show that each is non-zero.  For this, we describe the maps 
$\theta^k \circ \beta_0$ explicitly. By Prop.~\ref{prop-alpha-beta_k-theta-relations} 
we have $\theta^k \circ \beta_0 = \beta_k$. Thus 
it is induced by the map $\barA \xrightarrow{\lambda_k} \barA$
of Definition \ref{defn-lambda_k-the-insertion-of-k-1s-map}
which inserts $k$ $1s$ into an element of $\barA$.

For \eqref{item-homs-from-A-bartimes-M-to-A-bartimes-M}, we similarly
note that for each $k \geq 1$ the maps $\theta^k$ and 
$\theta^k \circ \beta_0 \circ \alpha$ have degree $-k$.  
For degree resons, it is enough to show for each $k \geq 0$ 
that the maps $\theta^k$ and $\theta^k \circ \beta_0 \circ \alpha$
are linearly independent. This is clear since they are induced by the maps
$$ \barA \otimes_\A \barA 
\xrightarrow{(-1)^{k-1} \mu \circ (\lambda_{k-1} \otimes \id)}
\barA, $$
$$ \barA \otimes_\A \barA \xrightarrow{\lambda_k \otimes \tau} \barA, $$
respectively.
\end{proof}

Similarly, we have: 

\begin{defn}
\label{defn-gamma-delta_k-kappa}
Let $\A$ and $\B$ be DG-categories and let $M \in \AmodbarB$. 
Define the maps 
\begin{align}
\label{eqn-delta_k-barhom-A-M-to-M}
\barhom_\A(\A,M) &\xrightarrow{\delta_k} M 
\\
\label{eqn-kappa-degree-minus-1-barhom-A-M-barhom-A-M-map}
\barhom_\A(\A,M) & \xrightarrow{\kappa} \barhom_\A(\A,M) 
\\
\label{eqn-gamma-M-to-barhom-A-M-map}
M &\xrightarrow{\gamma} \barhom_\A(\A,M)
\end{align}
of degree $-k$, $-1$, and $0$, respectively, in $\AmodbarB$ 
by the corresponding $\AmodB$ maps
\begin{align*}
\barA \otimes_\A \left( \homm_\A\left(\barA, M\right) \right) \otimes_\B \barB
\;
\xrightarrow{
\ev \circ \left( \lambda_k \otimes \id \otimes \tau \right) 
} \;
M
\\
\barA \otimes_\A \left( \homm_\A\left(\barA, M\right) \right) \otimes_\B \barB
\;
\xrightarrow{
\composition \circ 
\left(  \left(\left(\mu \circ (-)\right)  \circ \mlt \right) \otimes \id
\otimes \tau \right)
} \;
\homm_\A\left(\barA, M\right)
\\
\barA \otimes_\A \left( M\right) \otimes_\B \barB
\xrightarrow{ 
\left( (-) \circ \tau \right) 
\circ \left( \tau \otimes \eqref{eqn-DG-M-to-hom-A-M-isomorphism} \otimes
\tau \right) }
\homm_\A\left(\barA, M\right) \end{align*}
where $\left(\mu \circ (-)\right)  \circ \mlt$ denotes the composition 
$\barA \xrightarrow{\mlt} \homm_\A(\barA, \barA \otimes_\A \barA) 
\xrightarrow{\mu \circ (-)} \homm_\A(\barA, \barA)$. 
\end{defn}

\bf NB: \rm The map $\gamma$ is the canonical map
$\eqref{eqn-M-to-barhom-A-M}$, while $\delta_0$ is its left inverse
$\eqref{eqn-barhom-A-M-to-M}$. 

\begin{prps}
\label{prop-gamma-delta_k-kappa-relations}
Let $\A$ and $\B$ be DG-categories and let $M \in \AmodbarB$. 
We have:
\begin{enumerate}
\item \label{item-d(kappa)} 
$d \kappa = \gamma \circ \delta_0 - \id$. 
\item \label{item-delta_0-circ-gamma}
$ 0 =  \id - \delta_0 \circ \gamma$. 
\item \label{item-delta_k}
$\delta_k = \delta_0 \circ \kappa^k$. 
\item \label{item-delta_k-circ-gamma}
For any $k \geq 1$ we have $\delta_k \circ \gamma = 0$. 
\item \label{item-d(delta_k)}
For any $k \geq 1$ the map $d \delta_{k}$ equals $0$ when $k$ is odd,
and $\delta_{k-1}$ when $k$ is even. 
\end{enumerate}
\end{prps}
\begin{proof}
Analogous to the proof of
Prop.~\ref{prop-alpha-beta_k-theta-relations}. 
\end{proof}

\begin{prps}
\label{prps-gamma-delta-kappa}
Let $\A$ and $\B$ be DG-categories and let $M \in \AmodbarB$. 
Let $\gamma$, $\delta_0$, and $\kappa$ be the maps, respectively, 
introduced in Definition \ref{defn-gamma-delta_k-kappa}: 
\begin{equation}
\label{eqn-quiver-presentation-of-gamma-delta-kappa}
\begin{tikzcd}[column sep={2cm},row sep={1.5cm}] 
\barhom_\A(\A,M) 
\ar[bend left=20]{rr}{\delta_0}
\ar[out=-150, in=150,loop,distance=6em, dotted]{}{\kappa}
& &
M
\ar[bend left=20]{ll}{\gamma}
\end{tikzcd}
\end{equation}

The sub-DG-category of $\AmodbarB$ generated by $\gamma$, $\delta_0$, and 
$\kappa$ is the free DG-category generated by those elements modulo the
following relations:
\begin{enumerate}
\item \label{item-d(gamma)-and-d(delta)-are-zero}
$d\gamma = d\delta_0 = 0$, 
\item $d\kappa = \gamma \circ \delta_0 - \id$, 
\item $0 = \id - \delta_0 \circ \gamma$, 
\item \label{item-kappa-gamma-is-zero}
$\kappa \circ \gamma = 0$. 
\end{enumerate}
\end{prps}
\begin{proof}
Analogous to the proof of Prop.~\ref{prps-alpha-beta-theta}.
\end{proof}

Finally, the results in this section have been related so far to
the left action of $\A$ on $M$. For the right action of $\B$ on $M$ we need 
to define the maps as follows: the maps
\begin{align}
\label{eqn-beta_k-M-to-M-bartimes-B-map}
M &\xrightarrow{\beta_k}  M \bartimes_\B \B 
\\
\label{eqn-theta-degree-minus-1-M-bartimes-B-to--M-bartimes-B-map}
M \bartimes_\B \B & \xrightarrow{\theta} M \bartimes_\B \B
\\
\label{eqn-alpha-M-bartimes-B-to-M-map}
M \bartimes_\B \B &\xrightarrow{\alpha} M
\end{align}
of degree $-k$, $-1$, and $0$, are defined, respectively, in $\AmodbarB$ 
by the corresponding $\AmodB$ maps
\begin{align*}
\barA \otimes_\A \left( M \right) \otimes_\B \barB
\;&\xrightarrow{ (-1)^{\frac{k(k+1)}{2}} \tau \otimes \id \otimes \lambda_k }\;
M \otimes_\B \barB \\
\barA \otimes_\A \left(  M \otimes_\B \barB \right) \otimes_\B \barB
\;&\xrightarrow{-\tau \otimes \id \otimes \mu}\;
M \otimes_\B \barB
\\
\barA \otimes_\A \left(   M \otimes_\B \barB \right) \otimes_\B \barB
\;&\xrightarrow{\tau \otimes \id \otimes \tau \otimes \tau}\;
M,
\end{align*}
and the maps
\begin{align}
\label{eqn-delta_k-barhom-B-M-to-M}
\barhom_\B(\B,M) &\xrightarrow{\delta_k} M 
\\
\label{eqn-kappa-degree-minus-1-barhom-B-M-barhom-B-M-map}
\barhom_\B(\B,M) & \xrightarrow{\kappa} \barhom_\B(\B,M) 
\\
\label{eqn-gamma-M-to-barhom-B-M-map}
M &\xrightarrow{\gamma} \barhom_\B(\B,M)
\end{align}
of degree $-k$, $-1$, and $0$, are defined, respectively, in $\AmodbarB$ 
by the corresponding $\AmodB$ maps
\begin{align*}
\barA \otimes_\A \left( \homm_\B\left(\barB, M\right) \right) \otimes_\B \barB
\;
\xrightarrow{
(-1)^{\frac{k(k+1)}{2}} \ev \circ \left(\tau \otimes \id \otimes \lambda_k \right) 
} \;
M
\\
\barA \otimes_\A \left( \homm_\B\left(\barB, M\right) \right) \otimes_\B \barB
\;
\xrightarrow{
-\composition \circ 
\left(  \tau \otimes \id
\otimes  \left(\left(\mu \circ (-)\right)  \circ \mlt \right) \right)
} \;
\homm_\B\left(\barB, M\right)
\\
\barA \otimes_\A \left( M\right) \otimes_\B \barB
\xrightarrow{ 
\left( (-) \circ \tau \right) 
\circ \left( \tau \otimes \eqref{eqn-DG-M-to-hom-A-M-isomorphism} \otimes
\tau \right) }
\homm_\B\left(\barB, M\right) .
\end{align*}
Then Propositions \ref{prop-alpha-beta_k-theta-relations} and \ref{prps-alpha-beta-theta} hold for these versions of the maps $\alpha$, $\beta_k$, and $\theta$, 
and Propositions \ref{prop-gamma-delta_k-kappa-relations} and \ref{prps-gamma-delta-kappa} hold for these versions of the maps $\gamma$, $\delta_k$, and $\kappa$,
as well as for their left counterparts.

Setting $M$ to be $\B$ in 
$$ \B \bartimes_\B M \xrightarrow{\alpha_M} M $$
$$ M \bartimes_\B \B \xrightarrow{\alpha_M} M $$
we obtain two maps
$ \B \bartimes_\B \B \rightarrow \B$.
They are, in fact, equal since they both 
correspond to the $\BmodB$ map
\begin{align}
\label{eqn-alpha^l_B-and-alpha^r_B}
\barB \otimes_\B \barB \otimes_\B \barB \xrightarrow{\tau^{\otimes 3}} \B.
\end{align}
We generally denote this map by $\alpha_\B$, but occasionally use $\alpha^l_\B$ and $\alpha^r_\B$
to stress it to be the instance of the former or of the latter map, respectively. Similarly for the analogous 
map  $\A \bartimes_\A  \A \rightarrow \A$ which we denote by $\alpha_\A$, or occasionally $\alpha^l_\A$ or 
$\alpha^r_\A$.

On the other hand,  setting $M$ to be $\B$ in 
$$  M \xrightarrow{\beta_M} \B \bartimes_\B M$$
$$ M  \xrightarrow{\beta_M} M \bartimes_\B \B$$
we obtain two maps
$\B \rightarrow \B \bartimes_\B \B$ which are not the same and which we denote by $\beta^l_\B$ and $\beta^r_\B$, respectively. They are, however, homotopic:
\begin{defn}
\label{defn-beta^r_B-minus-beta^l_B-boundary-lift}
Let $\B$ be a DG-category. Define the degree $-1$ map 
\begin{align}
\pi_\B\colon \B \rightarrow \B \bartimes_\B \B  
\end{align}
in $\BmodbarB$ to be the map corresponding to the $\BmodB$ map $\barB \otimes \barB \xrightarrow{\mu} \barB$. 
\end{defn}
\begin{lemma}
Let $\B$ be a DG-category. We have 
$$ d \pi_\B = \beta^r_\B - \beta^l_\B.$$	
\end{lemma} 
\begin{proof}
The maps $\beta^l_\B$ and $\beta^r_\B$ correspond to $\BmodB$ maps $\barB \otimes_\B \barB \rightarrow \barB$ 
given by $\id \otimes \tau$ and $\tau \otimes \id$.  The desired assertion 
now follows from Lemma $\ref{lemma-d-mu-equals-tau-id-minus-id-tau}\eqref{item-d-mu-equals-tau-id-minus-id-tau}$
where the map $\tau \otimes \id - \id \otimes \tau$ was established 
to be the differential of the degree $-1$ map $\mu$. 
\end{proof}

\subsection{Functoriality of $\alpha$ and $\beta$}

We next consider the functorial properties of the maps $\alpha$ and
$\beta$. The latter give rise to genuine natural transformations 
$\id \rightarrow \A \bartimes_\A (-)$ and $\id \rightarrow (-) \bartimes_\B \B$, 
while the former --- only to homotopy ones. We also compute
the commutators of the corresponding squares for the non-closed maps $\theta$: 

\begin{prps}
\label{prps-alpha-and-tensoring-f-with-identity}
Let  $\A$ and $\B$ be DG-categories and let $M \xrightarrow{f} N$ be a morphism in $\AmodbarB$. 
\begin{enumerate}
\item 
\label{item-compositions-of-id-bartimes-f-and-f-bartimes-id-with-alpha-and-beta}
The compositions
\begin{align}
M \xrightarrow{\beta_M} \A \bartimes_\A  M \xrightarrow{\id \bartimes f} \A \bartimes_\A N \xrightarrow{\alpha_N} N \\
M \xrightarrow{\beta_M} M \bartimes_\B \B \xrightarrow{f \bartimes \id} N \bartimes_\B \B \xrightarrow{\alpha_N} N
\end{align}
both equal $M \xrightarrow{f} N$.

\item 
\label{item-the-commutators-of-f-and-alpha}
The squares 
\begin{equation}
\label{eqn-the-commutative-squares-for-f-and-alpha}
\begin{tikzcd}
\A \bartimes_\A M
\ar{r}{\id \bartimes f}
\ar{d}{\alpha_M}
&
\A \bartimes_\A N
\ar{d}{\alpha_N}
\\
M
\ar{r}{f}
&
N
\end{tikzcd}
\quad \text{ and } \quad 
\begin{tikzcd}
M \bartimes_\B \B
\ar{r}{f \bartimes \id}
\ar{d}{\alpha_M}
&
N \bartimes_\B \B
\ar{d}{\alpha_N}
\\
M
\ar{r}{f}
&
N
\end{tikzcd}
\end{equation}
have the commutators
\begin{align}
\label{eqn-the-commutator-of-f-and-alpha-A}
\alpha_N  \circ (\id \bartimes f) - f \circ \alpha_ M =
(-1)^{\deg(f)} \left( d\bigl(\alpha_N \circ (\id \bartimes f) \circ \theta_M\bigr) - 
\alpha_N \circ (\id \bartimes df) \circ \theta_M\right),
\\
\label{eqn-the-commutator-of-f-and-alpha-B}
\alpha_N  \circ (f \bartimes \id) - f \circ \alpha_ M = (-1)^{\deg(f)} \left( d\bigl(\alpha_N \circ (f \bartimes \id) \circ \theta_M\bigr) - 
\alpha_N \circ (df \bartimes \id) \circ \theta_M\right).
\end{align}
\item 
\label{item-f-and-beta-commute}
The following squares commute:
\begin{equation}
\label{eqn-commutative-squares-for-f-and-beta}
\begin{tikzcd}
M
\ar{r}{f}
\ar{d}{\beta_M}
&
N
\ar{d}{\beta_N}
\\
\A \bartimes_\A M
\ar{r}{\id \bartimes f}
&
\A \bartimes_\A N
\end{tikzcd}
\quad \text{ and } \quad
\begin{tikzcd}
M
\ar{r}{f}
\ar{d}{\beta_M}
&
N
\ar{d}{\beta_N}
\\
M \bartimes_\B \B
\ar{r}{f \bartimes \id}
&
N \bartimes_\B \B. 
\end{tikzcd}
\end{equation}

\item 
\label{item-the-commutators-of-f-and-theta}
The squares 
\begin{equation}
\label{eqn-the-commutative-squares-for-f-and-theta}
\begin{tikzcd}
\A \bartimes_\A M
\ar{r}{\id \bartimes f}
\ar{d}{\theta_M}
&
\A \bartimes_\A N
\ar{d}{\theta_N}
\\
\A \bartimes_\A M
\ar{r}{\id \bartimes f}
&
\A \bartimes_\A N
\end{tikzcd}
\quad \text{ and } \quad 
\begin{tikzcd}
M \bartimes_\B \B
\ar{r}{f \bartimes \id}
\ar{d}{\theta_M}
&
N \bartimes_\B \B
\ar{d}{\theta_N}
\\
M \bartimes_\B \B
\ar{r}{f \bartimes \id}
&
N \bartimes_\B \B
\end{tikzcd}
\end{equation}
have the commutators
\begin{align}
\label{eqn-the-commutator-of-f-and-theta-A}
\theta_N  \circ (\id \bartimes f) - (\id \bartimes f) \circ \theta_M =
- \beta_N \circ \alpha_N \circ (\id \bartimes f) \circ \theta_M,
\\
\label{eqn-the-commutator-of-f-and-theta-B}
\theta_N  \circ (f \bartimes \id) - (f \bartimes \id) \circ \theta_M =
- \beta_N \circ \alpha_N \circ (f \bartimes \id) \circ \theta_M.
\end{align}
\end{enumerate}
\end{prps}

\begin{proof}
\bf \eqref{item-compositions-of-id-bartimes-f-and-f-bartimes-id-with-alpha-and-beta}: \rm
We only prove the first equality, the second one is proved similarly. By the definitions of the morphisms involved
the composition
$$  M \xrightarrow{\beta_M} \A \bartimes_\A  M \xrightarrow{\id
\bartimes f} \A \bartimes_\A N \xrightarrow{\alpha_N} N $$
in $\AmodbarB$ corresponds to the composition 
\begin{small}
\begin{align*}
& \barA \otimes M \otimes \barB
\xrightarrow{\Delta^2 \otimes \id \otimes \Delta^2}
\barA \otimes \barA \otimes \left( \barA \otimes M \otimes \barB \right) \otimes \barB \otimes \barB
\xrightarrow{\id^{\otimes 5} \otimes \tau \otimes \id^{\otimes 2}}
\barA \otimes \left( \barA \otimes \barA \otimes M \otimes \barB \right)  \otimes \barB
\xrightarrow{\id^{\otimes 2} \otimes \Delta \otimes \id^{\otimes 3}} \\
\rightarrow &
\barA \otimes \left( \barA \otimes \barA \otimes \barA \otimes M \otimes \barB \right)  \otimes \barB
\xrightarrow{\id \otimes \tau \otimes f \otimes \id}
\barA \otimes \barA \otimes N \otimes \barB
\xrightarrow{\tau^{\otimes 2} \otimes \id \otimes \tau}
N
\end{align*}
\end{small}
in $\AmodB$. The latter
simplifies via the identity $\tau \circ \Delta = \id$ and the functoriality 
of $\otimes$ to
$$ \barA \otimes M \otimes \barB \xrightarrow{f} N $$
as required.

\bf \eqref{item-the-commutators-of-f-and-alpha}: \rm
We only treat the left square in 
\eqref{eqn-the-commutative-squares-for-f-and-theta}, 
the right one is treated similarly. 
By \eqref{item-compositions-of-id-bartimes-f-and-f-bartimes-id-with-alpha-and-beta} we have
\begin{align*}
\alpha_N  \circ (\id \bartimes f) - f \circ \alpha_ M = 
\alpha_N  \circ (\id \bartimes f) - \alpha_N \circ 
(\id \bartimes f) \circ \beta_M \circ  \alpha_ M.
\end{align*}

By Prop.\ref{prop-alpha-beta_k-theta-relations} 
we have $d\theta_M = \id - \beta_M \circ \alpha_M$ and thus 
\begin{align*}
\alpha_N  \circ (\id \bartimes f) - \alpha_N \circ 
(\id \bartimes f) \circ \beta_M \circ  \alpha_ M =
\alpha_N  \circ (\id \bartimes f) \circ (\id - \beta_M \circ \alpha_M) =
\alpha_N  \circ (\id \bartimes f) \circ d\theta_M
\end{align*}
whence the assertion
\eqref{eqn-the-commutator-of-f-and-alpha-A} readily follows. 

\bf
\eqref{item-f-and-beta-commute}: \rm 
We only prove that the right square in 
\eqref{eqn-commutative-squares-for-f-and-beta}
commutes, as the proof for the left square is similar. By the
definition of the morphisms involved the compositions
$$ M \xrightarrow{f} N \xrightarrow{\beta_N}  N \bartimes_\B \B $$ 
$$ M \xrightarrow{\beta_M} M \bartimes \B \xrightarrow{f \bartimes \id}
N \bartimes \B $$
in $\AmodbarB$ correspond to the compositions
$$ \barA \otimes M \otimes \barB 
\xrightarrow{\Delta \otimes \id \otimes \Delta}
\barA \otimes \left(\barA \otimes M \otimes \barB\right) \otimes \barB
\xrightarrow{\id \otimes f \otimes \id}
\barA \otimes N \otimes \barB
\xrightarrow{\tau \otimes \id \otimes \id}
N \otimes \barB $$
$$ \barA \otimes M \otimes \barB 
\xrightarrow{\Delta \otimes \id \otimes \Delta}
\barA \otimes \left(\barA \otimes M \otimes \barB\right) \otimes \barB
\xrightarrow{\tau \otimes \id \otimes \id}
\barA \otimes M \otimes \barB \otimes \barB
\xrightarrow{\id^{\otimes 2} \otimes \Delta \otimes \tau}
\barA \otimes M \otimes \barB \otimes \barB
\xrightarrow{f \otimes \id}
N \otimes \barB$$
in $\AmodB$. Both these $\AmodB$ compositions simplify to 
$$
\barA \otimes M \otimes \barB 
\xrightarrow{\id \otimes \id \otimes \Delta}
\barA \otimes M \otimes \barB \otimes \barB
\xrightarrow{f \otimes \id}
N \otimes \barB, 
$$
as required.

\eqref{item-the-commutators-of-f-and-theta}:\;
We only treat the left square in 
\eqref{eqn-the-commutative-squares-for-f-and-theta}, 
the right one is treated similarly. 
Consider the maps 
$$\A \bartimes M \longrightarrow \A \bartimes N$$
given by the compositions 
$(\id \bartimes f) \circ \theta_M$, 
$\theta_N  \circ (\id \bartimes f)$, 
and $\beta_N \circ \alpha_N \circ (\id \bartimes f) \circ \theta_M$
in $\AmodbarB$. 
After simplification, the corresponding $\AmodB$ maps 
$$ \barA \otimes \barA \otimes M \otimes \barB \longrightarrow \barA \otimes N$$
are given by the compositions
$$ \barA \otimes \barA \otimes M \otimes \barB
\xrightarrow{(\dots) \otimes \id \otimes \tau}
\barA \otimes \barA \otimes M \otimes \barB
\xrightarrow{\id \otimes f}
\barA \otimes N $$
where $(\dots)$ denotes the maps 
$\Delta \circ \mu$, 
$(\mu \otimes \id) \otimes (\id \otimes \Delta)$,
$(\id \otimes \mu) \otimes (\Delta \otimes \id)$, respectively. 
The desired result now follows from 
Prop.~\ref{prps-properties-of-mu-and-lambda_k}\eqref{item-composition-Delta-circ-mu}. 
\end{proof}

\begin{defn}
Let $\A$ and $\B$ be DG-categories and let $f\colon M \rightarrow N$ be a
morphism in $\AmodbarB$. We say that $f$ is a 
\em fiberwise $\modAopp$ morphism \rm if its 
fiber $M_b \rightarrow N_b$ over each $b \in \B$ lies in 
$\Upsilon(\modAopp)$, where $\Upsilon$ is the inclusion 
of Prop.~\ref{prps-embedding-of-modA-into-modbarA}. 
Similarly, $f$ is a \em fiberwise $\modB$ morphism \rm if its fiber
$\leftidx{_a}{M} \rightarrow \leftidx{_a}{N}$ over each $a \in \A$ 
lies in $\Upsilon(\modB)$. 
\end{defn}

\begin{cor}
\label{cor-the-square-commutes-if-f-is-a-fiberwise-nonbar-morphism}
Let $\A$ and $\B$ be DG-categories and let $f\colon M \rightarrow N$ 
be a morphism in $\AmodbarB$. 
If $f$ is a fiberwise $\modAopp$ (resp. fiberwise $\modB$) morphism,  
the left (resp. right) squares 
\eqref{eqn-the-commutative-squares-for-f-and-alpha}
and \eqref{eqn-the-commutative-squares-for-f-and-theta} commute. 
\end{cor}
\begin{proof}
We only treat the left squares in 
\eqref{eqn-the-commutative-squares-for-f-and-alpha}
and \eqref{eqn-the-commutative-squares-for-f-and-theta}, 
the right squares are treated similarly. 

Writing out and simplifying the $\AmodB$ morphism corresponding to 
$$ \A \bartimes M \xrightarrow{\theta_M} \A \bartimes M \xrightarrow{\id \bartimes f} \A \bartimes N \xrightarrow{\alpha_N}
N $$
in a fashion similar to the one employed in the proof of Prop.~\ref{prps-alpha-and-tensoring-f-with-identity}\eqref{item-the-commutators-of-f-and-alpha} we obtain 
\begin{align}
\label{eqn-commutator-of-id-bartimes-f-and-f-explicit}
\barA \otimes \barA \otimes M \otimes \barB \xrightarrow{\mu \otimes \id \otimes \id}
 \barA \otimes M \otimes \barB \xrightarrow{f} N.
\end{align}

Let now $f$ be a fiberwise $\modA$ morphism. Then 
$\barA \otimes M \otimes \barB \xrightarrow{f} N$
factors through $$\barA \otimes M \otimes \barB \xrightarrow{\tau
\otimes \id \otimes \id} M \otimes \barB,$$ and therefore the composition
\eqref{eqn-commutator-of-id-bartimes-f-and-f-explicit}  
factors through $(\tau \circ \mu) \otimes \id \otimes \id$. 
Since $\tau \circ \mu  = 0$ we conclude that 
$\alpha_N \circ (\id \bartimes f) \circ \theta_M = 0$. Furthermore, 
if $f$ is a fiberwise $\modA$ morphism, then so is $df$, and therefore $\alpha_N \circ (\id \bartimes df) \circ \theta_M = 0$. 
It now follows by Prop.~\ref{prps-alpha-and-tensoring-f-with-identity}
that the left squares in 
\eqref{eqn-the-commutative-squares-for-f-and-alpha}
and \eqref{eqn-the-commutative-squares-for-f-and-theta}
commute. 
\end{proof}

\subsection{Dualisation}
\label{section-dualisation}

In this section we look at the dualising functors for bar categories
of bimodules:

\begin{defn} 
\label{defn-dualising-functors-in-modbar}
Let $\A$ and $\B$ be DG-categories. Define the 
\em dualising functors \rm $\AmodbarB \rightarrow \left(\BmodbarA\right)^\opp$ 
\begin{align*}
(-)^\barA \overset{\text{def}}{=} \barhom_\A(-,\A) \\
(-)^\barB \overset{\text{def}}{=} \barhom_\B(-,\B). 
\end{align*}
\end{defn}

By Tensor-Hom adjunction we have for any $M \in \AmodbarB$ and $N \in \BmodbarA$
\begin{align}
\barhom_{(\BmodbarA)^\opp}(M^{\barA}, N) \simeq 
\barhom_{\AmodbarA}(M \bartimes_\B N, \A) \simeq
\barhom_{\AmodbarB}(M, N^{\barA}) 
\end{align}
It follows that the functor 
$$(-)^\barA\colon \left(\BmodbarA\right)^\opp \rightarrow  \AmodbarB $$
is left adjoint to the functor 
$$(-)^\barA\colon \AmodbarB \rightarrow  \left(\BmodbarA\right)^\opp. $$
The adjunction unit is the natural transformation
\begin{align}
\label{eqn-double-A-dual-natural-transformation}
\id_{\AmodbarB} \longrightarrow (-)^{\barA\barA}
\end{align}
of endofunctors of $\AmodbarB$ defined on every $M \in \AmodbarB$
by the right adjoint of the evaluation map 
$$ M \bartimes_\B \barhom_\A\left(M,\A\right) \xrightarrow{\ev} \A $$
with respect to the functor $(-) \bartimes_\B \barhom_\A\left(M,\A\right)$. 
The adjunction counit is the natural transformation 
$$ (-)^{\barA\barA} \rightarrow \id_{\left(\BmodbarA\right)^\opp} $$
which corresponds to the natural transformation $\id_{\BmodbarA} \rightarrow (-)^{\barA\barA}$ 
defined in the same way as \eqref{eqn-double-A-dual-natural-transformation}. 

We define similarly natural transformations
\begin{align}
\label{eqn-double-B-dual-natural-transformation}
\id_{\bullet} \longrightarrow (-)^{\barB\barB}
\end{align}
which give the unit and the counit of the analogous adjunction of $(-)^\barB$ with itself. 

\begin{lemma}
\label{lemma-double-duals-and-duals-are-equivalences}
Let $\A$ and $\B$ be DG-categories. 
\begin{enumerate}
\item 
\label{item-AB-double-dualisation-isomorphism-on-AB-perfect-bimodules}
The natural transformations 
\eqref{eqn-double-A-dual-natural-transformation}
and
\eqref{eqn-double-B-dual-natural-transformation}
are homotopy equivalences on $\A$- and $\B$-perfect bimodules, respectively. 

\item 
\label{item-AB-dualisation-functors-are-quasi-equivalences-on-AB-perfect-bimodules}
The functors $(-)^\barA$ and $(-)^\barB$ restrict to 
quasi-equivalences
\begin{align}
\label{eqn-Adualisation-equivalence}
\left((\AmodbarB)^{\Aperf}\right)^{\opp} 
\xrightarrow{(-)^{\barA}} (\BmodbarA)^{\Aperf},
\\
\label{eqn-Bdualisation-equivalence}
\left((\AmodbarB)^{\Bperf}\right)^{\opp} 
\xrightarrow{(-)^{\barB}} (\BmodbarA)^{\Bperf},
\end{align}
\end{enumerate}
\end{lemma}
\begin{proof}
\eqref{item-AB-double-dualisation-isomorphism-on-AB-perfect-bimodules}:
We proceed by reduction to a similar result for ordinary categories 
of DG-bimodules proved in \cite[\S2]{AnnoLogvinenko-SphericalDGFunctors}. 
Let $M \in \modbarB$ and let $a \in \A$. 
Since $(M^\barB)_a = (\aM)^\barB$, it is clear that 
the fiber over $a$ of 
\begin{align}
\label{eqn-double-B-dual-natural-transformation-of-bimodules-applied-to-M}
M \xrightarrow{\eqref{eqn-double-B-dual-natural-transformation}}
M^{\barB\barB} 
\end{align}
is the analogous natural transformation
\eqref{eqn-double-B-dual-natural-transformation}
of endofunctors of $\modbarB$ applied to $\aM$. 

It follows from the description of the adjunction unit 
\eqref{eqn-M-bartimes-adjunction-unit} that the $\modbarB$ map  
$$ \aM \xrightarrow{\eqref{eqn-double-B-dual-natural-transformation}} 
(\aM)^{\barB\barB} $$
is defined by the $\modB$ map 
\begin{align}
\label{eqn-modB-map-underlying-double-dual-nat-transform-in-modbarB}
& \aM \otimes_\B \barB \xrightarrow{\mlt} 
\homm_\B\left(
\left(\aM \otimes_\B \barB\right)^\B, \,
\aM \otimes_\B \barB \otimes_\B \left(\aM \otimes_\B \barB\right)^\B
\right) 
\xrightarrow{\ev \circ (-)} \\
\nonumber
\rightarrow \quad &
\homm_\B\left(
\left(\aM \otimes_\B \barB\right)^\B, \B \right)
\xrightarrow{(-) \circ (\id \otimes \tau)}
\homm_\B\left(
\left(\aM \otimes_\B \barB\right)^\B \otimes_\B \barB, 
\B \right).
\end{align}
If $M$ is $\B$-perfect, $\aM$ is perfect. On the other hand, $\barB$ is 
h-projective and both left and right perfect. 
By \cite[Prop. 2.5 and 2.14]{AnnoLogvinenko-SphericalDGFunctors}
when tensoring an $h$-projective (resp. perfect) bimodule with 
a bimodule which is $h$-projective (resp. perfect) 
on the side not involved in the tensor product the result is $h$-projective
(resp. perfect). It follows 
that $\aM \otimes_\B \barB$ is $h$-projective and perfect, and hence 
so is $\left(\aM \otimes_\B \barB\right)^\B$. 
Thus the last map in the composition above is a quasi-isomorphism. On 
the other hand, the first two maps define a natural transformation 
$\id \rightarrow (-)^{\B\B}$ of endofunctors of $\B$. It is
an isomorphism on representables, and hence a quasi-isomorphism
on all $h$-projective and perfect modules,
cf.~\cite[\S2.2]{AnnoLogvinenko-SphericalDGFunctors}. 
In particular, it is a quasi-isomorphism on $\aM \otimes_\B \barB$.  
Thus \eqref{eqn-modB-map-underlying-double-dual-nat-transform-in-modbarB}
is a quasi-isomorphism. 

We conclude that for a $\B$-perfect $M$ the $\AmodB$ map which defines 
\eqref{eqn-double-B-dual-natural-transformation-of-bimodules-applied-to-M}
is a quasi-isomorphism, since its every fiber over $\A$ is.  
It follows from
Cor.~\ref{cor-criterion-for-modbar-morphism-to-be-homotopy-equivalence}
that \eqref{eqn-double-B-dual-natural-transformation-of-bimodules-applied-to-M} 
itself is a homotopy equivalence, as desired. 

A similar argument shows that \eqref{eqn-double-A-dual-natural-transformation}
is a homotopy equivalence on $\A$-perfect bimodules. 

\eqref{item-AB-dualisation-functors-are-quasi-equivalences-on-AB-perfect-bimodules}:
 This follows from \eqref{item-AB-double-dualisation-isomorphism-on-AB-perfect-bimodules}
 since both the units and the counits of the adjunctions of $(-)^\barA$ and $(-)^\barB$
 with themselves were shown to be homotopy equivalences on the corresponding subcategories. 
\end{proof}

The following is 
the $\modbar$ analogue of the map \eqref{eqn-bringing-a-factor-into-the-hom}:
\begin{defn}
Let $\A$, $\B$, $\C$, and $\D$ be DG-categories. 
Let $M \in \AmodbarB$, $N \in \DmodbarB$, and $L \in \CmodbarA$. 

Define the $\CmodbarD$ map 
\begin{align}
\label{eqn-bringing-a-factor-into-the-barhom} 
L \bartimes_\A \barhom_\B(N,M) \longrightarrow \barhom_\B(N,L \bartimes_\A M) 
\end{align}
as the composition 
\begin{align*}
L \bartimes_\A \barhom_\B(N,M) 
\xrightarrow{\mlt \bartimes \id}
\barhom_\B(M, L \bartimes_\A M)
\bartimes_\A
\barhom_\B(N,M)
\xrightarrow{\composition}
\barhom_\B(N, L \bartimes_\A M).
\end{align*}
\end{defn}

\begin{lemma}
\label{lemma-bringing-a-factor-into-the-barhom-is-a-homotopy-equivalence}
Let $\A$, $\B$, $\C$, and $\D$ be DG-categories and let 
$M \in \AmodB$, $N \in \DmodB$, and $L \in \CmodA$. 
The $\CmodbarD$ map 
$$ 
L \bartimes_\A \barhom_\B(N,M)
\xrightarrow{\eqref{eqn-bringing-a-factor-into-the-barhom}}
\barhom_\B(N,L \bartimes_\A M)
$$ 
is a homotopy equivalence when $N$ is $\B$-perfect or $L$ is 
$\A$-perfect. 
\end{lemma}
\begin{proof}
It follows from the definitions of 
the adjunction unit \eqref{eqn-bartimes-M-adjunction-unit}
and the composition map \eqref{eqn-composition-map-in-modbar} 
that the $\CmodD$ map defining 
\eqref{eqn-bringing-a-factor-into-the-barhom}
is 
\begin{align*}
& L \otimes_\A \barA \otimes_\A \homm_\B(N \otimes_\B \barB, M) 
\xrightarrow{\mlt \otimes \id}
\homm_\B(M, L \otimes_\A \barA \otimes_\A M)
\otimes_\A \homm_\B(N \otimes_\B \barB, M) 
\xrightarrow{\composition} \\
\rightarrow \quad &
\homm_\B(N \otimes_\B \barB, L \otimes_\A \barA \otimes_\A M). 
\end{align*}
Thus it is an instance of the map
\eqref{eqn-bringing-a-factor-into-the-hom} and 
is therefore a quasi-isomorphism when 
$N \otimes_\B \barB \in \Bprfhpr$ or
$L \otimes_\A \barA \in \Aprfhpr$. 
Hence when $N$ is $\B$-perfect or $L$ is $\A$-perfect
the $\CmodD$ map defining \eqref{eqn-bringing-a-factor-into-the-barhom}
is a quasi-isomorphism, and by
Cor.~\ref{cor-criterion-for-modbar-morphism-to-be-homotopy-equivalence}
it follows that \eqref{eqn-bringing-a-factor-into-the-barhom} is a homotopy
equivalence, as desired.  
\end{proof}

\begin{defn}
\label{defn-the-natural-map-eta}
Let $\A$, $\B$, and $\C$ be DG-categories and let $M \in \AmodbarB$. 

Define the natural transformations
\begin{align}
\label{eqn-bring-in-map-otimes-M^B-to-Hom-M-M}
(-) \bartimes_\B M^\barB \xrightarrow{\eta_\B} \barhom_\B(M,-), 
\\
\label{eqn-bring-in-map-otimes-M^A-to-Hom-M-M}
M^\barA \otimes_\A (-) \xrightarrow{\eta_\A} \barhom_\A(M,-)
\end{align}
of functors $\CmodbarB \rightarrow \CmodbarA$ and 
$\AmodbarC \rightarrow \BmodbarC$, respectively, as
the compositions
$$ (-) \bartimes_\B \barhom_\B(M,\B)
\xrightarrow{\eqref{eqn-M-to-barhom-A-M} \bartimes \id}
\barhom_\B(\B, -) \bartimes_\B \barhom_\B(M,\B)
\xrightarrow{\composition} 
\barhom_\B(M,-), 
$$
$$ \barhom_\A(M,\A) \bartimes_\A (-)
\xrightarrow{\id \bartimes \eqref{eqn-M-to-barhom-A-M}}
\barhom_\A(M,\A) \bartimes_\A \barhom_\A(\A, -)
\xrightarrow{\composition} 
\barhom_\A(M,-). 
$$
\end{defn}

\begin{lemma}
\label{lemma-A-and-B-perfection-iff-the-evaluation-map-is-homotopy-equivalence}
Let $\A$ and $\B$ be DG-categories and let $M \in \AmodbarB$. 
\begin{enumerate}
\item 
\label{item-B-perfect-iff-evaluation-map-is-homotopy-equivalence}
$M$ is $\B$-perfect if and only if the map  
$$ (-) \bartimes_\B M^\barB \xrightarrow{\eta_\B} \barhom_\B(M,-) $$
is a homotopy equivalence of functors 
$\CmodbarB \rightarrow \CmodbarA$ for any DG-category $\C$. 
\item 
\label{item-A-perfect-iff-evaluation-map-is-homotopy-equivalence}
$M$ is $\A$-perfect if and only if the map  
$$ M^\barA \bartimes_\A (-) 
\xrightarrow{\eta_\A}
\barhom_\A(M,-) $$
is a homotopy equivalence of functors 
$\AmodbarC \rightarrow \BmodbarC$ for any DG-category $\C$. 
\end{enumerate}
\end{lemma}
\begin{proof}
We only give the proof for the assertion 
\eqref{item-B-perfect-iff-evaluation-map-is-homotopy-equivalence}, 
as the proof for 
\eqref{item-A-perfect-iff-evaluation-map-is-homotopy-equivalence}
is identical. 

Assume that $M$ is $\B$-perfect. Since 
$(-) \xrightarrow{\eqref{eqn-M-to-barhom-A-M}} \barhom_\B(\B, -)$
is the right adjoint of \eqref{eqn-A-bartimes-M-to-M}, it equals the composition 
$$ (-) \xrightarrow{\mlt} \barhom_\B\left(\B,(-) \bartimes_\B \B \right) 
\xrightarrow{\eqref{eqn-A-bartimes-M-to-M} \circ(-)} 
\barhom_\B\left(\B,-\right).$$
It follows that the map 
$$ (-) \bartimes_\B M^\barB \xrightarrow{\eta_\B} \barhom_\B(M,-) $$
equals the composition
\begin{scriptsize}
$$ (-) \bartimes_\B \barhom_\B(M,\B)
\xrightarrow{\mlt \bartimes \id}
\barhom_\B(\B, (-) \bartimes_\B \B) \bartimes_\B \barhom_\B(M,\B)
\xrightarrow{\eqref{eqn-A-bartimes-M-to-M} \circ(-)}
\barhom_\B(\B, -) \bartimes_\B \barhom_\B(M,\B)
\xrightarrow{\composition} 
\barhom_\B(M,-)
$$
\end{scriptsize}
and hence, by the definition of the map 
\eqref{eqn-bringing-a-factor-into-the-barhom}, to
$$ (-) \bartimes_\B M^\barB
\xrightarrow{\eqref{eqn-bringing-a-factor-into-the-barhom}} 
\barhom_\B(M,(-) \bartimes_\B \B) 
\xrightarrow{\eqref{eqn-A-bartimes-M-to-M} \circ(-)} 
\barhom_\B(M,-). 
$$
The first map in this composition is a homotopy equivalence by 
Lemma
\ref{lemma-bringing-a-factor-into-the-barhom-is-a-homotopy-equivalence}, 
and the second one is a homotopy equivalence since \eqref{eqn-A-bartimes-M-to-M}
is. We conclude that $\eta_\B$ is also a homotopy equivalence, as desired. 

Conversely, assume that
$$ (-) \bartimes_\B M^\barB \xrightarrow{\eta_\B} \barhom_\B(M,-) $$
is a homotopy equivalence on all of $\modbarB$. By
Prop.~\ref{prps-bartimes-and-barhom-are-identified-with-derived-versions}
it follows that 
$$ (-) \ldertimes_\B \MddB \simeq \rder\homm_\B(M,-). $$
Thus $\rder\homm_\B(M,-)$ commutes with infinite direct sums, 
i.e. $M$ is $\B$-perfect. 
\end{proof}

\begin{lemma}
Let $\A$, $\B$, and $\C$ be DG-categories. 
Let $M \in \AmodbarB$ and $N \in \BmodbarC$. 

If $M$ is $\B$-perfect, there is a $\CmodbarA$ homotopy equivalence
\begin{equation}
\label{eqn-right-dual-of-tensor-product}
N^\barC \bartimes_\B M^\barB \longrightarrow  (M\bartimes_\B N)^\barC.
\end{equation}
If $N$ is $\B$-perfect, there is a $\CmodbarA$ homotopy equivalence 
\begin{equation}
\label{eqn-left-dual-of-tensor-product}
N^\barB \bartimes_\B M^\barA \longrightarrow (M\bartimes_\B N)^\barA.
\end{equation}
\end{lemma}
\begin{proof}
Similarly to the proof of Lemma 2.12 in 
\cite{AnnoLogvinenko-SphericalDGFunctors},
define \eqref{eqn-right-dual-of-tensor-product}
to be the composition 
\begin{equation}
N^\barC\bartimes_\B M^\barB 
\xrightarrow{\eta_\B}
\barhom_\B (M, N^\barC)
\xrightarrow{\text{adjunction}}
\barhom_\C(M\bartimes_\B N,\C). 
\end{equation}
The first composant is a homotopy equivalence by 
Lemma
\ref{lemma-A-and-B-perfection-iff-the-evaluation-map-is-homotopy-equivalence}
and the second composant is the Tensor-Hom adjunction isomorphism. 
Thus \eqref{eqn-right-dual-of-tensor-product} is itself a homotopy
equivalence. 

We define \eqref{eqn-left-dual-of-tensor-product} similarly. 
\end{proof}

\subsection{Convolution functor for $\pretriag(\modbar)$}
\label{section-convolution-functor-for-pretriag-modbar}

Let $\A$ be a DG-category. It is well known that $\modA$ is a strongly
pretriangulated category, cf.
\cite{BondalKapranov-EnhancedTriangulatedCategories},
\cite[\S3.2]{AnnoLogvinenko-SphericalDGFunctors}. That is, the natural 
inclusion 
$$ \modA \hookrightarrow \pretriag \modA $$ 
is an equivalence of categories. 
Its quasi-inverse 
$$\pretriag \modA \xrightarrow{T} \modA$$
is the \em convolution functor\rm. Similarly, $\noddinfA$ is 
strongly pretriangulated and admits a convolution functor.
This is because the DG category of all DG $\infbar \A$-comodules is strongly pre-triangulated, 
and $\noddinfA$ is equivalent to its full subcategory consisting of the DG $\infbar \A$-comodules
which are free as graded comodules. On the level of graded DG-comodules the convolution functor 
is a direct sum with shifts. Thus if every object in the twisted complex is free as a graded comodule, so 
is its convolution. We thus obtain the convolution functor for $\noddinfA$, as required. 

The category $\modbarA$ is not strongly pretriangulated. It 
is however \em pretriangulated\rm. In other words, 
$$ \modbarA \hookrightarrow \pretriag \modbarA $$ 
is only a quasi-equivalence. This is readily seen via the
isomorphism $\modbarA \simeq \noddinfdgA$ of 
Prop.~\ref{prps-modbarA-to-noddinfdgA-isomorphism}. While $\noddinfA$
is strongly pretriangulated, its full subcategory $\noddinfdgA$ is not. 
This is because when we restrict the convolution functor 
$$ \pretriag \noddinfA \xrightarrow{T_\infty} \noddinfA $$
to $\pretriag \noddinfdgA$ its image doesn't restrict to
$\noddinfdgA$. Indeed, let $\left( E_i, \alpha_{ij}\right)$ 
be a twisted complex over $\noddinfdgA$. That is, $E_i$ are DG 
$\A$-modules and $\alpha_{ij}$ are $\Ainfty$-morphisms between them. 
Then, $T_\infty(E_i, \alpha_{ij})$ is the $\Ainfty$-module 
whose underlying DG-module is $T(E_i, \alpha_{ij})$
and whose higher $\Ainfty$-module structure is defined 
by the higher operations of the differentials $\alpha_{ij}$. This structure 
will not in general be trivial unless the higher operations are all zero, i.e. 
unless $\alpha_{ij}$ are regular DG morphisms. 

The fact that $T_\infty(\pretriag \noddinfdgA)$ doesn't land 
in $\noddinfdgA$ can be readily fixed by applying the semi-free
resolution $(-) \inftimes_\A \A$ of 
\S\ref{section-functorial-semi-free-resolution-for-noddinfhuA}. 
We then obtain a functor 
\begin{align}
\label{eqn-modbar-convolution-in-Ainfty-language}
\pretriag \noddinfdgA \xrightarrow{T_\infty} \noddinfhuA
\xrightarrow{(-) \inftimes_\A \A} \noddinfdgA
\end{align}
which is a quasi-equivalence because both of its composants are. 
It is also, by construction, a quasi-inverse of 
the natural inclusion 
$\noddinfdgA \rightarrow \pretriag \noddinfdgA $
and thus an analogue of a convolution functor. 
In this section, we translate these considerations to 
the case of $\modbarA$. 

Throughout this section, we adopt the following notation. 
Given a morphism $\alpha\colon E \rightarrow F$ in 
$\modbarA$ we denote by $\alphahat$ the underlying  
$\modA$ morphism $E \otimes_{\A} \barA \rightarrow  F$. 

Recall the natural inclusion 
$$ \modA \xrightarrow{\eqref{eqn-embedding-of-modA-into-modbarA}} \modbarA $$
of Prop.~\ref{prps-embedding-of-modA-into-modbarA}. 
It gets identified under the isomorphism $\modbarA \simeq \noddinfdgA$ of 
Prop.~\ref{prps-modbarA-to-noddinfdgA-isomorphism} with the inclusion 
$$ \modA =
\left(\noddinfA\right)_{\text{dg}}^{\text{strict}}
\hookrightarrow \noddinfdgA. $$ 
On the other hand, by 
Cor.~\ref{cor-inftimes-factors-through-strictlyunital/DG-modules}
the semi-free resolution 
$$
\noddinfhuA
\xrightarrow{ (-)\inftimes_\A \A }
\noddinfdgA
$$
factors through 
$\modA = \left(\noddinfA\right)_{\text{dg}}^{\text{strict}}$. 
As $(-)\inftimes_\A \A$ is identified with $(-) \bartimes_\A \A$ 
by the isomorphism $\modbarA \simeq \noddinfdgA$, it follows
that 
$$ \modbarA \xrightarrow{ (-) \bartimes_\A \A } \modbarA $$
factors as
\begin{align}
\label{eqn-factorisation-of-bartimes-A}
\modbarA 
\rightarrow
\modA
\xrightarrow{\eqref{eqn-embedding-of-modA-into-modbarA}} 
\modbarA. 
\end{align}

\begin{defn}
Define the functor 
\begin{align}
\modbarA \xrightarrow{\widetilde{(-)}} \modA
\end{align}
to be the first half of the factorisation of $(-) \bartimes_\A \A$ given in 
\eqref{eqn-factorisation-of-bartimes-A}. 
\end{defn}

\begin{lemma}
Let $E \in \modbarA$, then 
\begin{align}
\widetilde{E} = E \otimes_\A \barA.
\end{align}
Let $\alpha \in \homm_\modbarA(E,F)$, then 
\begin{align}
\label{eqn-the-underlying-morphism-of-the-tilde-functor}
\widetilde{\alpha}= 
\left(\alphahat \otimes \id \right) \circ \left(\id\otimes\Delta\right).
\end{align}
\end{lemma}
\begin{proof}
This follows directly from the definition of the $\bartimes$ bifunctor given in Definition \ref{defn-bar-tensor-bifunctor}. Note that it needs to be adjusted in an obvious way for having right modules, and not bimodules, in the left argument. 
 
The first assertion is immediate. For the second one, applying the definition to the morphisms $\alpha\colon E \rightarrow F$ and $\id_\A \colon \A \rightarrow \A$ we see that the underlying $\modA$ morphism of 
${\alpha \bartimes \id_\A}$ is the composition 
$$ E \otimes_\A \barA \otimes_\A \A \otimes_\A \barA
\xrightarrow{\id \otimes \Delta^{2} \otimes \id^{\otimes 2}}
E \otimes_\A \barA \otimes_\A \barA \otimes_\A \barA 
\otimes_\A \A \otimes_\A \barA 
\xrightarrow {\hat{\alpha} \otimes \id \otimes \hat{\id}_\A}
F \otimes_\A \barA \otimes_\A \A. 
$$
Since $\hat{\id}_\A = \tau \otimes \id \otimes \tau$, this can be rewritten as 
$$ E \otimes_\A \barA \otimes_\A \barA
\xrightarrow{\id^{\otimes 2} \otimes \tau}
E \otimes_\A \barA
\xrightarrow{\id \otimes \Delta}
E \otimes_\A \barA \otimes_\A \barA 
\xrightarrow {\hat{\alpha} \otimes \id}
F \otimes_\A \barA
$$
which is the image of \eqref{eqn-the-underlying-morphism-of-the-tilde-functor} under the category inclusion
$\eqref{eqn-embedding-of-modA-into-modbarA}$, as desired. 
\end{proof}

\begin{defn}
\label{defn-convolution-in-modbarA}
Define the \em convolution functor \rm 
\begin{align}
\label{eqn-convolution-functor-for-modbarA}
\pretriag \modbarA \xrightarrow{\barT} \modbarA
\end{align}
as the composition 
\begin{align}
\label{eqn-factorisation-of-convolution-in-modbarA}
\pretriag \modbarA \xrightarrow{\widetilde{(-)}} \pretriag \modA
\xrightarrow{T} \modA
\xrightarrow{\eqref{eqn-embedding-of-modA-into-modbarA}} 
\modbarA.
\end{align}
\end{defn}

\begin{prps}
The functor $\barT$ gets identified by the isomorphism 
$\modbarA \simeq \noddinfdgA$ with the functor
$$ \pretriag \noddinfdgA 
\xrightarrow{ \eqref{eqn-modbar-convolution-in-Ainfty-language} } 
\noddinfdgA.$$
\end{prps}
\begin{proof}
By its definition, $\modbarA \xrightarrow{\widetilde{(-)}} \modA$ gets identified 
by the isomorphism $\modbarA \simeq \noddinfdgA$ with the functor 
$\noddinfdgA \xrightarrow{(-) \inftimes_\A \A} \modA$. Therefore 
$\barT$ gets identified with the path in the diagram below which travels along the upper-right half 
of the lower rectangle:
\begin{equation*}
\begin{tikzcd}
&
\noddinfhuA
\ar{r}{(-) \inftimes_\A \A}
\ar[hookrightarrow]{d}
&
[5em]
\modA
\ar[hookrightarrow]{d}
&
\\
\pretriag \noddinfdgA
\ar[hookrightarrow]{r}
&
\pretriag \noddinfhuA
\ar{r}{(-) \inftimes_\A \A}
\ar{d}{T_\infty}
&
\pretriag \modA
\ar{d}{T}
&
\\
&
\noddinfhuA
\ar{r}{(-) \inftimes_\A \A}
&
\modA
\ar[hookrightarrow]{r}
&
\noddinfdgA.
\end{tikzcd}
\end{equation*}
On the other hand, the 
path which travels along the lower-left half of the lower rectangle  is precisely
\eqref{eqn-modbar-convolution-in-Ainfty-language}. The upper rectangle and the outer perimeter of the two rectangles clearly
commute. Since all the vertical arrows are category equivalences, the lower rectangle 
commutes as well. The desired assertion follows. 
\end{proof}

\begin{cor}
\label{cor-convolution-functor-is-a-quasi-equivalence}
The convolution functor 
$$ \pretriag \modbarA \xrightarrow{\barT} \modbarA $$
is a quasi-equivalence and a homotopy inverse of the natural 
inclusion $\modbarA \hookrightarrow \pretriag \modbarA$. 
\end{cor}

Let $\A$ and $\B$ be DG-categories. We define the convolution functor 
$$ \pretriag \AmodbarB \xrightarrow{\barT} \AmodbarB $$
similarly. 

\begin{lemma}  
\label{lemma-tensoring-and-homming-for-twisted-complexes}
Let $(E_i,\alpha_{ij})$ be a one-sided twisted complex in $\AmodbarB$, 
and let $E$ be its convolution.
\begin{enumerate}
\item
\label{item-tensoring-twisted-complexes}
Let $(F_i, \beta_{ij})$ be a one-sided twisted complex in $\BmodbarC$, 
and let $F$ be its convolution. Then there is an $\AmodbarC$ homotopy equivalence
\begin{small}
\begin{equation}\label{equation-lemma-tensoring-twisted-complexes}
\left\{
\bigoplus\limits_{k+l=i} E_k\bartimes_\B F_l, 
\sum\limits_{l+m=j} (-1)^{l(k-m+1)}\alpha_{km}\bartimes\id_l+\sum\limits_{k+n=j}(-1)^k\id_k\bartimes \beta_{ln} 
\right\} 
\to E\bartimes_\B F.
\end{equation}
\end{small}
\item
\label{item-hommming-twisted-complexes-B-version}
Let $(F_i, \beta_{ij})$ be a one-sided twisted complex in $\CmodbarB$, 
and let $F$ be its convolution. Then there is a $\CmodbarA$ homotopy equivalence
\begin{small}
\begin{equation}\label{equation-lemma-homs-between-convolutions}
\left\{
\bigoplus\limits_{l-k=i} \barhom_\B (E_k,F_l),
\sum\limits_{l-m=j} (-1)^{m(m-k)+l+1} (-)\circ \alpha_{mk} + 
\sum\limits_{n-k=j} (-1)^{(l-n+1)k} \beta_{ln} \circ (-)
\right\}
\to \barhom_\B (E,F).
\end{equation}
\end{small}

\item
\label{item-hommming-twisted-complexes-A-version}
Let $(F_i, \beta_{ij})$ be a one-sided twisted complex in $\AmodbarC$, 
and let $F$ be its convolution. Then there is a $\BmodbarC$ homotopy equivalence
\begin{small}
\begin{equation}
\left\{
\bigoplus\limits_{l-k=i} \barhom_\A (E_k,F_l),
\sum\limits_{l-m=j} (-1)^{m(m-k)+l+1} (-)\circ \alpha_{mk} + 
\sum\limits_{n-k=j} (-1)^{(l-n+1)k} \beta_{ln} \circ (-)
\right\}
\to \barhom_\A (E,F).
\end{equation}
\end{small}
\end{enumerate}
\end{lemma}
\begin{proof}
\eqref{item-tensoring-twisted-complexes}:
By the definition of the convolution functor as the composition 
\eqref{eqn-factorisation-of-convolution-in-modbarA}
it suffices to show that $T \circ \widetilde{(-)}$ applied to the
twisted complex in the 
LHS of \eqref{equation-lemma-tensoring-twisted-complexes}
is quasi-isomorphic to 
to $E \otimes_\B \barB \otimes_\B F$ in $\AmodC$. Since the natural 
inclusion 
$\AmodC \xrightarrow{\eqref{eqn-embedding-of-modA-into-modbarA}} \AmodbarC$ 
maps quasi-isomorphisms to homotopy equivalences, the desired 
assertion then follows. 

By \cite[Lemma 3.4]{AnnoLogvinenko-SphericalDGFunctors} 
we have the following isomorphism in $\AmodC$: 
\begin{equation}
\label{eqn-tilde-version-of-tensoring-tisted-complexes}
\left\{
\bigoplus\limits_{k+l=i} \widetilde{E}_k\otimes_\B \widetilde{F}_l, 
\sum\limits_{l+m=j}
(-1)^{l(k-m+1)}\widetilde{\alpha}_{km}\otimes\id_l+\sum\limits_{k+n=j}(-1)^k\id_k\otimes \widetilde{\beta}_{ln} 
\right\} 
\simeq E\otimes_\B F.
\end{equation}
Let $(G_i, \gamma_{ij})$ be the image of the twisted complex in the LHS of
\eqref{equation-lemma-tensoring-twisted-complexes} under
$\widetilde{(-)}$. We therefore have 
$$ G_i = \bigoplus\limits_{k+l=i} 
\barA \otimes_\A E_k \otimes_\B \barB \otimes_\B F_l \otimes_\C \barC $$
and thus there is a map from $(G_i, \gamma_{ij})$ to 
the twisted complex in the LHS of
\eqref{eqn-tilde-version-of-tensoring-tisted-complexes}
whose only non-zero components are degree $0$ homotopy equivalences  
\begin{equation*}
\barA \otimes_\A E_k \otimes_\B \barB \otimes_\B F_l \otimes_\C \barC 
\xrightarrow{\id^{\otimes 2} \otimes \Delta \otimes \id^{\otimes 2}}
(\barA \otimes_\A E_k \otimes_\barB \barB) \otimes_\B 
(\barB \otimes_\B F_l \otimes_\C \barC).  \end{equation*}
One readily checks that these commute with the differentials 
of the twisted complexes, thus the resulting map 
from $\left\{ G_i, \gamma_{ij} \right\}$ 
to the LHS of \eqref{eqn-tilde-version-of-tensoring-tisted-complexes}
is closed. Therefore by
Cor.~\ref{cor-every-vertical-comp-is-quasiiso-then-map-is-itself-quasiiso}
it is a quasi-isomorphism. Thus $\left\{ G_i, \gamma_{ij} \right\}$ 
is quasi-isomorphic to $E \otimes_\B F$, and thus to $E \otimes_\B
\barB \otimes_\B F$, as desired. 

\eqref{item-hommming-twisted-complexes-B-version}: 
Similar to the proof of \eqref{item-tensoring-twisted-complexes}
let $(G_i, \gamma_{ij})$
be the image of the twisted complex in the LHS of
\eqref{equation-lemma-homs-between-convolutions}
under $\widetilde{(-)}$. It suffices to show that 
$\left\{G_i, \gamma_{ij}\right\}$ is quasi-isomorphic
to $\homm_\B(E \otimes \barB,F)$ in $\CmodA$. 

By \cite[Lemma 3.4]{AnnoLogvinenko-SphericalDGFunctors}, 
$\homm_\B(E \otimes \barB,F)$ is isomorphic in $\CmodA$
to the convolution of 
\begin{equation}\label{equation-homs-between-convolutions-twisted-complex}
\left(
\bigoplus\limits_{l-k=i} \homm_\B (\widetilde{E}_k\otimes\barB,\widetilde{F}_l),
\sum\limits_{l-m=j} (-1)^{m(m-k)+l+1} (-)\circ (\widetilde{\alpha}_{mk}\otimes\id) + 
\sum\limits_{n-k=j} (-1)^{(l-n+1)k} \widetilde{\beta}_{ln} \circ (-)
\right).
\end{equation}

We have 
$$ G_i = 
\bigoplus\limits_{l-k = i} 
\barC \otimes_\C \homm_\B\left(E_k \otimes_\B \barB,F_l\right)
\otimes_\A \barA. $$ 
Consider therefore the following twisted complex over $\CmodA$:
\begin{equation}\label{equation-homs-between-convolutions-intermediate}
\left(
\bigoplus\limits_{l-k=i} \homm_\B (\widetilde{E}_k\otimes\barB,F_l),
\sum\limits_{l-m=j} (-1)^{m(m-k)+l+1} (-)\circ (\widetilde{\alpha}_{mk}\otimes\id) + 
\sum\limits_{n-k=j} (-1)^{(l-n+1)k} \beta_{ln} \circ (-)
\right).
\end{equation}
Consider the map 
from \eqref{equation-homs-between-convolutions-twisted-complex} to 
\eqref{equation-homs-between-convolutions-intermediate}
whose components  
$$
\homm_\B(\barA \otimes_\A {E}_k \otimes_\B \barB \otimes_\B \barB,
\barC \otimes_\C {F}_l \otimes_\B \barB)
\xrightarrow{(\tau \otimes \id \otimes \tau) \circ (-)}
\homm_\B(\barA \otimes_\A {E}_k \otimes_\B \barB \otimes_\B \barB,F_l)
$$
These are homotopy equivalences in $\CmodA$, and thus 
by Cor.~\ref{cor-every-vertical-comp-is-quasiiso-then-map-is-itself-quasiiso}
the induced map between the convolutions of 
\eqref{equation-homs-between-convolutions-twisted-complex} and 
\eqref{equation-homs-between-convolutions-intermediate} is a
quasi-isomorphism. 

On the other hand, consider the map of twisted complexes from 
$(G_i, \gamma_{ij})$ to
\eqref{equation-homs-between-convolutions-intermediate}
whose components are the maps
$$
\barC \otimes_\C \homm_\B\left(E_k \otimes_\B \barB,F_l\right) \otimes_\A \barA
\xrightarrow{ \tau \otimes \left( (-) \circ (\id \otimes \id \otimes
\tau) \right) \otimes \tau }
\homm_\B({E}_k \otimes_\B \barB \otimes_\B \barB,F_l). 
$$
Likewise, these are homotopy equivalences in $\CmodA$ and therefore
$\left\{G_i, \gamma_{ij}\right\}$ is quasi-isomorphic to the 
convolution of \eqref{equation-homs-between-convolutions-intermediate}. 
It is therefore quasi-isomorphic to 
\eqref{equation-homs-between-convolutions-twisted-complex}. 
We conclude that $\left\{G_i, \gamma_{ij}\right\}$ is quasi-isomorphic
to $\homm_\B(E \otimes \barB,F)$, as desired.

\eqref{item-hommming-twisted-complexes-A-version}: 
This is proved similarly to the assertion 
\eqref{item-hommming-twisted-complexes-B-version}. 

\end{proof}

\begin{lemma}
\label{lemma-duals-of-twisted-complexes}
  Let $(E_i,\alpha_{ij})$  be a twisted complex over 
$\AmodbarB$. Then there are homotopy equivalences:
\begin{align}
\label{eqn-right-dual-of-twisted-complex}
\{E_i,\alpha_{ij}\}^\barB & \rightarrow  \{E_{-i}^\barB, (-1)^{j^2+ij+1} \alpha_{(-i)(-j)}^\barB\};\\
\label{eqn-left-dual-of-twisted-complex}
\{E_i,\alpha_{ij}\}^\barA & \rightarrow  \{E_{-i}^\barA, (-1)^{j^2+ij+1} \alpha_{(-i)(-j)}^\barA\}.
\end{align}
\end{lemma}
\begin{proof}
This is proved similarly to 
Lemma \ref{lemma-tensoring-and-homming-for-twisted-complexes}. 
Use \cite[Lemma 3.5]{AnnoLogvinenko-SphericalDGFunctors} 
to write the LHS of \eqref{eqn-left-dual-of-twisted-complex}
and \eqref{eqn-right-dual-of-twisted-complex} as
twisted complexes over $\BmodA$. There are obvious maps
from these to the images under $\widetilde{(-)}$ 
of the RHS twisted complexes in \eqref{eqn-left-dual-of-twisted-complex} 
and \eqref{eqn-right-dual-of-twisted-complex} whose components
are all homotopy equivalences. The claim of the Lemma then 
follows by 
Cor.~\ref{cor-every-vertical-comp-is-quasiiso-then-map-is-itself-quasiiso}. 
\end{proof}

\section{Homotopy adjunction for tensor functors}
\label{section-homotopy-adjunction-for-tensor-functors}

Let $\A$, $\B$, and $\C$ be DG-categories. In this paper, we frequently 
consider $\AbimB$-, $\BbimC$-, etc.~bimodules as DG-enhancements 
of continuous exact functors $D(\A) \rightarrow D(\B)$, $D(\B)
\rightarrow D(\C)$, etc. Accordingly, whenever it is convenient
we adopt the following ``functorial'' notation: given 
$F \in \AmodbarB$ and $G \in \BmodbarC$ we write
\begin{align*}
GF \quad\quad\text{ for }\quad\quad  F \bartimes_\B G \ \in \ \AmodbarC. 
\end{align*}
This is because we work with categories of right modules, and for any $E \in \modbarA$ we have
$$ E \bartimes_\A (F \bartimes_{\B} G) = (E \bartimes_{\A} F) \bartimes_{\B} G. $$
Thus $F \bartimes_{\B} G$ enhances the functor which is the composition 
of first the functor enhanced by $F$, and then the functor enhanced by $G$, 
whence our shorthand $GF$. 

\subsection{Tensor functors}
\label{section-tensor-functors}

Let $\A$ and $\B$ be DG categories and let 
$f: D(\A) \rightarrow D(\B)$ be a continuous exact functor. 
As the following proposition demonstrates, 
this is equivalent to $f$ being a \em tensor functor\rm, that
is --- a functor given by tensoring by an $\AbimB$-bimodule:

\begin{prps}
\label{prps-tfae-functor-is-continuous}
The following are equivalent:
\begin{enumerate}
\item 
\label{item-f-has-right-adjoint}
$f$ has a right adjoint $r\colon D(\B) \rightarrow D(\A)$. 
\item 
\label{item-f-is-continuous}
$f$ is continuous. 
\item 
\label{item-f-is-a-tensor-functor}
$f$ is isomorphic to $H^0((-) \bartimes_\A M)$ for some $M \in
\AmodbarB$. 
\end{enumerate}
\end{prps}
\begin{proof}
The implication 
$\eqref{item-f-has-right-adjoint} 
\Rightarrow 
\eqref{item-f-is-continuous}$ is well-known and straightforward, 
the implication 
$\eqref{item-f-is-continuous}
\Rightarrow 
\eqref{item-f-is-a-tensor-functor}$
follows from \cite[\S6.4]{Keller-DerivingDGCategories}, 
and the implication 
$\eqref{item-f-is-a-tensor-functor}
\Rightarrow 
\eqref{item-f-has-right-adjoint}$ follows since by the Tensor-Hom
adjunction the functor $(-) \bartimes_\A M$ has
the right adjoint $\barhom_\B(M,-)$. 
\end{proof}

Let $f$ satisfy these equivalent conditions. Fix $M \in \AmodbarB$
such that $(-) \bartimes_\A M$ enhances $f$ as above. 
By Prop.~\ref{prps-bartimes-and-barhom-are-dg-adjoint}
the functor $\barhom_\B(M,-)$ is genuinely adjoint
to $(-) \bartimes_\A M$ and thus enhances $r$. 
We conclude that any adjoint pair $(f,r)$ of functors 
$D(\B) \longleftrightarrow D(\A)$ with $f$ enhanceable 
can be enhanced by a pair of genuinely adjoint DG-functors. 

We are however, more interested in the case where $f$ has 
left and right adjoints $l$ and $r$ which are also tensor functors. 
Note, that $r$ always exists, but might not be a tensor functor, 
while $l$ may not exist, but when it does --- it is automatically
a tensor functor. The conditions of the existence of $l$ and of $r$
being a tensor functor are easily stated in terms of the properties
of the DG-bimodule enhancing $f$:

\begin{theorem}
\label{theorem-TFAE-M-is-B-and-A-perfect}
Let $\A$ and $\B$ be DG-categories and let $f\colon D(\A) \rightarrow
D(\B)$ be a tensor functor. Let $M \in \AmodbarB$ be any enhancement of $f$. 
\begin{enumerate}
\item 
\label{item-TFAE-M-is-B-perfect}
The following are equivalent:
\begin{enumerate}
\item 
\label{item-TFAE-M-is-B-perfect-item-right-adjoint-is-cts}
The right adjoint $r$ of $f$ is continuous.
\item 
\label{item-TFAE-M-is-B-perfect-item-f-preserves-compact-objects}
$f$ restricts to $D_{c}(\A) \rightarrow D_{c}(\B)$. 
\item 
\label{item-TFAE-M-is-B-perfect-item-M-is-B-perfect}
$M$ is $\B$-perfect.
\item 
\label{item-TFAE-M-is-B-perfect-item-MbarB-is-the-right-adjoint}
$H^0\left((-) \bartimes_\B M^{\barB}\right)$ is the right adjoint of $f$
(see Definition \ref{defn-dualising-functors-in-modbar}).
\end{enumerate}
\item 
\label{item-TFAE-M-is-A-perfect}
The following are equivalent:
\begin{enumerate}
\item 
\label{item-TFAE-M-is-A-perfect-item-l-exists}
The left adjoint $l$ of $f$ exists.
\item 
\label{item-TFAE-M-is-A-perfect-item-l-preserves-compact-objects}
The left adjoint $l$ of $f$ exists and restricts to $D_{c}(\B) \rightarrow D_{c}(\A)$.
\item 
\label{item-TFAE-M-is-A-perfect-item-l-is-A-perfect}
$M$ is $\A$-perfect. 
\item 
\label{item-TFAE-M-is-A-perfect-item-MbarA-is-the-left-adjoint}
$H^0\left((-) \bartimes_\B M^{\barA}\right)$ is the left adjoint of $f$
(see Definition \ref{defn-dualising-functors-in-modbar}).
\end{enumerate}
\end{enumerate}
\end{theorem}

\begin{proof}
\eqref{item-TFAE-M-is-B-perfect}: \quad
By the definition of adjunction we have 
$\homm_{D(\B)}\left(f(-), \bigoplus_\infty(-)\right) \simeq 
\homm_{D(\A)}\left(-, r \left(\bigoplus_\infty(-)\right)\right).$
Thus if $r$ is continuous then $f$ preserves compact objects, i.e. 
$\eqref{item-TFAE-M-is-B-perfect-item-right-adjoint-is-cts} 
\Rightarrow 
\eqref{item-TFAE-M-is-B-perfect-item-f-preserves-compact-objects}$. 
For any $a \in A$ we have $f(\aA) \simeq \aM$ 
in $D(\B)$ which shows 
$\eqref{item-TFAE-M-is-B-perfect-item-f-preserves-compact-objects}
\Rightarrow 
\eqref{item-TFAE-M-is-B-perfect-item-M-is-B-perfect}$.
When $M$ is $\B$-perfect $(-) \bartimes_\B M^\barB$ is homotopy right adjoint 
to $(-) \bartimes_\A M$ by Prop.~\ref{prps-homotopy-adjunction}, whence 
$\eqref{item-TFAE-M-is-B-perfect-item-M-is-B-perfect}
\Rightarrow
\eqref{item-TFAE-M-is-B-perfect-item-MbarB-is-the-right-adjoint}$. 
Finally, the implication 
$\eqref{item-TFAE-M-is-B-perfect-item-MbarB-is-the-right-adjoint}
\Rightarrow 
\eqref{item-TFAE-M-is-B-perfect-item-right-adjoint-is-cts}$ is
trivial since tensor functors are continuous.

\eqref{item-TFAE-M-is-A-perfect}: \quad If $l$ exists, then, by above,
it has to preserve compact objects since $f$ is continuous, so 
$\eqref{item-TFAE-M-is-A-perfect-item-l-exists}
\Rightarrow 
\eqref{item-TFAE-M-is-A-perfect-item-l-preserves-compact-objects}$. 
The implications 
$\eqref{item-TFAE-M-is-A-perfect-item-l-preserves-compact-objects}
\Rightarrow 
\eqref{item-TFAE-M-is-A-perfect-item-l-is-A-perfect}$
and 
$\eqref{item-TFAE-M-is-A-perfect-item-l-is-A-perfect}
\Rightarrow 
\eqref{item-TFAE-M-is-A-perfect-item-MbarA-is-the-left-adjoint}$
are proved analogously to 
$\eqref{item-TFAE-M-is-B-perfect-item-f-preserves-compact-objects}
\Rightarrow 
\eqref{item-TFAE-M-is-B-perfect-item-M-is-B-perfect}$
and 
$\eqref{item-TFAE-M-is-B-perfect-item-M-is-B-perfect}
\Rightarrow 
\eqref{item-TFAE-M-is-B-perfect-item-MbarB-is-the-right-adjoint}$
and again the implication 
$\eqref{item-TFAE-M-is-A-perfect-item-MbarA-is-the-left-adjoint}
\Rightarrow 
\eqref{item-TFAE-M-is-A-perfect-item-l-exists}$ is trivial. 
\end{proof}

\subsection{Homotopy adjunction for tensor functors}
\label{section-subsection-homotopy-adjunction-for-tensor-functors}

Let $f\colon D(\A) \rightarrow D(\B)$ be a continuous functor which 
has left and right adjoints $l$ and $r$ which are also continuous. 
By Prps.~\ref{prps-tfae-functor-is-continuous} we can enhance 
$f$, $l$, and $r$ by DG-bimodules. It is not, to our knowledge, always 
possible to lift the adjunctions of $f$, $l$, and $r$ to genuine adjunctions 
between the corresponding DG tensor functors. In this section 
we demonstrate that it is always possible to lift them 
to homotopy adjunctions in an economical and mutually compatible way. 

First, we demonstrate that when $M$ is $\B$-perfect 
the functors $\bigl((-)\bartimes_\A M, (-)\bartimes_\B M^\barB\bigr)$ form 
a homotopy adjoint pair. That is, there exist maps of bimodules 
in $\AmodbarA$ and $\BmodbarB$ such
that the corresponding natural transformations of tensor functors 
define the unit and the counit of the adjunction in homotopy 
categories. Similarly, when $M$ is $\A$-perfect
$\bigl((-)\bartimes_\B M^\barA, (-)\bartimes_\A M\bigr)$ form a
homotopy adjoint pair. 

It follows immediately from Lemma 
\ref{lemma-A-and-B-perfection-iff-the-evaluation-map-is-homotopy-equivalence}
and the Tensor-Hom adjunction that when $M$ is $\B$-perfect 
$\bigl((-)\bartimes_\A M, (-)\bartimes_\B M^\barB\bigr)$ are homotopy
adjoint. Similarly, when $M$ is $\A$-perfect
$\bigl((-)\bartimes_\B M^\barA, (-)\bartimes_\A M^{\barA\barA}\bigr)$ 
are homotopy adjoint. When $M$ is $\A$-perfect 
the natural map $M \rightarrow M^{\barA\barA}$ is an isomorphism in
$D(\AbimB)$, 
thus $\bigl((-)\bartimes_\B M^\barA, (-)\bartimes_\A M\bigr)$ are 
also homotopy adjoint.

However, the abstract fact of these functors being homotopy adjoint is 
not enough. We next write down certain natural lifts of adjunctions units 
and counits involved to the maps in $\modbar$ between the DG-bimodules 
involved.  We then compute the homotopies which arise when writing down 
relations between these maps. Our choice of natural lifts significantly 
reduces the number of choices involved and thus the number of higher
differentials in the explicit computations:
\begin{defn} 
\label{defn-homotopy-trace-maps}
Let $\A$ and $\B$ be DG-categories and let $M \in
\AmodbarB$. Define the \em homotopy trace maps \rm 
\begin{align}
\label{eqn-homotopy-trace-maps}
M \bartimes_\B M^\barA \xrightarrow{\trace} \A 
\quad\quad\quad 
\text{ and } 
\quad\quad\quad
M^\barB \bartimes_\A M \xrightarrow{\trace} \B
\end{align}
to be the Tensor-Hom adjunction counits applied to the diagonal
bimodules $\A$ and $\B$, respectively. 
\end{defn}

To define homotopy action maps and to work with the resulting homotopy
adjunctions we need to choose and fix the following homotopy inverses
and higher homotopies as per
\eqref{eqn-quiver-presentation-of-generating-cofibration-category}: 

\begin{defn}
\label{defn-homotopy-inverses-of-evaluation-maps} 
Let $\A$ and $\B$ be DG-categories.
\begin{enumerate}
\item For every $\B$-perfect $M \in \AmodbarB$ fix once and for all
a homotopy inverse 
\begin{align}
\barhom_\B(M,M) \xrightarrow{\zeta_\B} M \bartimes_\B M^\barB 
\end{align}
of the map $\eta_\B$
defined in Defn.~\ref{defn-the-natural-map-eta}. Furthermore, choose and fix 
\begin{align}
\label{eqn-omega-B-definition}
\omega_\B \in \barhom^{-1}_{\AbimA}\left(
\barhom_\B(M,M),
\barhom_\B(M,M)
\right) 
\quad \quad 
\text{ such that } d\omega_\B = \eta_\B \circ \zeta_\B - \id, 
\\
\omega'_\B \in \barhom^{-1}_{\AbimA}\left(
M \bartimes_\B M^\barB,
M \bartimes_\B M^\barB
\right) 
\quad \quad 
\text{ such that } d\omega'_\B = \id - \zeta_\B \circ \eta_\B, 
\\
\phi_\B \in \barhom^{-2}_{\AbimA}\left(
M \bartimes_\B M^\barB,
\barhom_\B(M,M)
\right) 
\quad \quad 
\text{ such that } d\phi_\B = \omega_B \circ \eta + \eta \circ \omega'_\B  
\end{align}
\item For every $\A$-perfect $M \in \AmodbarB$ fix once and for all
a homotopy inverse 
\begin{align}
\barhom_\A(M,M) \xrightarrow{\zeta_\A} M^\barA \bartimes_\A M 
\end{align}
of the map $\eta_\A$
defined in Defn.~\ref{defn-the-natural-map-eta}. Furthermore, 
choose and fix $\omega_\A$, $\omega'_\A$ and $\phi_\A$ analogously. 
\end{enumerate}
\end{defn}

We now define the homotopy action maps:
\begin{defn}
\label{defn-homotopy-action-maps} 
Let $\A$ and $\B$ be DG-categories. 
For all $\B$-perfect (resp. $\A$-perfect) $M \in \AmodbarB$ 
define the \em homotopy $\A$-action (resp. $\B$-action) map \rm 
\begin{align}
\label{eqn-homotopy-A-action-map}
& \A \xrightarrow{\action} M \bartimes_\B M^\barB, \\
\label{eqn-homotopy-B-action-map}
\text{resp. }\; & \B \xrightarrow{\action} M^\barA \bartimes_\A M 
\end{align}
to be the composition
\begin{align}
& \A \xrightarrow{\action} \barhom_\B(M,M) \xrightarrow{\zeta_\B}
M \bartimes_\B M^\barB,  \\
\text{resp. }\; 
& \B \xrightarrow{\action} \barhom_\A(M,M) \xrightarrow{\zeta_\A}
M^\barA \bartimes_\A M.
\end{align}
\end{defn}

Finally, we define the maps whose boundaries we prove below to 
be the difference between our homotopy adjunctions and genuine ones:

\begin{defn}
\label{defn-higher-homotopies-of-homotopy-adjunctions}
\begin{enumerate}
\item 
Define $\chi_\B,\chi_\A \in \barhom^{-1}_{\AbimB}(M,M)$ to 
be the compositions
\begin{align}
M 
\xrightarrow{\action\bartimes\id} 
\barhom_\B(M,M)\bartimes_\A M
\xrightarrow{\omega_\B \bartimes \id}
\barhom_\B(M,M)\bartimes_\A M
\xrightarrow{\ev}
M,
\\
M 
\xrightarrow{\action\bartimes\id} 
M \bartimes_\B \barhom_\A(M,M)
\xrightarrow{\id \bartimes \omega_\A}
M \bartimes_\B \barhom_\A(M,M)
\xrightarrow{\ev}
M.
\end{align}
\item 
Define $\chi'_{\B} \in \barhom^{-1}_{\BbimA}(M^\barB,M^\barB)$
and 
$\chi'_{\A} \in \barhom^{-1}_{\BbimA}(M^\barA,M^\barA)$
to be the $\BmodA$ maps 
\begin{small}
\begin{align}
M^\barB 
\xrightarrow{\id\bartimes\action} 
M^\barB \bartimes_\A \barhom_\B(M,M)
\xrightarrow{\id\bartimes\omega_\B}
M^\barB \bartimes_\A \barhom_\B(M,M)
\xrightarrow{\composition}
M^\barB
\\
M^\barA 
\xrightarrow{\action\bartimes\id} 
\barhom_\A(M,M) \bartimes_\B M^\barA 
\xrightarrow{\omega_\A\bartimes\id}
\barhom_\A(M,M) \bartimes_\B M^\barA 
\xrightarrow{\composition}
M^\barA
\end{align}
\end{small}
\item 
Define $\xi'_B \in  \barhom^{-1}_{\BbimB}(M^\barB \bartimes_\A M ,M^\barB \bartimes_\A M)$  and
$\xi'_A \in \barhom^{-1}_{\AbimA}(M \bartimes_\B M^\barA ,M \bartimes_\B M^\barA)$  to be the maps
\begin{align}
M^\barB \bartimes_\A M 
\xrightarrow{\omega_\B(\id_M)^\barB \bartimes \id - \id \bartimes \omega_\B(\id_M)} 
M^\barB \bartimes_\A M, 
\\
M \bartimes_\B M^\barA 
\xrightarrow{\id \bartimes \omega_\A(\id_M)^\barA - \omega_\A(\id_M) \bartimes \id} 
M \bartimes_\B M^\barA,
\end{align}
where e.g. $\omega_\B(\id_M) \in \barhom^{-1}_\B(M,M)$ is the image of $\id_M$ under the 
map $\omega_\B$ defined in \eqref{eqn-omega-B-definition}. 
\end{enumerate}
\end{defn}

It is now convenient to adopt the functorial notation explained  
in the beginning of this section. Let $F$ denote the bimodule 
$M \in \AmodbarB)$ and $R$ and $L$ denote the bimodules 
$M^\barB$ and $M^\barA$ in $\BmodbarA$. 

We introduce further conventions regarding the diagonal bimodules 
$\A$ and $\B$:
\begin{itemize}
\item When they occur on their own, they are denoted as $\id_\A$ or
$\id_\B$, respectively.
\item When working with tensor products of several bimodules, we
suppress all appearance of diagonal bimodules in them by 
implicit use of the homotopy equivalences
$\alpha$ and $\beta$ defined in
\S\ref{section-on-non-invertibility-of-A-bartimes-M-to-M}. This
is analogous to implicit use of the equalities $\Phi \circ \id =
\id \circ \Phi = \Phi$ which exist for any functor $\Phi$.

\item More specifically, given a map whose source
is a diagonal bimodule we write it as applied ``in between'' 
two factors of a tensor product of bimodules. This means first 
applying $\beta$ to either of the two factors. It doesn't 
matter which  --- the resulting map is the same. 
The corresponding map in the ordinary module category applies 
$\Delta$ to the bar complex in the middle. 

For example, we write $ FR \xrightarrow{F\action R} FRFR $
to denote the composition
$$ M^\barB \bartimes_\A M 
\xrightarrow{\beta \bartimes \id} 
M^\barB \bartimes_\A \A \bartimes_\A M 
\xrightarrow{\id \bartimes \action \bartimes \id}
M^\barB \bartimes_\A M \bartimes_\B M^\barB \bartimes_\A M. $$
We could have used $\id \bartimes \beta$ instead of $\beta \bartimes
\id$ as they both correspond to the $\BmodB$ map 
$$ \barB \otimes_\B \homm_\B(M \otimes \barB, \B) \otimes_\A \barA \otimes_\A M \otimes_\B \barB
\xrightarrow{\tau \otimes \id \otimes \Delta \otimes \id \otimes \tau}
\homm_\B(M \otimes \barB, \B) \otimes_\A \barA \otimes_\A \barA
\otimes_\A M. $$

\item Similarly, when we apply a map whose target is a diagonal bimodule  
to a part of a tensor product of bimodules, we suppress this diagonal bimodule
in the resulting product.  This stands for using $\alpha$  on the product of this diagonal 
bimodule with one of its two possible neighbouring factors. When both are present 
we choose the left one. This choice matters: in the underlying ordinary 
module category it chooses which of the two bar complexes 
in the middle to contract with the map $\tau$.  

For example, we write $FLFR \xrightarrow{F\trace L} FR$ to denote the 
composition 
$$ M^\barB \bartimes_\A M \bartimes_\B M^\barA \bartimes_\A M
\xrightarrow{\id \bartimes \trace \bartimes \id}
M^\barB \bartimes_\A \A \bartimes_\A M
\xrightarrow{\alpha \bartimes \id}
M^\barB \bartimes_\A M.$$
The map $\alpha \bartimes \id$ corresponds to the $\BmodB$ map 
$$  
\barB  \otimes_\B \homm_\B(M \otimes \barB, \B) \otimes_\A \barA \otimes_\A \barA \otimes_\A M \otimes_\B \barB
\xrightarrow{\tau \otimes \id \otimes \tau \otimes \id \otimes \id \otimes \tau}
\homm_\B(M \otimes \barB, \B) \otimes_\A \barA \otimes_\A M, 
$$
while $\id \bartimes \alpha$ corresponds to the map 
$\tau \otimes \id \otimes \id \otimes \tau \otimes \id \otimes \tau$. 

\item Finally, we have a special convention regarding the maps
$RFR \xrightarrow{R\trace} R$ and $LFL \xrightarrow{\trace L} L$. 
By the general rules laid out above $RFR \xrightarrow{R\trace} R$ 
should denote the map 
\begin{align}
\label{eqn-Rtrace-wrong-way-around}
M^\barB \bartimes_\A M \bartimes_\B M^\barB 
\xrightarrow{\trace \bartimes \id} \B \bartimes_\B M^\barB
\xrightarrow{\alpha} M^\barB. 
\end{align}
Instead, we denote by $R\trace$ the map  
\begin{align}
\label{eqn-Rtrace-right-way-around}
M^\barB \bartimes_\A M \bartimes_\B M^\barB 
\xrightarrow{\id \bartimes \gamma \bartimes \id} 
M^\barB \bartimes_\A \barhom_\B(\B,M) \bartimes_\B M^\barB 
\xrightarrow{\composition} M^\barB. 
\end{align}
Note that it follows from Lemma 
\ref{lemma-explicit-description-of-units-counits-of-Tensor-Hom-adjunction}
that the maps \eqref{eqn-Rtrace-wrong-way-around} and
\eqref{eqn-Rtrace-right-way-around} are homotopic in $\BmodbarA$ and 
thus are isomorphic in $D(\BbimA)$.

Similarly, we denote by $\trace L$ the map 
\begin{align}
M^\barA \bartimes_\A M \bartimes_\B M^\barA
\xrightarrow{ \id \bartimes \gamma \bartimes \id }
M^\barA \bartimes_\A \barhom_\A(\A,M) \bartimes_\B M^\barA
\xrightarrow{\composition}
M^\barA. 
\end{align}

In a tensor product of several maps we always evaluate all
instances of the maps $R \trace$ and $\trace L$ first. 
For example, $FRFR \xrightarrow{\trace\;\trace} \id$
denotes the map 
$$
M^\barB \bartimes_\A M \bartimes_\B M^\barB \bartimes_\A M
\xrightarrow{\id \bartimes \gamma \bartimes \id \bartimes \id} 
M^\barB \bartimes_\A \barhom_\B(\B,M) \bartimes_\B M^\barB \bartimes_\A M 
\xrightarrow{\composition \bartimes \id}
M^{\barB} \bartimes_\A M 
\xrightarrow{\trace}
\B
$$
and not the map 
$$
M^\barB \bartimes_\A M \bartimes_\B M^\barB \bartimes_\A M
\xrightarrow{\id \bartimes \id \bartimes \trace} 
M^\barB \bartimes_\A \barhom_\B(\B,M) \bartimes_\B \B 
\xrightarrow{\id \bartimes \alpha} 
M^\barB \bartimes_\A M 
\xrightarrow{\trace}
\B.
$$
\end{itemize}

\begin{prps}
\label{prps-homotopy-adjunction}
Let $\A$ and $\B$ be DG-categories and let $M \in \AmodbarB$. 

\begin{enumerate}
\item
\label{item-LF-homotopy-adjunction}
If $M$ is $\A$-perfect, 
we have in $\AmodbarB$ and $\BmodbarA$,
respectively:
\begin{align}
\label{eqn-F-FLF-F-composition}
F \xrightarrow{\action F} FLF \xrightarrow{F \trace} F \quad & = 
\id + d\chi_\A, 
\\
\label{eqn-L-LFL-L-composition}
L \xrightarrow{L\action} LFL \xrightarrow{\trace L} L \quad & = 
\id + d\chi'_\A.
\end{align}
\item
\label{item-FR-homotopy-adjunction}
If $M$ is $\B$-perfect, we have in $\AmodbarB$ and $\BmodbarA$,
respectively:
\begin{align}
\label{eqn-F-FRF-F-composition}
F \xrightarrow{F\action} FRF \xrightarrow{\trace F} F \quad & = 
\id + d\chi_\B, 
\\
\label{eqn-R-RFR-R-composition}
R \xrightarrow{\action R} RFR \xrightarrow{R\trace} R \quad & = 
\id + d\chi'_\B.
\end{align}

\end{enumerate}
\end{prps}
\begin{proof}
We only prove the assertion
\eqref{item-LF-homotopy-adjunction}, 
the assertion
\eqref{item-FR-homotopy-adjunction}
is proved similarly. 

Consider the following diagram of morphisms in $\AmodbarB$:
\begin{equation*}
\begin{tikzcd}[column sep={2cm},
execute at end picture={\node at ([xshift=-5mm,yshift=3mm]-3,0) [circle,draw]{$A$};  }]
M
\ar{rr}{\id \bartimes \action}
\ar{ddrr}[']{\id \bartimes \action}
& &
M\bartimes_\B M^\barA \bartimes_\A M
\ar{rr}{\trace\bartimes\id}
\ar{dd}{\id \bartimes \eta_\A}
& &
\A \bartimes_\A M  
\ar{dd}{\alpha}
\ar{dl}{\id\bartimes\gamma}
\\ 
& & &
\A \bartimes_\A \barhom_\A(\A,M) 
\ar{dr}{\ev}
&  
\\
& &
M \bartimes_\B \barhom_\A(M,M) 
\ar{rr}{\ev}
& &
M.
\end{tikzcd}
\end{equation*}

The triangle $(A)$ commutes up to 
$d \left(\left(\id\bartimes \omega_\A\right) \circ \left( \id \bartimes \action \right) \right)$. 
The rest of the diagram commutes: the pentagon --- by definition of the evaluation map, 
and the triangle --- by direct verification. Therefore the perimeter 
of the diagram commutes up to $d\chi_\A$. 

The composition of the lower-left half of the perimeter is readily 
verified to be the identity morphism. On the other hand, the
composition of the upper-right half of the perimeter is 
the LHS of \eqref{eqn-F-FLF-F-composition}. 
Since the perimeter of the diagram commutes up to $d\chi_\A$, the
equality in \eqref{eqn-F-FLF-F-composition} follows. 

Next, consider the following diagram of morphisms in $\BmodbarA$:
\begin{equation*}
\begin{tikzcd}[column sep={2cm}, 
execute at end picture={\node at ([xshift=-5mm,yshift=2mm]-3,0) [circle,draw]{$1$};  }]
M^\barA 
\ar{rr}{\action\bartimes\id}
\ar{ddrr}[']{\action\bartimes\id}
& &
M^\barA\bartimes_\A M \bartimes_\B M^\barA
\ar{rr}{\quad \quad \quad \; \quad \id \bartimes \gamma \bartimes \id}
\ar{dd}{\eta_\A\bartimes\id}
& &
M^\barA \bartimes_\A \barhom_\A(\A,M) \bartimes_\B M^\barA
\ar{dd}{\composition}
\\ 
& & & &  
\\
& &
\barhom_\A(M,M) \bartimes_\B M^\barA
\ar{rr}{\composition}
& &
M^\barA.
\end{tikzcd}
\end{equation*}

Similar to the above, the lower-left half of the perimeter composes to $\id$, 
while the upper-right --- to the LHS of \eqref{eqn-L-LFL-L-composition}. 
The triangle $(A)$ commutes up to $d\left( \left( \omega_\A \bartimes \id \right) \circ \left( \action \bartimes \id\right)\right)$ 
and the quadrilateral commutes by the associativity of the composition in $\modbarA$. 
It follows that the perimeter commutes up to $d\chi'_\A$, whence \eqref{eqn-L-LFL-L-composition}. 
\end{proof}

\begin{cor}
\label{cor-FR-FRFR-FR-equals-zero-and-other}
Let $\A$ and $\B$ be DG-categories and let $M \in \AmodbarB$. 
\begin{enumerate}
\item If $M$ is $\A$-perfect, we have in $\BmodbarB$,
respectively:
\begin{align}
\label{eqn-LF-LFLF-LF-composition}
LF \xrightarrow{L\action F}  LFLF \xrightarrow{\trace LF - LF\trace} LF
\quad = \quad d\xi'_\A,
\end{align}
\item If $M$ is $\B$-perfect, we have in $\AmodbarA$,
respectively:
\begin{align}
\label{eqn-FR-FRFR-FR-composition}
FR \xrightarrow{F\action R}  FRFR \xrightarrow{FR\trace  - \trace FR} FR 
\quad = \quad d\xi'_\B,
\end{align}
\end{enumerate}
\end{cor}
\begin{proof}
By Prop.~\ref{prps-homotopy-adjunction} the LHS of 
\eqref{eqn-LF-LFLF-LF-composition} equals 
$d(\chi'_\A F - L \chi_\A)$, and $\chi'_\A F - L \chi_\A$ is the composition
\begin{align*}
M \bartimes_\B M^\barA 
\xrightarrow{\id \bartimes (\omega_\A \circ \action) \bartimes \id} 
M \bartimes_\B \barhom_\A (M,M) \bartimes_\B M^\barA
\xrightarrow{	\id \bartimes \composition - \ev \bartimes \id}
M \bartimes_\B M^\barA 
\end{align*}
which is precisely $\xi'_\A$. The other assertion works out similarly.
\end{proof}

Next we treat the two compositions dual to \eqref{eqn-LF-LFLF-LF-composition}
and \eqref{eqn-FR-FRFR-FR-composition}: 
\begin{align}
\label{eqn-FL-FLFL-FL-composition}
FL \xrightarrow{FL\action - \action FL}  FLFL \xrightarrow{F \trace L} FL, 
\\
\label{eqn-RF-RFRF-RF-composition}
RF \xrightarrow{\action RF - RF\action}  RFRF \xrightarrow{R \trace F} RF. 
\end{align}
Here things do not work
out so well, because, for instance, in 
\eqref{eqn-FL-FLFL-FL-composition}
after contracting 
$M^\barA \bartimes M \bartimes M^\barA \bartimes M$ to 
$M^\barA \bartimes \barhom_\A(\A,\A) \bartimes M$ 
the map $F{\trace}L$ composes $\barhom_\A(\A,\A)$ with 
$M^\barA$ on the left. 
This works well when composed with $FL\action$, but not 
with ${\action}FL$. Thus we do not get something as simple as 
$d(F\chi'_\A - \chi_\A L)$. 

We can fix this by composing
\eqref{eqn-FL-FLFL-FL-composition} with the homotopy equivalence 
$M^\barA \bartimes M \xrightarrow{\eta} \barhom_\A(M,M)$. Then it doesn't
matter whether in $M^\barA \bartimes \barhom_\A(\A,\A) \bartimes M$ 
we contract $\barhom_\A(\A,\A)$ to the left or to the right. 
And any boundary which lifts 
$\eta \circ \eqref{eqn-FL-FLFL-FL-composition}$ induces a lift of 
\eqref{eqn-FL-FLFL-FL-composition} itself. 
More generally: 

\begin{lemma}
\label{lemma-lifting-to-a-boundary-via-homotopy-equivalence}
Let $E,F,F' \in \AmodbarB$ and let $F \xrightarrow{s} F'$ and $F'
\xrightarrow{s'} F$ be closed maps of degree $0$ such that there exists 
$F \xrightarrow{t} F$ with $dt = \id - s' \circ s$. 

\begin{enumerate}
\item
\label{eqn-homotopy-retract-induced-lift-stage-one} 
Let $E \xrightarrow{f} F$ be a closed map and let $E
\xrightarrow{g} F'$ be such that 
$dg = s \circ f$. Then 
$$ d(t \circ f + s' \circ g) = f. $$
We say that $t \circ f + s' \circ g$ is the lift of $f$ induced by 
the lift $g$ of $s \circ f$.  
\item
\label{eqn-homotopy-retract-induced-lift-stage-two} 
Let $f$ and $g$ be as above and let $h$ be another lift of $f$. Let
$E \xrightarrow{j} F'$ be such that 
$dj = g - s \circ h$. Then 
$$ d(s' \circ j - t\circ h) = t \circ f + s' \circ g - h.$$
We say that $s' \circ j - t\circ h$ is the lift of $t \circ f + s' \circ g - h$ induced by the lift $j$ of $g - s \circ h$. 
\end{enumerate}
\end{lemma}
\begin{proof}
Direct computation. 
\end{proof}

\begin{prps}
\label{prps-FL-FLFL-FL-and-RF-RFRF-RF-composed-with-eta}
Let $\A$ and $\B$ be DG-categories and let $M \in \AmodbarB$. 
\begin{enumerate}
\item  
\label{item-FL-FLFL-FL-eta-composition}
If $M$ is $\A$-perfect, then in $\AmodbarA$ the map $\eta \circ \eqref{eqn-FL-FLFL-FL-composition}$ is the boundary of 
\begin{align}
\label{eqn-FL-FLFL-FL-eta-composition-lift}
 \composition \circ \Bigl((\omega  \bartimes \eta) \circ (\action
\bartimes \id) - (\eta \bartimes \omega) \circ (\id \bartimes \action)\Bigr).
\end{align}
Define $\xi_\A$ to be the induced lift of \eqref{eqn-FL-FLFL-FL-composition}
as per Lemma \ref{lemma-lifting-to-a-boundary-via-homotopy-equivalence}\eqref{eqn-homotopy-retract-induced-lift-stage-one} with $s = \eta_\A$, $s' = \zeta_\A$, $t = \omega'_\A$, $f = \eqref{eqn-FL-FLFL-FL-composition}$, and $g = \eqref{eqn-FL-FLFL-FL-eta-composition-lift}$. 
\item 
\label{item-RF-RFRF-RF-eta-composition}
If $M$ is $\B$-perfect, then in $\BmodbarB$ the map $\eta \circ \eqref{eqn-RF-RFRF-RF-composition}$ is the boundary of 
\begin{align}
\label{eqn-RF-RFRF-RF-eta-composition-lift}
\composition \circ \Bigl( (\eta \bartimes \omega) \circ (\id
\bartimes \action) - (\omega \bartimes \eta) \circ (\action \bartimes
\id)\Bigr). 
\end{align}
Define $\xi_\B$ to be the induced lift of \eqref{eqn-RF-RFRF-RF-composition}
as per Lemma \ref{lemma-lifting-to-a-boundary-via-homotopy-equivalence}\eqref{eqn-homotopy-retract-induced-lift-stage-one} with $s = \eta_\B$, $s' = \zeta_\B$, $t = \omega'_\B$, $f = \eqref{eqn-RF-RFRF-RF-composition}$, and 
$g =  \eqref{eqn-RF-RFRF-RF-eta-composition-lift}$. 
\end{enumerate} 
\end{prps}
\begin{proof}
We only prove the assertion \eqref{item-FL-FLFL-FL-eta-composition}, the
assertion \eqref{item-RF-RFRF-RF-eta-composition} is proved analogously. 
Consider the diagram:
\begin{equation*}
\begin{tikzcd}[column sep={3cm}]
M^\barA \bartimes_\A M
\ar{r}{\action \bartimes \id - \id \bartimes \action}
\ar{d}{\eta}
&
M^\barA \bartimes_\A M \bartimes_\B
M^\barA \bartimes_\A M
\ar{d}{\eta \bartimes \eta}
\ar{r}{\composition \bartimes \id}
& 
M^\barA \bartimes_\A M
\ar{d}{\eta}
\\
\barhom_\A(M,M) 
\ar{r}{\action \bartimes \id - \id \bartimes \action}
&
\barhom_\A(M,M) \bartimes_\B \barhom_\A(M,M)
\ar{r}{\composition}
& 
\barhom_\A(M,M) 
\end{tikzcd}
\end{equation*}
The left rectangle commutes up to
$$ d
\Bigl((\omega  \bartimes \eta) \circ (\action
\bartimes \id) - (\eta \bartimes \omega) \circ (\id \bartimes
\action)\Bigr), $$
while the right rectangle genuinely commutes.  
The composition in the bottom row is zero. 
As the composition in the top row is 
\eqref{eqn-FL-FLFL-FL-composition}, the assertion 
\eqref{item-FL-FLFL-FL-eta-composition} follows. 
\end{proof}

\begin{cor}
\label{cor-RF-RFRF-RF-equals-dFxiB-and-other}
Let $\A$ and $\B$ be DG-categories and let $M \in \AmodbarB$. 
\begin{enumerate}
\item 
\label{item-FL-FLFL-FL-composition-lift}
If $M$ is $\A$-perfect, we have in $\AmodbarA$
\begin{align}
\label{eqn-FL-FLFL-FL-composition-lift}
FL \xrightarrow{FL\action - \action FL}  FLFL \xrightarrow{F \trace L} FL 
\quad = \quad d \xi_\A.
\end{align}
\item 
\label{item-RF-RFRF-RF-composition-lift}
If $M$ is $\B$-perfect, we have in $\BmodbarB$,
respectively:
\begin{align}
\label{eqn-RF-RFRF-RF-composition-lift}
RF \xrightarrow{\action RF - RF\action}  RFRF \xrightarrow{R \trace F} RF
\quad = \quad d \xi_\B .
\end{align}
\end{enumerate}
\end{cor}

\subsection{Canonical twisted complexes associated to a homotopy adjunction}

In Corollaries \ref{cor-FR-FRFR-FR-equals-zero-and-other}
and 
\ref{cor-RF-RFRF-RF-equals-dFxiB-and-other}
we've shown that each of the four compositions in 
\eqref{eqn-LF-LFLF-LF-composition}-\eqref{eqn-RF-RFRF-RF-composition} is a boundary. 
We thus obtain four natural three-term twisted complexes e.g. 
\begin{equation}
\label{eqn-FL-FLFL-FL-twisted-complex}
\begin{tikzcd}[column sep={3cm},row sep={1.5cm}] 
FL
\ar{r}{FL{\action}-{\action}FL}
\ar[bend left=20,dashed]{rr}{-\xi_\A}
&
FLFL
\ar{r}{F{\trace}L}
&
FL. 
\end{tikzcd}
\end{equation}

In this section, we show that these extend to natural four-term 
twisted complexes whose convolutions are the squares of the 
twist $T$, the dual twist $T'$, the co-twist $C$ and the dual co-twist
$C'$ of $F$. 

\begin{defn}
\label{defn-upsilon}
Let  $\A$ and $\B$ be DG-categories and let $M \in \AmodbarB$ . 
\begin{enumerate}
\item If $M$ is $\B$-perfect, define the degree $-1$ morphism
$$ \id_\A \xrightarrow{\upsilon_\B} RFRF $$ 
to be the composition
\begin{equation}
\A 
\xrightarrow{ \pi_\A } 
\A \bartimes_\A \A
\xrightarrow{\action \bartimes \action}
M \bartimes_\B M^\barB \bartimes_\A M \bartimes_\B M^\barB
\end{equation}
\item If $M$ is $\A$-perfect, define the degree $-1$ morphism
$$ \id_\B \xrightarrow{\upsilon_\A} FLFL $$
to be the composition
\begin{equation}
\B 
\xrightarrow{ 
\text{-}\pi_\B
} 
\B \bartimes_\B \B
\xrightarrow{\action \bartimes \action}
M^\barA \bartimes_\A M \bartimes_\B M^\barA \bartimes_\A M. 
\end{equation}
\end{enumerate}
Here $\pi_\A$ and $\pi_\B$ are the maps defined in 
\eqref{defn-beta^r_B-minus-beta^l_B-boundary-lift}. 
\end{defn}

\begin{prps}
\label{prps-defining-the-maps-nu_A-and-nu_B}
Let  $\A$ and $\B$ be DG-categories and let $M \in \AmodbarB$ . 
\begin{enumerate}
\item
\label{item-defining-nu_A}
 If $M$ is $\A$-perfect, then in $\BmodbarB$ the map 
\begin{align}
\label{eqn-FtrL-upsilon-plus-xiA-action-eta-version}
\eqref{eqn-FL-FLFL-FL-eta-composition-lift} \circ \action \;-\; 
\eta_\A \circ F{\trace}L\circ \upsilon_\A
\end{align}
is the boundary of 
\begin{align}
\label{eqn-the-lift-of-FtrL-upsilon-plus-xiA-action-eta-version}
\composition \circ \bigl(\omega \bartimes (\eta \circ \zeta) + 
\id \bartimes \omega\bigr) \circ (\action \bartimes \action) \circ \pi
\quad - \quad \composition \circ (\omega \bartimes \omega) \circ (\action \bartimes \action) \circ \beta^r.
\end{align}
Define $\nu_\A$ to be the induced lift of $\xi_\A \circ \action - F{\trace}L \circ \upsilon_\A$ 
as per Lemma \ref{lemma-lifting-to-a-boundary-via-homotopy-equivalence}\eqref{eqn-homotopy-retract-induced-lift-stage-two} with $s = \eta_\A$, $s' = \zeta_\A$, $t = \omega'_\A$, $f = \eqref{eqn-FL-FLFL-FL-composition} \circ \action$, $g = \eqref{eqn-FL-FLFL-FL-eta-composition-lift} \circ \action$, $h = F{\trace}L \circ \upsilon_\A$, and $j = \eqref{eqn-the-lift-of-FtrL-upsilon-plus-xiA-action-eta-version}$.  
\item 
\label{item-defining-nu_B}
If $M$ is $\B$-perfect, then in $\AmodbarA$ the map 
$$ \eqref{eqn-RF-RFRF-RF-eta-composition-lift} \circ \action - \eta_\B \circ F{\trace}R \circ \upsilon_\B $$
is the boundary of
\begin{align}
\label{eqn-the-lift-of-FtrR-upsilon-plus-xiB-action-eta-version}
\composition \circ (\omega \bartimes \omega) \circ (\action \bartimes
\action) \circ \beta^l 
\quad - \quad
\composition \circ \bigl((\eta \circ \zeta) \bartimes \omega  + 
\omega \bartimes \id\bigr) \circ (\action \bartimes \action) \circ \pi
 \end{align}
Denote by $\nu_\B$ the induced lift of 
$\xi_\B \circ \action - F{\trace}R \circ \upsilon_\B$ 
as per Lemma \ref{lemma-lifting-to-a-boundary-via-homotopy-equivalence}\eqref{eqn-homotopy-retract-induced-lift-stage-two} with $s = \eta_\B$, $s' = \zeta_\B$, $t = \omega'_\B$, $f = \eqref{eqn-RF-RFRF-RF-composition} \circ \action$, 
$g =  \eqref{eqn-RF-RFRF-RF-eta-composition-lift} \circ \action$, $h = F{\trace}R \circ \upsilon_\B$, and $j = \eqref{eqn-the-lift-of-FtrR-upsilon-plus-xiB-action-eta-version}$. 
\end{enumerate}
\end{prps}
\begin{proof}
We only prove the assertion
\eqref{item-defining-nu_A} as the proof
of \eqref{item-defining-nu_B} is similar. 
Applying the differential 
\eqref{eqn-the-lift-of-FtrL-upsilon-plus-xiA-action-eta-version}
we obtain
\begin{align*}
& \composition \circ \bigl(
(\eta \circ \zeta - \id) \bartimes (\eta \circ \zeta) +
\id \bartimes (\eta \circ \zeta - \id) \bigr) \circ (\action \bartimes \action) \circ \pi
\quad - \\
- \quad & \composition \circ \bigl(\omega \bartimes (\eta \circ \zeta) + 
\id \bartimes \omega\bigr) \circ (\action \bartimes \action) \circ
(\beta^r  - \beta^l) \quad - 
\\
- \quad & \composition \circ 
\bigl((\eta \circ \zeta - \id) \bartimes \omega - 
\omega \bartimes (\eta \circ \zeta - \id)
\bigr)
\circ (\action \bartimes \action) \circ \beta^r 
\end{align*}
which simplifies to 
\begin{align*}
& \composition \circ \bigl(
(\eta \circ \zeta) \bartimes (\eta \circ \zeta) - 
\id \bartimes \id \bigr) \circ (\action \bartimes \action) \circ \pi
\quad + \\
+ \quad & \composition \circ \bigl(\omega \bartimes (\eta \circ \zeta) + 
\id \bartimes \omega\bigr) \circ (\action \bartimes \action) \circ
\beta^l
\quad - \\
- \quad & \composition \circ 
\bigl((\eta \circ \zeta) \bartimes \omega + 
\omega \bartimes \id
\bigr)
\circ (\action \bartimes \action) \circ \beta^r.
\end{align*}

By Prop.~\ref{prps-alpha-and-tensoring-f-with-identity}\eqref{item-f-and-beta-commute}
the maps 
$\composition \circ (\id \bartimes \omega) \circ (\action \bartimes \action) 
\circ \beta^l$ and $\composition \circ (\omega \bartimes \id)
\circ (\action \bartimes \action) \circ \beta^r$ 
are readily seen to both be equal to the map $\omega \circ \action$. 
On the other hand, the map 
$\composition \circ (\action \bartimes \action) \circ \pi$ corresponds
in $\BmodB$ to the map 
$$ \barB \otimes \barB
\xrightarrow{\mu} \barB \xrightarrow{\tau} \B \xrightarrow{\action}
\homm_\A(M,M) \xrightarrow{(-)\circ (\tau \otimes \id)}
\homm_\A(\barA \otimes M, M) $$
which vanishes as $\tau \circ \mu = 0$. 

Hence the expression above for the boundary of 
\eqref{eqn-the-lift-of-FtrL-upsilon-plus-xiA-action-eta-version}
simplifies further to 
\begin{align*}
& \composition \circ \bigl(
(\eta \circ \zeta) \bartimes (\eta \circ \zeta) \bigr) \circ (\action \bartimes \action) \circ \pi
\quad + \\
+ \quad & \composition \circ \bigl(\omega \bartimes (\eta \circ \zeta) \bigr) 
\circ (\action \bartimes \action) \circ \beta^l
\quad - \\
- \quad & \composition \circ 
\bigl((\eta \circ \zeta) \bartimes \omega \bigr)
\circ (\action \bartimes \action) \circ \beta^r.
\end{align*}
which is precisely the map 
\eqref{eqn-FtrL-upsilon-plus-xiA-action-eta-version} as desired. 
\end{proof}

\begin{theorem}
\label{theorem-canonical-twisted-complex-associated-to-homotopy-adjunction}
Let $\A$ and $\B$ be DG-categories and let $M \in \AmodbarB$. 
\begin{enumerate}
\item If $M$ is $\B$-perfect, the following is 
a twisted complex over $\BmodbarB$
\begin{equation}
\label{eqn-FR-FRFR-FR-Id-twisted-complex}
\begin{tikzcd}[column sep={3cm},row sep={1.5cm}] 
FR
\ar{r}{F{\action}R}
\ar[bend left=20,dashed]{rr}{\xi'_\B}
&
FRFR
\ar{r}{FR\trace - \trace FR}
&
FR
\ar{r}{\trace}
&
\underset{\degzero}{\id_\B},
\end{tikzcd}
\end{equation}
and the following is a twisted complex over $\AmodbarA$
\begin{equation}
\label{eqn-Id-RF-RFRF-RF-twisted-complex}
\begin{tikzcd}[column sep={3cm},row sep={1.5cm}] 
\underset{\degzero}{\id_\A}
\ar{r}{\action}
\ar[bend left=20,dashed]{rr}{-\upsilon_\B}
\ar[bend left=25,dashed]{rrr}{\nu_\B}
&
RF
\ar{r}{\action{RF} - RF\action}
\ar[bend right=20,dashed]{rr}{\xi_\B}
&
RFRF
\ar{r}{F\trace{R}}
&
RF. 
\end{tikzcd}
\end{equation}

\item 
If $M$ is $\A$-perfect, the following is 
a twisted complex over $\AmodbarA$
\begin{equation}
\label{eqn-LF-LFLF-LF-Id-twisted-complex}
\begin{tikzcd}[column sep={3cm},row sep={1.5cm}] 
LF
\ar{r}{L{\action}F}
\ar[bend left=20,dashed]{rr}{\xi'_\A}
&
LFLF
\ar{r}{\trace{LF} - LF\trace}
&
LF
\ar{r}{\trace}
&
\underset{\degzero}{\id_\A},
\end{tikzcd}
\end{equation}
and the following is a twisted complex over $\BmodbarB$
\begin{equation}
\label{eqn-Id-FL-FLFL-FL-twisted-complex}
\begin{tikzcd}[column sep={3cm},row sep={1.5cm}] 
\underset{\degzero}{\id_\B}
\ar{r}{\action}
\ar[bend left=20,dashed]{rr}{-\upsilon_\A}
\ar[bend left=25,dashed]{rrr}{\nu_\A}
&
FL
\ar{r}{{FL}\action - \action{FL}}
\ar[bend right=20,dashed]{rr}{\xi_\A}
&
FLFL
\ar{r}{F\trace{L}}
&
FL. 
\end{tikzcd}
\end{equation}

\end{enumerate}
\end{theorem}

\begin{proof}

We only prove that \eqref{eqn-FR-FRFR-FR-Id-twisted-complex}
and \eqref{eqn-Id-RF-RFRF-RF-twisted-complex} are twisted complexes,
the proofs for \eqref{eqn-LF-LFLF-LF-Id-twisted-complex}
and \eqref{eqn-Id-FL-FLFL-FL-twisted-complex}
are similar.  

The definition of a twisted complex
\cite[\S{3.1}]{AnnoLogvinenko-SphericalDGFunctors}
implies that \eqref{eqn-FR-FRFR-FR-Id-twisted-complex} is 
a twisted complex if and only if: 
\begin{enumerate}
\item 
\label{item-FR-FRFR-FR-composition}
$\left(FR\trace-\trace FR)\right) \circ F{\action}R = d\xi'_\B$, 
\item 
\label{item-FRFR-FR-Id-composition}
$\trace \circ \left(FR\trace-\trace FR\right) = 0$, 
\item 
\label{item-FR-FRFR-FR-Id-composition}
$\trace \circ \xi'_\B = 0$, 
\end{enumerate}

The identity \eqref{item-FR-FRFR-FR-composition} is the content of Corollary
\ref{cor-FR-FRFR-FR-equals-zero-and-other}. For the identity 
\eqref{item-FRFR-FR-Id-composition}, observe that following our
conventions $FRFR \xrightarrow{\trace \circ FR\trace} \id$
denotes the composition  
\begin{align}
\label{eqn-FRtr-tr-composition-in-bimodules}
M^\barB \bartimes M \bartimes M^\barB \bartimes M 
\xrightarrow{\id \bartimes \gamma \bartimes \id^2}
M^\barB \bartimes \barhom_\B(\B,M) \bartimes M^\barB \bartimes M 
\xrightarrow{\composition \bartimes \id}
M^\barB \bartimes M 
\xrightarrow{\ev}
\B
\end{align}
On the other hand, $FRFR \xrightarrow{\trace \circ \trace FR} \id$
denotes the composition
\begin{align}
\label{eqn-trFR-tr-composition-in-bimodules}
M^\barB \bartimes M \bartimes M^\barB \bartimes M 
\xrightarrow{\id^2 \bartimes \ev}
M^\barB \bartimes M \bartimes \B 
\xrightarrow{\id \bartimes \alpha}
M^\barB \bartimes M 
\xrightarrow{\ev}
\B
\end{align}
Finally, recall that the map $M \bartimes \B \xrightarrow{\alpha} M$
is the same as the composition $$M \bartimes \B \xrightarrow{\gamma
\bartimes \id} \barhom_\B(\B,M) \bartimes \B \xrightarrow{\ev} M.$$ 
It follows that \eqref{eqn-FRtr-tr-composition-in-bimodules}
and \eqref{eqn-trFR-tr-composition-in-bimodules} are the same map, 
and thus \eqref{item-FRFR-FR-Id-composition} holds as required. 

For the identity \eqref{item-FR-FRFR-FR-Id-composition}, we observe that, by definition, $\trace \circ \xi'_\B$ is the map 
$$
\barhom_\B(M,\B) \bartimes_\A M 
\xrightarrow{\left(- \circ \omega_\B(\id_M)\right) \bartimes \id - \id \bartimes \omega_\B(\id_M)} 
\barhom_\B(M,\B) \bartimes_\A M 
\xrightarrow{\ev}
\B
$$
which is zero by the associativity of the composition of morphisms in the category $\modbarB$. 

Similarly, 
\eqref{eqn-Id-RF-RFRF-RF-twisted-complex} is 
a twisted complex if and only if: 
\begin{enumerate}
\item 
\label{item-RF-RFRF-RF-composition}
$F{\trace}R \circ \left({\action}RF - RF{\action}\right)  = d\xi_\B$, 
\item 
\label{item-Id-RFRF-RF-composition}
$\left({\action}RF - RF{\action}\right) \circ \action = d\upsilon_\B$, 
\item 
\label{item-xi_A-act-minus-FtrL-upsilon_A}
$\xi_\B \circ \action - F{\trace}L \circ \upsilon_\B = d\nu_\B$. 
\end{enumerate}

The identities \eqref{item-RF-RFRF-RF-composition} and 
\eqref{item-xi_A-act-minus-FtrL-upsilon_A} follow from 
Cor.~\ref{cor-RF-RFRF-RF-equals-dFxiB-and-other}
and Prop.~\ref{prps-defining-the-maps-nu_A-and-nu_B}, respectively.
It remains to establish \eqref{item-Id-RFRF-RF-composition}. 
Following our conventions, the maps
$\id \xrightarrow{{\action}RF \circ \action} RFRF$
and
$\id \xrightarrow{RF{\action} \circ \action} RFRF$
denote the compositions 
\begin{align*}
\A \xrightarrow{\action} M \bartimes M^\barB 
\xrightarrow{\id \bartimes \beta}
M \bartimes M^\barB \bartimes \A
\xrightarrow{\id^2 \bartimes \action}
M \bartimes M^\barB \bartimes M \bartimes M^\barB \\
\A \xrightarrow{\action} M \bartimes M^\barB 
\xrightarrow{\action \bartimes \id}
\A \bartimes M \bartimes M^\barB 
\xrightarrow{\action \bartimes \id^2 }
M \bartimes M^\barB \bartimes M \bartimes M^\barB 
\end{align*}
which by
Prop.~\ref{prps-alpha-and-tensoring-f-with-identity}\eqref{item-f-and-beta-commute}
are equal to 
\begin{align*}
\A \xrightarrow{\beta^r} \A \bartimes \A
\xrightarrow{\action \bartimes \action}
M \bartimes M^\barB \bartimes M \bartimes M^\barB \\
\A \xrightarrow{\beta^l} \A \bartimes \A
\xrightarrow{\action \bartimes \action }
M \bartimes M^\barB \bartimes M \bartimes M^\barB. 
\end{align*}
Thus the map $\left({\action}RF - RF{\action}\right) \circ \action$
equals 
\begin{align*}
\A \xrightarrow{\beta^r - \beta^l} \A \bartimes \A
\xrightarrow{\action \bartimes \action}
M \bartimes M^\barB \bartimes M \bartimes M^\barB. 
\end{align*}
The desired assertion follows 
since $d\pi = \beta^r - \beta^l$. 
\end{proof}

We next prove that the convolutions of
\eqref{eqn-FR-FRFR-FR-Id-twisted-complex}-\eqref{eqn-Id-FL-FLFL-FL-twisted-complex}
are isomorphic in $D(\BbimB)$ and $D(\AbimA)$ to 
the squares of twists and co-twists
of $F$. It follows, in particular, that 
the isomorphism classes in $D(\BbimB)$ and $D(\AbimA)$ of
the convolutions of 
\eqref{eqn-FR-FRFR-FR-Id-twisted-complex}-\eqref{eqn-Id-FL-FLFL-FL-twisted-complex}
depend only on the isomorphism class in $D(\AbimB)$ of $F$. We 
can therefore think of them as canonical twisted complexes 
associated to the homotopy adjunctions $(F,R)$ and $(L,F)$. 

\begin{theorem}
\label{theorem-the-convolution-of-FR-FRFR-FR-Id-is-T^2-etc}
Let $\A$ and $\B$ be DG-categories and let $M \in \AmodbarB$. 
\begin{enumerate}
\item If $M$ is $\B$-perfect, then we have in $D(\BbimB)$
\begin{equation}
\label{eqn-FR-FRFR-FR-Id-twisted-complex-is-isomorphic-to-T^2}
\left\{
\begin{tikzcd}[column sep={3cm},row sep={1.5cm}] 
FR
\ar{r}{F{\action}R}
\ar[bend left=20,dashed]{rr}{\xi'_\B}
&
FRFR
\ar{r}{FR\trace - \trace FR}
&
FR
\ar{r}{\trace}
&
\underset{\degzero}{\id_\B}
\end{tikzcd}
\right\}
\simeq T^2
\end{equation}
where 
$T = \left\{ FR \xrightarrow{\trace} \underset{\degzero}{\id_\B} \right\}$ 
is the twist of $F$. We also have in $D(\AbimA)$
\begin{equation}
\label{eqn-Id-RF-RFRF-RF-twisted-complex-is-isomorphic-to-C^2}
\left\{
\begin{tikzcd}[column sep={3cm},row sep={1.5cm}] 
\underset{\degzero}{\id_\A}
\ar{r}{\action}
\ar[bend left=25,dashed]{rrr}{\nu_\B}
\ar[bend left=20,dashed]{rr}{-\upsilon_\B}
&
RF
\ar{r}{\action{RF} - RF\action}
\ar[bend right=20,dashed]{rr}{\xi_\B}
&
RFRF
\ar{r}{F\trace{R}}
&
RF 
\end{tikzcd}
\right\}
\simeq C^2
\end{equation}
where 
$C = \left\{ \underset{\degzero}{\id_\A} \xrightarrow{\action} RF \right\}$ 
is the co-twist of $F$. 
\item 
If $M$ is $\A$-perfect, then we have in $D(\AbimA)$
\begin{equation}
\label{eqn-LF-LFLF-LF-Id-twisted-complex-is-isomorphic-to-C'^2}
\left\{
\begin{tikzcd}[column sep={3cm},row sep={1.5cm}] 
LF
\ar{r}{L{\action}F}
\ar[bend left=20,dashed]{rr}{\xi'_\A}
&
LFLF
\ar{r}{LF\trace - \trace LF}
&
LF
\ar{r}{\trace}
&
\underset{\degzero}{\id_\A}
\end{tikzcd}
\right\}
\simeq 
C'^2
\end{equation}
where 
$C' = \left\{ LF \xrightarrow{\trace} \underset{\degzero}{\id_\A} \right\}$ 
is the dual co-twist of $F$. We also have in $\D(\BbimB)$
\begin{equation}
\label{eqn-Id-FL-FLFL-FL-twisted-complex-is-isomorphic-to-T'^2}
\left\{
\begin{tikzcd}[column sep={3cm},row sep={1.5cm}] 
\underset{\degzero}{\id_\B}
\ar{r}{\action}
\ar[bend left=25,dashed]{rrr}{\nu_\A}
\ar[bend left=20,dashed]{rr}{-\upsilon_\A}
&
FL
\ar{r}{{FL}\action - \action{FL}}
\ar[bend right=20,dashed]{rr}{\xi_\A}
&
FLFL
\ar{r}{F\trace{L}}
&
FL. 
\end{tikzcd}
\right\}
\simeq T'^2
\end{equation}
where 
$T' = \left\{ \underset{\degzero}{\id_\B} \xrightarrow{\action} FL \right\}$ 
is the dual twist of $F$.
\end{enumerate}
\end{theorem}
\begin{proof}
We only construct the isomorphism 
\eqref{eqn-FR-FRFR-FR-Id-twisted-complex-is-isomorphic-to-T^2}. 
The isomorphisms
 \eqref{eqn-Id-RF-RFRF-RF-twisted-complex-is-isomorphic-to-C^2}, 
\eqref{eqn-LF-LFLF-LF-Id-twisted-complex-is-isomorphic-to-C'^2}, 
and 
\eqref{eqn-Id-FL-FLFL-FL-twisted-complex-is-isomorphic-to-T'^2}
are constructed similarly. 
By  
Lemma \ref{lemma-tensoring-and-homming-for-twisted-complexes}
\eqref{item-tensoring-twisted-complexes}
the object $T^2$ is isomorphic in $D(\BbimB)$ to the convolution of 
\begin{equation*}
\begin{tikzcd}[ampersand replacement=\&,column sep={3cm},row sep={1.5cm}]  
M^\barB \bartimes M \bartimes M^\barB \bartimes M
\ar{r}{
\left(\begin{smallmatrix}
- \id^2 \bartimes \trace \; \\
 \; \trace \bartimes \id^2  
\end{smallmatrix}\right)
}
\&
\bigl(M^\barB \bartimes M \bartimes \B\bigr)
\;\oplus\;
\bigl( \B \bartimes M^\barB \bartimes M \bigr)
\ar{r}{
\left(\begin{smallmatrix}
\trace \bartimes \id \; & \; \id \bartimes \trace
\end{smallmatrix}\right)
}
\&
\underset{\degzero}{\B \bartimes  \B}. 
\end{tikzcd}
\end{equation*}
The twisted complex above is readily seen to be homotopy equivalent to 
the twisted complex
\begin{equation}
\label{eqn-T^2-as-a-twisted-complex}
\begin{tikzcd}[ampersand replacement=\&,column sep={3cm},row sep={1.5cm}]  
FRFR
\ar{r}{
\left(\begin{smallmatrix}
- \trace FR \; \\ \; FR \trace 
\end{smallmatrix}\right)
}
\&
FR
\;\oplus\;
FR
\ar{r}{
\left(\begin{smallmatrix}
\trace & \trace
\end{smallmatrix}\right)
}
\&
\underset{\degzero}{\id_\B}.
\end{tikzcd}
\end{equation}
Indeed, by the Rectangle Lemma and 
Prop.~\ref{prps-alpha-and-tensoring-f-with-identity}
the following is a homotopy equivalence between the two:
\begin{equation*}
\begin{tikzcd}[ampersand replacement=\&,column sep={3cm},row sep={1.5cm}]  
M^\barB \bartimes M \bartimes M^\barB \bartimes M
\ar{r}{
\left(\begin{smallmatrix}
- \id^2 \bartimes \trace  \\  \trace \bartimes \id^2  
\end{smallmatrix}\right)
}
\ar[dashed]{dr}[description]{
\left(\begin{smallmatrix}
0  \\  
- (\composition \bartimes \id) \circ \left((\kappa \circ \composition) \bartimes \id^2\right)  
\end{smallmatrix}\right)
}
\ar[equals]{d}
\&
\bigl(M^\barB \bartimes M \bartimes \B\bigr)
\;\oplus\;
\bigl( \B \bartimes M^\barB \bartimes M \bigr)
\ar{r}{
\left(\begin{smallmatrix}
\trace \bartimes \id \; & \; \id \bartimes \trace
\end{smallmatrix}\right)
}
\ar[near start]{d}{
\left(\begin{smallmatrix}
\alpha &  0 \\
0 & \alpha   
\end{smallmatrix}\right)
}
\ar[dashed]{dr}[description]{
\left(\begin{smallmatrix}
\alpha \circ (\trace \bartimes \id) \circ \theta \; &  0 
\end{smallmatrix}\right)
}
\&
\underset{\degzero}{\B \bartimes  \B}
\ar{d}{\alpha}
\\
M^\barB \bartimes M \bartimes M^\barB \bartimes M
\ar{r}{
\left(\begin{smallmatrix}
- \alpha \circ (\id^2 \bartimes \trace) \\ \alpha \circ (\trace \bartimes \id^2)  
\end{smallmatrix}\right)
}
\&
\bigl(M^\barB \bartimes M \bigr)
\;\oplus\;
\bigl(M^\barB \bartimes M \bigr)
\ar{r}{
\left(\begin{smallmatrix}
\trace \; & \; \trace
\end{smallmatrix}\right)
}
\&
\underset{\degzero}{\B}. 
\end{tikzcd}
\end{equation*}

Consider now the following closed degree zero map of twisted complexes
\begin{equation}
\label{eqn-homotopy-equivalence-from-T^2-to-FR-FRFR-FR-Id}
\begin{tikzcd}[ampersand replacement=\&,column sep={3cm},row sep={1cm}]  
\&
FRFR
\ar{r}{
\left(\begin{smallmatrix}
- \trace FR \\ FR \trace 
\end{smallmatrix}\right)
}
\ar[equal]{d}
\&
FR
\;\oplus\;
FR
\ar{r}{
\left(\begin{smallmatrix}
\trace \; & \; \trace
\end{smallmatrix}\right)
}
\ar{d}{
\left(\begin{smallmatrix}
\id \; & \; \id
\end{smallmatrix}\right)
}
\
\&
\id_\B
\ar[equal]{d}
\\
FR
\ar{r}{F{\action}R}
\ar[bend right=15,dashed]{rr}{\xi'_\B}
\&
FRFR
\ar{r}{FR\trace - \trace FR}
\&
FR
\ar{r}{\trace}
\&
\underset{\degzero}{\id_\B}.
\end{tikzcd}
\end{equation}

By the Rectangle Lemma (Lemma \ref{lemma-rectangle-lemma}) the total
complex of \eqref{eqn-homotopy-equivalence-from-T^2-to-FR-FRFR-FR-Id}
is isomorphic to the total complex of
\begin{tiny}
\begin{equation*}
\begin{tikzcd}[ampersand replacement=\&,column sep={2cm},row sep={1cm}]  
FR
\ar[bend left=15,dashed,']{rrr}{0\rightarrow 0\colon \xi'_\B}
\ar{r}{0\rightarrow0\colon F{\action}R}
\&
\left(
FRFR \xrightarrow{\id} \underset{\degzero}{FRFR}
\right)
\ar{rr}{\text{-}1\rightarrow\text{-}1\colon
- \left(\begin{smallmatrix}- \trace FR \\ FR \trace\end{smallmatrix}\right),}[']{0\rightarrow 0\colon FR\trace - \trace FR}
\& \&
\left(
FR
\oplus
FR
\xrightarrow{ 
-
\left(\begin{smallmatrix}
\id & \id
\end{smallmatrix}\right)
}
\underset{\degzero}{FR}
\right)
\ar{r}{\text{-}1\rightarrow\text{-}1\colon
- \left(\begin{smallmatrix}
\trace & \trace
\end{smallmatrix}\right),}[']{0\rightarrow 0\colon \trace}
\&
\underset{\degzero}
{
\left(
\id_\B
\xrightarrow{\id}
\underset{\degzero}{\id_\B}
\right)
}.
\end{tikzcd}
\end{equation*}
\end{tiny}
Here and below, we use the following convention for labelling the arrows which correspond to maps of twisted complexes. To specify such map we list all its non-zero components in the format $i \rightarrow j \colon \alpha$. This means that the component of the map which goes from the degree $i$ element of the source complex to the degree $j$ element of the target complex is $\alpha$. 

By the Extraction Lemma (Lemma \ref{lemma-extraction-lemma}) 
the complex above is homotopy equivalent to
\begin{equation*}
\begin{tikzcd}[ampersand replacement=\&,column sep={3cm},row sep={1cm}]  
FR
\ar{r}
\ar[',bend right=10,dashed]{rr}{
0 \rightarrow \text{-}1\colon
\left(\begin{smallmatrix} \text{-} \trace FR \circ F\action R \\ FR\trace \circ F\action R\end{smallmatrix}\right), \;
0 \rightarrow 0\colon \xi'_\B
}
\&
0
\ar{r}
\&
\underset{\degminusone}{
\bigl\{
FR
\oplus
FR
\xrightarrow{ 
-
\left(\begin{smallmatrix}
\id \; & \; \id
\end{smallmatrix}\right)
}
\underset{\degzero}{FR}
\bigr\}
}
\end{tikzcd}
\end{equation*}
and that can be rearranged as 
\begin{equation*}
\begin{tikzcd}[ampersand replacement=\&,column sep={3cm},row sep={1cm}]  
\bigl\{
FR
\xrightarrow{- \trace FR \circ F \action R}
\underset{\degzero}{FR}
\bigr\}
\ar{r}{ \text{-}1 \rightarrow \text{-}1\colon
FR\trace \circ F\action R}[']{0 \rightarrow 0\colon \id, \; \text{-}1 \rightarrow 0\colon \xi'_\B }
\&
\underset{\degminusone}{
\bigl\{
FR
\xrightarrow{- \id}
\underset{\degzero}{FR}
\bigr\}
}
\end{tikzcd}
\end{equation*}
which is null-homotopic by the Extraction Lemma. Here and below, for any DG-category $\C$, we say that two $\mathbb{Z}$-graded collections of elements of $\pretriag \C$ and maps  between them can be \em rearranged \rm one into another if they have the same totalisation, considered as a $\mathbb{Z}$-graded collection of elements of $\C$ and maps between them. 
Note that a $\mathbb{Z}$-graded collection of elements of $\pretriag \C$ and maps between them is a twisted complex if and only if its totalisation is a twisted complex. Thus if we start with an element of $\pretriag \pretriag \C$, anything we can rearrange it into is also an element of $\pretriag \pretriag \C$.  Note further that this implies that a map of two elements of 
$\pretriag \C$ is closed of degree $0$ if and only if its totalisation is a twisted complex. 
\end{proof}

\subsection{Derived category perspective}
\label{section-derived-category-perspective}

As described in \S\ref{section-zolotom-po-mramoru}
any twisted complex over a DG-category defines several Postnikov systems 
in the homotopy category which compute its convolution. 
In this section, we study in detail four
Postnikov systems in $D(\BbimB)$ obtained from the twisted
complex \eqref{eqn-FR-FRFR-FR-Id-twisted-complex}
of Theorem 
\ref{theorem-canonical-twisted-complex-associated-to-homotopy-adjunction}
by repeated application of 
Corollary \ref{cor-postnikov-systems-from-twisted-complex}. 
These turn out to depend only on the isomorphism class of $M$ 
in $D(\AbimB)$ and thus induce canonical functorial Postnikov systems
which exist for any adjoint pair $(f,r)$ of enhanced derived functors.
Similarly, each of the twisted complexes 
\eqref{eqn-Id-RF-RFRF-RF-twisted-complex}, 
\eqref{eqn-LF-LFLF-LF-Id-twisted-complex}
and \eqref{eqn-Id-FL-FLFL-FL-twisted-complex}
of Theorem 
\ref{theorem-canonical-twisted-complex-associated-to-homotopy-adjunction}
defines four canonical Postnikov systems in the derived category.
An interested reader would have no trouble working these out 
following our treatment of the complex \eqref{eqn-FR-FRFR-FR-Id-twisted-complex}
in this section. 

Let 
\begin{align}
\label{eqn-trace-exact-triangle}
FR \xrightarrow{\trace} \id \xrightarrow{p} T \xrightarrow{q} FR[1] \\
\label{eqn-action-exact-triangle}
C \xrightarrow{r} \id \xrightarrow{\action} RF \xrightarrow{s} C[1]. 
\end{align}
be any exact triangles in $D(\AbimB)$ which complete $FR
\xrightarrow{\trace} \id$ and $\id \xrightarrow{\action} RF$. 
The objects $T$ and $C$ are called the \em twist \rm and the \em
co-twist \rm of $F$, respectively. 

Since $F$ and $R$ are genuinely adjoint in $D(\AbimB)$ 
$$ R \xrightarrow{\action R} RFR \xrightarrow{R \trace} R $$
$$ F \xrightarrow{F \action} FRF \xrightarrow{\trace F} F $$
are retracts. In a triangulated category all retracts are split, 
and thus we have isomorphisms
\begin{align}
\label{eqn-splittings-of-FRFR-FRT-1} 
FR \oplus FRT[-1] \xrightarrow{(F\action{R} \;\; FRq)} FRFR \\
\label{eqn-splittings-of-FRFR-TFR-1} 
FR \oplus TFR[-1] \xrightarrow{(F\action{R} \;\; qFR)} FRFR \\ 
FRFR \xrightarrow{\left(\begin{smallmatrix}FR\trace \\  FsR\end{smallmatrix}\right)} FR \oplus FCR[1] \\
\label{eqn-splittings-of-FRFR-FCR1-trFR} 
FRFR \xrightarrow{\left(\begin{smallmatrix}\trace{FR} \\ FsR\end{smallmatrix}\right)} FR \oplus FCR[1].
\end{align}

Below, we adopt the following convention. Given objects 
$A$, $B$ and $C$ in a triangulated category we say that a triangle
$$ A \xrightarrow{\deg i} B \xrightarrow{\deg j} C \xrightarrow{\deg k} A $$
with $i+j+k = 1$ is exact if the following induced triangle is exact:
$$ A \xrightarrow{\deg 0} B[i] \xrightarrow{\deg 0} C[i+j]
\xrightarrow{\deg 0} A[1]. $$

\begin{theorem}
\label{theorem-canonical-postnikov-systems-for-FR-FRFR-FR-Id}
Let $M \in \AmodbarB$ be $\B$-perfect and let $(F,R)$ be 
the corresponding homotopy adjoint pair. For any exact
triangle \eqref{eqn-trace-exact-triangle} completing 
$FR \xrightarrow{\trace} \id$:

\begin{enumerate}
\item 
\label{item-the-canonical-postnikov-system-for-FR-FRFR-FR-Id}
The following is a Postnikov system in $D(\AbimB)$:

\begin{equation}
\label{eqn-canonical-twisted-complex-postnikov-system-right-to-left}
\begin{tikzcd}[column sep={1.732cm,between origins}, row sep={2cm,between origins}] 
FR
\ar{rr}{F{\action}R}[']{\simeq \left(
\begin{smallampmatrix}
\id \\ 0
\end{smallampmatrix}
\right)}
\ar[dotted]{ddrrrr}[description]{\left(
\begin{smallampmatrix}
\id \\ 0
\end{smallampmatrix}
\right)}
&&
\underset{\simeq FR \oplus FRT[-1]}{FRFR}
\ar{rr}{FR\trace - \trace FR}[']{\simeq (0 \;\; \text{-}q\circ \trace{T})}
\ar[dotted]{rrrd}[description]{(0 \;\; - \trace T)}
\ar[phantom]{dddr}[description, pos=0.6]{\star}
&&
FR
\ar{rr}{\trace}
\ar[phantom]{dd}[description, pos=0.6]{\star}
&&
\id_\B.
\ar{ld}{p}
\arrow[to=N24,phantom]{}[description, pos=0.275]{\star}
\\
&&
& |[alias=N24]| &
& 
T
\ar[dashed]{lu}[description]{q}
\ar{ld}{\left(
\begin{smallampmatrix}
0 \\ pT
\end{smallampmatrix}
\right)}
&
\\
&&
&&
FR[2] \oplus T^2
\ar[dashed]{lluu}[description]{\left(
\begin{smallampmatrix}
\id \amsamp 0 \\
0 \amsamp qT
\end{smallampmatrix}
\right)}
\ar{ld}{(0 \;\; \id)}
&&
\\
&&
&
T^2
\ar[dashed]{llluuu}{0}
&
&&
\end{tikzcd}
\end{equation}
Here the triangles denoted by $\star$ are exact, the remaining
triangles are commutative, the morphisms of $\deg > 0$ are drawn 
with dashed arrows, and the morphisms of $\deg < 0$ are drawn with
dotted arrows. 

\item 
\label{item-three-remaining-postnikov-systems-for-FR-FRFR-FR-Id}
For any exact triangles
\begin{align}
\label{eqn-FRTminus1-FR-X-exact-triangle}
FRT[-1] \xrightarrow{\text{-}q \circ \trace T} FR  \xrightarrow{t} X
\xrightarrow{u} FRT \\
\label{eqn-Y-id-T^2-exact-triangle}
Y \xrightarrow{v} \id \xrightarrow{p^2} T^2 \xrightarrow{w} Y[1]
\end{align}
in $D(\AbimB)$ completing $\text{-}q \circ \trace T$ and $p^2$, there 
exists an isomorphism $X \simeq Y$ in $D(\AbimB)$
such that the following are Postnikov systems:
\begin{equation}
\label{eqn-canonical-twisted-complex-postnikov-system-left-to-right}
\begin{tikzcd}[column sep={1.732cm,between origins}, row sep={2cm,between origins}] 
FR
\ar{rr}{F{\action}R}[']{\simeq \left(
\begin{smallampmatrix}
\id \\ 0
\end{smallampmatrix}
\right)}
&
\ar[phantom]{d}[description, pos=0.4]{\star}
&
\underset{\simeq FR \oplus FRT[-1]}{FRFR}
\ar{rr}{FR\trace - \trace FR}[',pos=0.4]{\simeq (0 \;\; \text{-}q\circ \trace{T})}
\ar{ld}[description]{(0 \;\; \id)}
&
&
FR
\ar{rr}{\trace}
\ar{lldd}[description]{t}
&
&
\id_\B
\ar{lllddd}{p^2}
\\
& 
FRT[-1]
\ar[dashed]{lu}{0}
\ar{rrru}[description]{\text{-}q \circ \trace{T}}
&
&
\ar[phantom]{ll}[description, pos=0.5]{\star}
&
&&
\\
&&
X \simeq Y
\ar{rrrruu}[description]{v}
\ar[dashed]{lu}{u}
&&
\ar[phantom]{lllu}[description, pos=0.2]{\star}
&&
\\
&&
&
T^2
\ar[dashed]{lu}{w}
&
&&
\end{tikzcd}
\end{equation}

\begin{equation}
\label{eqn-canonical-twisted-complex-postnikov-system-center-left}
\begin{tikzcd}[column sep={1.732cm,between origins}, row sep={2cm,between origins}] 
FR
\ar{rr}{F{\action}R}[']{\simeq \left(
\begin{smallampmatrix}
\id \\ 0
\end{smallampmatrix}
\right)}
\ar[dotted]{drrr}[description]{\left(
\begin{smallampmatrix}
\id \\ 0
\end{smallampmatrix}
\right)}
&
\ar[phantom]{ddr}[description,pos=0.42]{\star}
&
\underset{\simeq FR \oplus FRT[-1]}{FRFR}
\ar{rr}{FR\trace - \trace FR}[']{\simeq (0 \;\; \text{-}q\circ \trace{T})}
&
\ar[phantom]{d}[description,pos=0.4]{\star}
&
FR
\ar{rr}{\trace}
\ar{dl}[description]{\left(
\begin{smallampmatrix}
0 \\ t
\end{smallampmatrix}
\right)}
&&
\id_\B
\ar{lllddd}{p^2}
\\
&&
&
FR[1] \oplus X \simeq Y
\ar{dl}[description]{(0 \;\; \id)}
\ar[dashed]{lu}[description]{\left(
\begin{smallampmatrix}
\id \amsamp 0 \\
0 \amsamp v
\end{smallampmatrix}
\right)}
&
&&
\\
&&
X \simeq Y
\ar[dashed]{uull}{0}
\ar{rrrruu}[description]{v}
&&
\ar[phantom]{lu}[description, pos=0.4]{\star}
&&
\\
&&
&
T^2
\ar[dashed]{lu}{w}
&
&&
\end{tikzcd}
\end{equation}

\begin{equation}
\label{eqn-canonical-twisted-complex-postnikov-system-center-right}
\begin{tikzcd}[column sep={1.732cm,between origins}, row sep={2cm,between origins}] 
FR
\ar{rr}{F{\action}R}[']{\simeq\left(
\begin{smallampmatrix}
\id \\ 0
\end{smallampmatrix}
\right)} 
\ar[dotted]{ddrrrr}[description, pos=0.4]{\left(
\begin{smallampmatrix}
\id \\ 0
\end{smallampmatrix}
\right)}
&&
\underset{\simeq FR \oplus FRT[-1]}{FRFR}
\ar[phantom]{dddr}[description, pos=0.6]{\star}
\ar{rr}{FR\trace - \trace FR}[']{\simeq (0 \;\; \text{-}q\circ \trace{T})}
&
\ar[phantom]{d}[description,pos=0.4]{\star}
&
FR
\ar{rr}{\trace}
\ar{dl}[description]{\left(
\begin{smallampmatrix}
0 \\ t
\end{smallampmatrix}
\right)}
&
\ar[phantom]{ldd}[description,pos=0.42]{\star}
&
\id_\B
\ar{lldd}[description]{\left(
\begin{smallampmatrix}
0 \\ p^2
\end{smallampmatrix}
\right)}
\\
&&
&
FR[1] \oplus X \simeq Y
\ar{urrr}[description]{(0 \;\; u)}
\ar[dashed]{lu}[description]{\left(
\begin{smallampmatrix}
\id \amsamp 0 \\
0 \amsamp v
\end{smallampmatrix}
\right)}
&
&&
\\
&&
&
&
FR[2] \oplus T^2
\ar[dashed]{lu}[description,pos=0.6]{\left(
\begin{smallampmatrix}
\id \amsamp 0 \\
0 \amsamp w 
\end{smallampmatrix}
\right)}
\ar{ld}{(0 \;\; \id)}
&&
\\
&&
&
T^2
\ar[dashed]{llluuu}{0}
&
&&
\end{tikzcd}
\end{equation}

Here the triangles denoted by $\star$ are exact, the remaining
triangles are commutative, the morphisms of $\deg > 0$ are drawn 
with dashed arrows, and the morphisms of $\deg < 0$ are drawn with
dotted arrows.
\end{enumerate}
\end{theorem}

\bf NB: \rm There are also versions of this theorem for each of 
the three other splittings
\eqref{eqn-splittings-of-FRFR-TFR-1}-\eqref{eqn-splittings-of-FRFR-FCR1-trFR}
of $FRFR$. We leave them as an exercise to the reader. 

\begin{proof}
\eqref{item-the-canonical-postnikov-system-for-FR-FRFR-FR-Id}:\\
To show that 
\eqref{eqn-canonical-twisted-complex-postnikov-system-right-to-left} 
is a Postnikov system in $D(\AbimB)$ means to show that all its 
non-starred pieces commute and all its starred pieces are exact. 
It suffices to establish this when \eqref{eqn-trace-exact-triangle} is the canonical exact triangle
\begin{align}
\label{eqn-trace-exact-triangle-canonical}
FR \xrightarrow{\quad\trace\quad} \id \xrightarrow{\quad 0,0\colon \id\quad } 
\Bigl\{ FR \xrightarrow{\trace} \underset{\degzero}{\id} \Bigr\}
\xrightarrow{\quad \text{-}1,\text{-}1\colon \id \quad} FR[1]. 
\end{align}
This is because any other exact triangle \eqref{eqn-trace-exact-triangle}
is isomorphic to \eqref{eqn-trace-exact-triangle-canonical}, and 
this isomorphism is readily seen to identify the corresponding diagrams
\eqref{eqn-canonical-twisted-complex-postnikov-system-right-to-left}. 
Hence if one is a Postnikov system, so is the other. 

Thus, let \eqref{eqn-trace-exact-triangle} be the exact triangle 
\eqref{eqn-trace-exact-triangle-canonical}. 
Then there is the following natural lift of the whole diagram  
\eqref{eqn-canonical-twisted-complex-postnikov-system-right-to-left} 
into $\pretriag(\AmodbarB)$. We lift the objects of 
\eqref{eqn-canonical-twisted-complex-postnikov-system-right-to-left} 
as follows:
\begin{itemize}
\item We lift $FR$ and $\id$ to themselves, 
\item We lift $FR \oplus FRT[-1]$ to the twisted complex
$\underset{\degzero}{FR \oplus FRFR} \xrightarrow{
\left(\begin{smallmatrix} 0 \; & \; -FR\trace \end{smallmatrix}\right) 
}
FR$, 
\item We lift $T$ to the twisted complex 
$FR \xrightarrow{\trace} \underset{\degzero}{\id}$, 
\item As per the proof of 
Theorem \ref{theorem-the-convolution-of-FR-FRFR-FR-Id-is-T^2-etc}
we lift $T^2$ to the twisted complex 
\begin{equation}
\label{eqn-lifts-of-the-objects-of-the-right-to-left-system-T^2} 
\begin{tikzcd}[ampersand replacement=\&,column sep={3cm}]  
FRFR
\ar{r}{
\left(\begin{smallmatrix}
- \trace FR \\ FR \trace 
\end{smallmatrix}\right)
}
\&
FR
\;\oplus\;
FR
\ar{r}{
\left(\begin{smallmatrix}
\trace \; & \; \trace
\end{smallmatrix}\right)
}
\&
\underset{\degzero}{\id_\B},
\end{tikzcd}
\end{equation}
\item Similarly, we lift $FR[2]\oplus T^2$ to the twisted complex 
\begin{equation}
\label{eqn-lifts-of-the-objects-of-the-right-to-left-system-FR(2)-plus-T^2} 
\begin{tikzcd}[ampersand replacement=\&,column sep={3cm}]  
FR \oplus FRFR
\ar{r}{
\left(\begin{smallmatrix}
0 & -\trace FR \\
0 &  FR \trace 
\end{smallmatrix}\right)
}
\&
FR
\;\oplus\;
FR
\ar{r}{
\left(\begin{smallmatrix}
\trace \; & \; \trace
\end{smallmatrix}\right)
}
\&
\underset{\degzero}{\id_\B}.
\end{tikzcd}
\end{equation}
\end{itemize}

The maps $p$ and $q$ in \eqref{eqn-trace-exact-triangle-canonical} 
and the trace map are all defined by maps in $\pretriag(\AmodbarB)$. 
Since all the maps in  
\eqref{eqn-canonical-twisted-complex-postnikov-system-right-to-left} 
are written in terms of $p$, $q$ and the trace map, they all 
have natural lifts to $\pretriag(\AmodbarB)$. 
We thus have a natural lift of
\eqref{eqn-canonical-twisted-complex-postnikov-system-right-to-left} to 
$\pretriag(\AmodbarB)$ and it is then straightforward to verify
directly on the level of twisted complexes that 
\eqref{eqn-canonical-twisted-complex-postnikov-system-right-to-left}
is a Postnikov system. 

However, there is a more conceptual approach.  
Recall the following complex of the Theorem 
\ref{theorem-canonical-twisted-complex-associated-to-homotopy-adjunction}:
\begin{equation}
\label{eqn-FR-FRFR-FR-Id-twisted-complex-derived-cat-perspective}
\begin{tikzcd}[column sep={3cm}] 
FR
\ar{r}{F{\action}R}
\ar[bend left=20,dashed]{rr}{\xi'_\B}
&
FRFR
\ar{r}{FR\trace - \trace FR}
&
FR
\ar{r}{\trace}
&
\underset{\degzero}{\id_\B}.
\end{tikzcd}
\end{equation}
Repeated applications of
Corollary \ref{cor-postnikov-systems-from-twisted-complex}
to this twisted complex
produce the following diagram in $\pretriag(\AmodbarB)$
whose image $D(\AbimB)$ is a Postnikov system: 
\begin{equation}
\label{eqn-canonical-twisted-complex-postnikov-system-right-to-left-in-DG}
\begin{tikzcd}[column sep={1.732cm,between origins}, row
sep={2cm,between origins}] 
FR
\ar{rr}{F{\action}R}
\ar[dotted]{ddrrrr}[description, sloped]{0,\text{-}2\colon
F\action{R}, \; 0,\text{-}1\colon \xi'_\B}
&&
FRFR
\ar{rr}{FR\trace - \trace FR}
\ar[dotted]{rrrd}[description, sloped]{0,\text{-}1\colon
FR\trace - \trace FR}
\ar[phantom]{dddr}[description, pos=0.6]{\star}
&&
FR
\ar{rr}{\trace}
\ar[phantom]{dd}[description, pos=0.6]{\star}
&&
\id_\B.
\ar{d}[description]{\id}
\arrow[to=N24,phantom]{}[description, pos=0.275]{\star}
&
\\
&&
& |[alias=N24]| &
& 
\ar[dashed]{lu}[description]{\text{-}1,0\colon \id}
\Bigl( FR \ar{r}
\ar{d}[description]{\id}
&
\underset{\degzero}{\id_\B} \Bigr)
\ar{d}[description]{\id}
\\
&&
&&
\ar[dashed]{lluu}[description]{-2,0\colon \id}
\Bigl(FRFR
\ar{d}[description]{\id}
\ar{r}
&
FR
\ar{r}
\ar{d}[description]{\id}
&
\underset{\degzero}{\id_\B} \Bigr)
\ar{d}[description]{\id}
\\
&&
&
\ar[dashed]{llluuu}[description]{\text{-}3,0\colon \id}
\Bigl(FR
\ar{r}
&
FRFR
\ar{r}
&
FR
\ar{r}
&
\underset{\degzero}{\id_\B} \Bigr)
\end{tikzcd}
\end{equation}

Consider the following $\pretriag(\AmodbarB)$ homotopy equivalences 
between the objects of  
\eqref{eqn-canonical-twisted-complex-postnikov-system-right-to-left-in-DG}
and the aforementioned natural lift of 
\eqref{eqn-canonical-twisted-complex-postnikov-system-right-to-left}:
\begin{itemize}
\item The homotopy equivalence  
\begin{equation*}
\begin{tikzcd}[ampersand replacement=\&,column sep={1.25cm},row sep={0.75cm}]  
\underset{\degzero}{FRFR}
\ar{r}{
\left(\begin{smallmatrix}
0  \\ FR \trace 
\end{smallmatrix}\right)
}
\ar{d}{\id}
\&
FR \oplus FR
\\
\underset{\degzero}{FRFR}.
\&
\end{tikzcd}
\end{equation*}
It descends in $D(\AbimB)$ to the splitting isomorphism 
$FR \oplus FRT[-1] 
\xrightarrow{\eqref{eqn-splittings-of-FRFR-FRT-1}} 
FRFR$. 
\item 
The homotopy equivalence 
\eqref{eqn-homotopy-equivalence-from-T^2-to-FR-FRFR-FR-Id}
which was demonstrated in the proof of 
Theorem \ref{theorem-the-convolution-of-FR-FRFR-FR-Id-is-T^2-etc}
to descend in $D(\AbimB)$ to an isomorphism
\begin{equation*}
T^2 \xrightarrow{\sim} 
\left\{
\begin{tikzcd}[column sep={1cm}] 
FR
\ar{r}
&
FRFR
\ar{r}
&
FR
\ar{r}
&
\underset{\degzero}{\id_\B}
\end{tikzcd}
\right\}.
\end{equation*}
\item 
The homotopy equivalence 
\begin{equation}
\label{eqn-homotopy-equivalence-from-FR(2)+T^2-to-FRFR-FR-Id}
\begin{tikzcd}[ampersand replacement=\&,column sep={3cm},row
sep={1cm}]  
FR \oplus FRFR
\ar{r}{
\left(\begin{smallmatrix}
0 & -\trace FR \\
0 &  FR \trace 
\end{smallmatrix}\right)
}
\ar[dotted]{dr}[description,sloped]{
\left(\begin{smallmatrix}
\xi'_\B \; & \; 0  
\end{smallmatrix}\right)
}
\ar{d}{
\left(\begin{smallmatrix}
F\action{R} \; & \; \id  
\end{smallmatrix}\right)
}
\&
FR
\;\oplus\;
FR
\ar{r}{
\left(\begin{smallmatrix}
\trace \; & \; \trace
\end{smallmatrix}\right)
}
\ar{d}{
\left(\begin{smallmatrix}
\id \; & \; \id
\end{smallmatrix}\right)
}
\
\&
\id_\B
\ar[equal]{d}
\\
FRFR
\ar{r}{FR\trace - \trace FR}
\&
FR
\ar{r}{\trace}
\&
\underset{\degzero}{\id_\B}
\end{tikzcd}
\end{equation}
obtained by rearranging the terms of 
$\eqref{eqn-homotopy-equivalence-from-T^2-to-FR-FRFR-FR-Id}$. 
It descends in $D(\AbimB)$ to an isomorphism 
\begin{equation*} FR[2] \oplus T^2 \xrightarrow{\sim}
\left\{
\begin{tikzcd}[column sep={1cm}] 
FRFR
\ar{r}
&
FR
\ar{r}
&
\underset{\degzero}{\id_\B} 
\end{tikzcd}
\right\}.
\end{equation*}
\end{itemize}

It can be readily checked on the level of twisted complexes 
that these equivalences identify up to homotopy 
the natural lift of 
\eqref{eqn-canonical-twisted-complex-postnikov-system-right-to-left}
and
\eqref{eqn-canonical-twisted-complex-postnikov-system-right-to-left-in-DG}. 
Hence the corresponding isomorphisms identify
\eqref{eqn-canonical-twisted-complex-postnikov-system-right-to-left}
and the image of 
\eqref{eqn-canonical-twisted-complex-postnikov-system-right-to-left-in-DG}
in $D(\AbimB)$. 
Since the latter is a Postnikov system, so must be the former. 

\eqref{item-three-remaining-postnikov-systems-for-FR-FRFR-FR-Id}:

As in the proof of 
\eqref{item-the-canonical-postnikov-system-for-FR-FRFR-FR-Id}, 
it is enough to prove the desired assertion when 
$$ T = \Bigl\{ FR \xrightarrow{\quad\trace\quad} \underset{\degzero}{\id} \Bigr\} $$
and \eqref{eqn-trace-exact-triangle} is the canonical exact triangle 
\eqref{eqn-trace-exact-triangle-canonical}. For similar reasons, 
it is enough to assume that the exact triangles  
\eqref{eqn-FRTminus1-FR-X-exact-triangle}
and 
\eqref{eqn-Y-id-T^2-exact-triangle}
are 
\begin{scriptsize}
\begin{equation}
\label{eqn-FRTminus1-FR-X-exact-triangle-canonical}
\left\{
\underset{\degzero}{FRFR} \xrightarrow{-FR\trace} FR
\right\}
\xrightarrow{0,0\colon -\trace FR}
FR 
\xrightarrow{0,0\colon \left(\begin{smallmatrix}0 \\ \id\end{smallmatrix}\right)}
\left\{
{FRFR} \xrightarrow{\left(
\begin{smallmatrix}-\trace{FR} \\ FR\trace \end{smallmatrix}
\right)} 
\underset{\degzero}{FR \oplus FR}
\right\}
\underset{-1, 0\colon \id}{\xrightarrow{0,1\colon (0\;\;\id)}}
\end{equation}
\begin{equation}
\label{eqn-Y-Id-T^2-exact-triangle-canonical}
\left\{
{FRFR} \xrightarrow{ \left(
\begin{smallmatrix}-\trace{FR} \\ FR\trace\end{smallmatrix}
\right) }
\underset{\degzero}{FR \oplus FR}
\right\}
\xrightarrow{0,0\colon (\trace \;\; \trace)}
\id
\xrightarrow{0,0\colon \id}
\left\{
\begin{tikzcd}[column sep={0.9cm}]
{FRFR} 
\ar{rr}{\left(
\begin{smallmatrix}-\trace{FR} \\ FR\trace\end{smallmatrix}
\right) } 
&&
{FR \oplus FR}
\ar{r}{(\trace\;\;\trace)}
&
\underset{\degzero}{\id}
\end{tikzcd}
\right\}
\underset{\text{-}2,\text{-}1\colon \id}{\xrightarrow{
\text{-1,0}\colon 
\left(\begin{smallmatrix}
\id & 0 \\
0 & \id
\end{smallmatrix}\right)
}
}
\end{equation}
\end{scriptsize}
In particular, 
$$ X = Y = \left\{
{FRFR} \xrightarrow{ \left(
\begin{smallmatrix}-\trace{FR} \\ FR\trace\end{smallmatrix}
\right)  } \underset{\degzero}{FR \oplus FR}
\right\} $$
and we can take the requisite isomorphism $X \simeq Y$ to be the identity map. 

The rest of the proof proceeds analogously to that 
of \eqref{item-the-canonical-postnikov-system-for-FR-FRFR-FR-Id}. 
We lift the objects of 
\eqref{eqn-canonical-twisted-complex-postnikov-system-left-to-right}-\eqref{eqn-canonical-twisted-complex-postnikov-system-center-right}
to $\pretriag(\AmodbarB)$ as in the proof of 
\eqref{item-the-canonical-postnikov-system-for-FR-FRFR-FR-Id}, plus: 
\begin{itemize}
\item We lift $X = Y$ to the twisted complex 
\begin{align}
{FRFR} \xrightarrow{\left(
\begin{smallmatrix}-\trace{FR} \\ FR\trace\end{smallmatrix}
\right)  }
\underset{\degzero}{FR \oplus FR}. 
\end{align}
\item We lift $FR[1] \oplus X = Y$ to the twisted complex
\begin{equation}
\begin{tikzcd}[column sep={1.5cm}]
FR \oplus FRFR
\ar{rr}{
\left( 
\begin{smallampmatrix}
0 \amsamp -\trace FR \\
0 \amsamp FR\trace
\end{smallampmatrix}
\right)
}
&&
\underset{\degzero}{FR \oplus FR}. 
\end{tikzcd}
\end{equation}
\end{itemize}
Then the morphisms in \eqref{eqn-canonical-twisted-complex-postnikov-system-left-to-right}-\eqref{eqn-canonical-twisted-complex-postnikov-system-center-right} 
all have natural lifts to $\pretriag(\AmodbarB)$, since they are
all written in terms of the trace map and the maps in the exact triangles 
\eqref{eqn-trace-exact-triangle-canonical}, 
\eqref{eqn-FRTminus1-FR-X-exact-triangle-canonical}, 
and \eqref{eqn-Y-Id-T^2-exact-triangle-canonical}. And these were all
defined by maps in $\pretriag(\AmodbarB)$.

Thus we have natural lifts of  
\eqref{eqn-canonical-twisted-complex-postnikov-system-left-to-right}-\eqref{eqn-canonical-twisted-complex-postnikov-system-center-right}
to $\pretriag(\AmodbarB)$. It is then straightforward to verify  
on the level of twisted complexes that
\eqref{eqn-canonical-twisted-complex-postnikov-system-left-to-right}-\eqref{eqn-canonical-twisted-complex-postnikov-system-center-right}
are Postnikov systems. Alternatively, these lifts can be identified 
up to homotopy with the three remaining diagrams induced, in addition to  
\eqref{eqn-canonical-twisted-complex-postnikov-system-right-to-left-in-DG},
by the canonical twisted complex 
\eqref{eqn-FR-FRFR-FR-Id-twisted-complex-derived-cat-perspective}.
The identifying homotopy equivalences are those used in the proof of
\eqref{item-the-canonical-postnikov-system-for-FR-FRFR-FR-Id} plus:
\begin{itemize}
\item The following restriction of
\eqref{eqn-homotopy-equivalence-from-T^2-to-FR-FRFR-FR-Id} to 
corresponding subcomplexes:
\begin{equation}
\begin{tikzcd}[column sep={1.5cm}]
\label{eqn-X-Y-to-FR-FRFR-FR-homotopy-equivalence}
&
FRFR
\ar{rr}{\left(
\begin{smallmatrix}-\trace{FR} \\ FR\trace\end{smallmatrix}
\right) }
\ar{d}{-\id}
&&
\underset{\degzero}{FR \oplus FR}
\ar{d}{
(\id \; \; \id)
}
\\
FR 
\ar{r}{\text{-} F\action{R}}
\ar[bend right=15,dashed]{rrr}{-\xi'_\B}
&
FRFR
\ar{rr}{\text{-} (FR\trace - \trace{FR})}
&&
\underset{\degzero}{FR}.
\end{tikzcd}
\end{equation}
It descends in $D(\AbimB)$ to an isomorphism 
\begin{equation*}
(X = Y) \xrightarrow{\sim} 
\left\{
\begin{tikzcd}[column sep={2.5cm}] 
FR
\ar{r}{- F{\action}R}
\ar[bend right=15,dashed]{rr}{-\xi'_\B}
&
FRFR
\ar{r}{-(FR\trace - {\trace}FR)}
&
\underset{\degzero}{FR}
\end{tikzcd}
\right\}.
\end{equation*}
\item
The following restriction of
\eqref{eqn-homotopy-equivalence-from-FR(2)+T^2-to-FRFR-FR-Id}
to corresponding subcomplexes:
\begin{equation}
\begin{tikzcd}[column sep={2.5cm}]
FR \oplus FRFR
\ar[dashed]{drr}[description, pos = 0.4]{
\left( 
\begin{smallampmatrix}
- \xi'_\B \\
0 
\end{smallampmatrix}
\right)
}
\ar{rr}{
\left( 
\begin{smallampmatrix}
0 \amsamp -\trace FR \\
0 \amsamp FR\trace
\end{smallampmatrix}
\right)
}
\ar{d}{
\left(
\begin{smallmatrix}
- F\action{R} \; \amsamp \; - \id 
\end{smallmatrix}
\right)
}
&&
\underset{\degzero}{FR \oplus FR}
\ar{d}{
\left(\begin{smallmatrix}
\id \; \amsamp \; \id
\end{smallmatrix}\right)
}
\\
FRFR
\ar{rr}{-(FR\trace - \trace{FR})}
&&
\underset{\degzero}{FR} 
\end{tikzcd}
\end{equation}
obtained by rearranging the terms of 
\eqref{eqn-X-Y-to-FR-FRFR-FR-homotopy-equivalence}. It descends in
$D(\AbimB)$ to an isomorphism 
\begin{equation*}
FR[1] \oplus X \simeq Y \xrightarrow{\sim} 
\left\{
\begin{tikzcd}[column sep={2.5cm}] 
FRFR
\ar{r}{FR\trace - {\trace}FR}
&
\underset{\degzero}{FR}
\end{tikzcd}
\right\}.
\end{equation*}

\end{itemize}
It follows that \eqref{eqn-canonical-twisted-complex-postnikov-system-left-to-right}-\eqref{eqn-canonical-twisted-complex-postnikov-system-center-right}
are Postnikov systems as desired. 
\end{proof}

\bibliography{references}
\bibliographystyle{amsalpha}
\end{document}